\documentclass[10pt, a4paper, onecolumn]{IEEEtran}

\usepackage{graphicx}
\usepackage{amsmath}
\usepackage{amsfonts}
\usepackage{amssymb}
\usepackage{color}
\usepackage{rotating}
\usepackage{caption}
\usepackage{subcaption}

\newtheorem{lemma}{Lemma}
\newtheorem{theorem}{Theorem}
\newtheorem{corollary}{Corollary}
\newtheorem{definition}{Definition}

\newtheorem{proposition}{Proposition}
\newtheorem{problem}{Problem}

\title{\Huge  An approximate solution to the decentralized two-controller infinite-horizon scalar {LQG} problem: Part II- slow dynamics}

\author{Se Yong Park (separk@eecs.berkeley.edu), Anant Sahai (sahai@eecs.berkeley.edu)
\thanks{A part of the results in this paper was presented in Conference on Decision and Control, 2012. The authors are with the department of Electrical Engineering and Computer Sciences at the University of California at Berkeley.}
}

\begin{document}
\maketitle
\begin{abstract}
Continuing the first part of the paper, we consider scalar decentralized average-cost infinite-horizon LQG problems with two controllers. This paper focuses on the slow dynamics case when the eigenvalue of the system is small and prove that the single-controller optimal strategies ---linear strategies--- are constant ratio optimal among all distributed control strategies.
\end{abstract}

\section{Introduction}

In the first part of this paper~\cite{Park_Approximation_Journal_Parti}, we consider the simplest decentralized LQG (linear quadratic Gaussian) problem, the scalar infinite-horizon LQG problem with two controllers. In particular, we focused on the fast dynamics case when the eigenvalue of the system is large. The most interesting fact in this case is that a nonlinear control strategy can infinitely outperform any linear strategy especially when the two controllers are asymmetric. When the first controller has a better observation with high control cost and the second controller has a worse observation with small control cost, there is a huge incentive for the first controller to communicate its observation to the second controller. Moreover, this communication is implicitly through the plant and for such implicit communication, nonlinear strategies are more efficient than linear strategies. The Signal-to-Noise Ratio (SNR) for this implicit communication is upper bounded by the eigenvalue of the system. Thus, as the eigenvalue of the system goes to infinity, the performance gap between nonlinear and linear strategy can unboundedly diverge.

In this paper, we focus on slow dynamics where the eigenvalue of the system is bounded by a constant. The SNR for implicit communication in this case is bounded and the performance gap between the best nonlinear and linear strategies is bounded by a constant. In the scalar system considered in this paper, the system is observable and controllable by both controllers. It turns out that control by a single controller is good enough to achieve a constant-ratio of the optimal cost. 

The rest of the paper consists as follows: In Section~\ref{sec:statement}, we will state the problem and main results. In Section~\ref{sec:central},  we will revisit the centralized control results and intuitively understand them. In Section~\ref{sec:lower}, we will derive a fundamental lower bound on the control performance, and prove that the centralized control performance and the derived lower bound are within a constant ratio.

\section{Problem Statement and Main Result}
\label{sec:statement}

Throughout this paper, we will consider the same problem considered in \cite{Park_Approximation_Journal_Parti}, the scalar infinite-horizon decentralized LQG problems with two controllers. However, while the focus of \cite{Park_Approximation_Journal_Parti} was the fast-dynamics case (when $|a| \geq 2.5$), the focus of this paper is the slow-dynamics case (when $|a| < 2.5$).

\begin{problem}[scalar infinite-horizon decentralized LQG problems with two controllers]
Consider the system dynamics given as
\begin{align}
& x[n+1]=ax[n]+u_1[n]+u_2[n]+w[n]  \\
& y_1[n]=x[n]+ v_1[n]  \\
& y_2[n]=x[n]+ v_2[n]
\end{align}
where $x[0] \sim \mathcal{N}(0,\sigma_0^2)$, $w[n] \sim \mathcal{N}(0,1)$, $v_1[n] \sim \mathcal{N}(0,\sigma_{v1}^2)$, $v_2[n] \sim \mathcal{N}(0,\sigma_{v2}^2)$ are independent Gaussian random variables. The control inputs, $u_1[n]$ and $u_2[n]$, must be causal functions of $y_1[n]$ and $y_2[n]$ respectively, i.e. $u_1[n]=f_{1,n}(y_1[0],\cdots,y_1[n])$ and $u_2[n]=f_{2,n}(y_2[0],\cdots,y_2[n])$. 

For $q, r_1, r_2 \geq 0$, the control objective is to minimize a long-term average quadratic cost:
\begin{align}
\limsup_{N \rightarrow \infty} \frac{1}{N}
\sum_{0 \leq n < N} q \mathbb{E}[x^2[n]] + r_1 \mathbb{E}[u_1^2[n]] + r_2 \mathbb{E}[u_2^2[n]]. \label{eqn:part11}
\end{align}
\label{prob:a}
\end{problem}
Even though we normalized the problem parameters (the variance of $w[n]$, the gains for $u_1[n], u_2[n], y_1[n], y_2[n]$), this problem includes all scalar two-controller decentralized LQG problems by a proper scaling. We refer \cite{Park_Approximation_Journal_Parti} for the justification.


In \cite{Park_Approximation_Journal_Parti}, we saw that in fast-dynamics cases, implicit communication communication between two controllers is crucial to achieve the optimal performance within a constant ratio. Moreover, essentially memoryless controllers, which only exploit the information at the current time step, were constant-ratio optimal.

Therefore, a natural question for slow-dynamics cases is that whether the same type of controllers are enough to achieve constant ratio optimality. In other words, is implicit communication crucial in performance? Are memoryless controllers can achieve constant ratio optimality? In this paper, we will see that the answers for both questions are negative.

To understand why the answer for the first question is negative, let's revisit fast-dynamics cases. Even though the mathematical definition of implicit communication is still unclear, we can roughly measure the SNR (signal-to-noise ratio) of implicit communication. The blurry controller (the controller with higher observation noise) utilize the transmitted signal from the other controller as soon as the transmitted signal's power exceed its observation level. Therefore, the maximum SNR for implicit communication cannot exceed $a^2$, which is the ratio at which the system dynamics amplify the signals in one time step. From this, we can conjecture that for slow dynamics cases $(|a| \leq 2.5)$, the SNR is bounded and implicit communication may not be crucial for constant-ratio optimality.

However, justification is not that simple since the time-horizon is infinite. In other words, even though we could justify that the SNR at each time step is bounded, accumulation of such information may result in unbounded gain. Furthermore, a precise definition of implicit communication and the corresponding SNR requires further study.

For the second question, we will see that all observations from the past have to be utilized to achieve constant-ratio optimality. For this, Kalman filtering must be used.


In other words, we will prove that in the slow-dynamics case, single-controller optimal strategies --- Kalman filtering linear strategies --- are approximately optimal within a constant ratio. For this, let's first define the single-controller strategies which involve only one parameter $k$.

\begin{definition}[Single Controller Optimal Strategy $L_{lin,kal}$]
$L_{lin,kal}$ is the set of all controllers which can be written in either one of two following forms for some $k \in \mathbb{R}$\\
(i) $u_1[n]=-k\mathbb{E}[x[n]|y_1^n,u_1^{n-1}]$, $u_2[n]=0$\\
(ii) $u_1[n]=0$, $u_2[n]=-k\mathbb{E}[x[n]|y_2^n,u_2^{n-1}]$
\end{definition}
Here, we can notice that since the system is linear and underlying random variables are Gaussian, the conditional expectations are linear in the observations~\cite{Bertsekas}.

Now, we can state the main theorem of this paper, which states that when $|a| \leq 2.5$ the optimization only over $L_{lin,kal}$ is enough to achieve approximate optimality within a constant ratio among all possible strategies.
\begin{theorem}
Consider the decentralized LQG problem shown in Problem~\ref{prob:a}. Let $L$ be the set of all measurable causal strategies. Then, there exists a constant $c \leq 2 \cdot 10^6$ such that for all $|a| \leq 2.5$, $q$, $r_1$, $r_2$, $\sigma_0$, $\sigma_{v1}$ and $\sigma_{v2}$,
\begin{align}
\frac{
\underset{u_1,u_2 \in L_{lin,kal}}{\inf}
\underset{N \rightarrow \infty}{\limsup}
\frac{1}{N}
\underset{0 \leq n < N}{\sum}
\mathbb{E}[qx^2[n]+r_1 u_1^2[n]+r_2 u_2^2[n]]
}
{
\underset{u_1,u_2 \in L}{\inf}
\underset{N \rightarrow \infty}{\limsup}
\frac{1}{N}
\underset{0 \leq n < N}{\sum}
\mathbb{E}[qx^2[n]+r_1 u_1^2[n]+r_2 u_2^2[n]]
}
\leq c \nonumber
\end{align}
\label{thm:2}
\end{theorem}
\begin{proof}
See Section~\ref{subsec:mainthm} for the proof.
\end{proof}

The basic proof strategy is following. Rather than directly considering the average cost problem of Problem~\ref{prob:a}, we consider the power-distortion tradeoff problem of Problem~\ref{prob:c}. Then, since we have an explicit constraints on the controller power, we can divide the tradeoff curve to multiple regions based on the control power. For these finite number of regions, we derive different upper and lower bounds on the performance. By comparing them, we characterize the tradeoff curve within a constant ratio.Then, we finally convert the constant ratio characterization of the tradeoff curve into the constant ratio result on the average cost.

\begin{problem}[Decentralized LQG problem with average power constraints]
Consider the same dynamics as Problem~\ref{prob:a}. But, now the control objective is minimizing the state distortion for given input power constraints $P_1, P_2 \in \mathbb{R}^+$. We will say the power-distortion tradeoff, $D(P_1, P_2)$ is achievable if and only if there exist causal control strategies $u_1[n], u_2[n]$ such that
\begin{align}
&\limsup_{N \rightarrow \infty} \frac{1}{N} \sum^N_{n=1} \mathbb{E}[x^2[n]] \leq D(P_1, P_2), \\
&\limsup_{N \rightarrow \infty} \frac{1}{N} \sum^N_{n=1} \mathbb{E}[u_1^2[n]] \leq P_1, \\
&\limsup_{N \rightarrow \infty} \frac{1}{N} \sum^N_{n=1} \mathbb{E}[u_2^2[n]] \leq P_2.
\end{align}
\label{prob:c}
\end{problem}

\section{Qualitative Understanding of Centralized LQG Problems}
\label{sec:central}
Before we present the technical details, we first explain the insight behind the results.
Theorem~\ref{thm:2} states that control by a single controller is enough to achieve an approximately optimal performance.
The optimal control by a single controller is a well-studied LQG control problem. The optimal average cost, the weighted sum of the input power and the state distortion, is the solution of a Riccati equation.

However, even though Riccati equations gives exact optimal costs for centralized control problems, their quantitative results are hard to interpret. Therefore, in this section, we will approximate the optimal costs to simple functions, so that we can gain intuitive and qualitative understanding about the centralized control problems. Furthermore, we will take take a distortion-power-tradeoff perspective to the problems rather than a minimum-cost point-of-view.


Let's first formally state the scalar centralized LQG problems.
\begin{problem}[Centralized LQG with average power constraints]
Consider the following dynamic system with a single controller.
\begin{align}
&x[n+1]=ax[n]+u[n]+w[n] \label{eqn:centralized}\\
&y[n]=x[n]+v[n]
\end{align}
where $x[0] \sim \mathcal{N}(0,\sigma_0^2)$, $w[n] \sim \mathcal{N}(0,1)$, $v[n] \sim \mathcal{N}(0,\sigma_v^2)$ are independent Gaussian random variables. The control input $u[n]$ must be a causal function of $y[n]$, i.e. $u[n]=f_n(y_1[0], \cdots, y_1[n])$. 

The control objective is minimizing the state distortion for a given input power constraint $P \in \mathbb{R}^+$.  We say the power-distortion tradeoff $D_{\sigma_v}(P)$ is achievable if and only if there exists a causal control strategy $u[n]$ such that
\begin{align}
&\limsup_{N \rightarrow \infty} \frac{1}{N} \sum^N_{n=1} \mathbb{E}[x^2[n]] \leq D_{\sigma_v}(P),\\
&\limsup_{N \rightarrow \infty} \frac{1}{N} \sum^N_{n=1} \mathbb{E}[u^2[n]] \leq P.
\end{align}
\label{prob:b}
\end{problem}

\begin{definition}[Optimal Linear Strategy $L_{lin,cen}$ for Centralized LQG problems] 
Consider the centralized LQG problem of Problem~\ref{prob:b}. Let $L_{lin,cen}$ be the set of all controllers which can be written in the following form. For some $k \in \mathbb{R}$, $u[n]= -k \mathbb{E}[x[n]|y^n, u^n]$.
\label{def:cen}
\end{definition}

\begin{lemma}
Consider the centralized LQG problem of Problem~\ref{prob:b}. Define $\Sigma_E$ as
\begin{align}
\Sigma_E := \frac{(a^2-1)\sigma_v^2 - 1 + \sqrt{((a^2-1)\sigma_v^2-1)^2 + 4a^2 \sigma_v^2}}{2a^2}.
\label{eqn:kalmanperf}
\end{align}
Then, for all $k$ such that $|a-k| < 1$, the linear strategy of Definition~\ref{def:cen} can achieve the following Power-distortion tradeoff:
\begin{align}
(D_{\sigma_v}(P),P) = ( \frac{(2ak-k^2)\Sigma_E + 1}{1-(a-k)^2}, k^2 ( \frac{(2ak-k^2)\Sigma_E + 1}{1-(a-k)^2} -\Sigma_E)). \label{eqn:powerdis}
\end{align}
Furthermore, this power-distortion tradeoff is optimal in the sense that for a given $P$, there is no control strategy which can achieve an expected squared state smaller than $D_{\sigma_v}(P)$.
\label{lem:aless21}
\end{lemma}
\begin{proof}
Let $\hat{x}[n]:=\mathbb{E}[x[n]|y^n, u^{n-1}]$. Since $u[n]=-k\hat{x}[n]$,
\begin{align}
x[n+1]&=ax[n]-k\hat{x}[n]+w[n] \\
&=a(x[n]-\hat{x}[n])+(a-k)\hat{x}[n]+w[n]. \label{eqn:central1}
\end{align}
Define $\Sigma_{X,n}:= \mathbb{E}[x^2[n]]$, $\Sigma_{\hat{X},n}:=\mathbb{E}[\hat{x}^2[n]]$, $\Sigma_{E,n}:=\mathbb{E}[(x[n]-\hat{x}[n])^2]$. Then, we have
\begin{align}
\Sigma_{X,n}&= \mathbb{E}[(x[n]-\hat{x}[n]+\hat{x}[n])^2]\\
&=\mathbb{E}[(x[n]-\hat{x}[n])^2] + \mathbb{E}[\hat{x}^2[n]]\\
&=\Sigma_{E,n} + \Sigma_{\hat{X},n} \label{eqn:central2}
\end{align}
where the second equality comes from the orthogonality of $x[n]-\hat{x}[n]$ and $\hat{x}[n]$.
Likewise, by \eqref{eqn:central1} we also have
\begin{align}
\Sigma_{X,n+1} &= a^2 \Sigma_{E,n} + (a-k)^2 \Sigma_{\hat{X},n}+1 \\
&= a^2 \Sigma_{E,n} + (a-k)^2 (\Sigma_{X,n} - \Sigma_{E,n}) + 1 \label{eqn:aless21}
\end{align}
where the last inequality comes from \eqref{eqn:central2}.

Moreover, it is well-known that Kalman filtering performance converges to a steady state.
In other words, by \cite{Bertsekas} we have
\begin{align}
\Sigma_E:=\lim_{n \rightarrow \infty}\Sigma_{E,n}&= \frac{(a^2-1)\sigma_v^2 - 1 + \sqrt{((a^2-1)\sigma_v^2-1)^2 + 4a^2 \sigma_v^2}}{2a^2}
\end{align}

Thus, by \eqref{eqn:aless21}, $\Sigma_{X,n}$ converges as long as $|a - k | < 1$. Let $\lim_{n \rightarrow \infty}\Sigma_{X,n}=\Sigma_X$. Then, by \eqref{eqn:aless21} we have
\begin{align}
\Sigma_X &= \frac{ (a^2 - (a-k)^2)\Sigma_E + 1 }{1-(a-k)^2} \\
&=\frac{(2ak-k^2)\Sigma_E + 1}{1-(a-k)^2}.
\end{align}
Since $u[n]=-k \hat{x}[n]$, using \eqref{eqn:central2} the input power converges as follows.
\begin{align}
\lim_{n \rightarrow \infty}\mathbb{E}[u^2[n]]&= k^2(\Sigma_X -\Sigma_E) \\
&=k^2 ( \frac{(2ak-k^2)\Sigma_E + 1}{1-(a-k)^2} -\Sigma_E).
\end{align}

This finishes the achievability proof of the tradeoff. The tightness of the tradeoff and the optimality of centralized linear controllers are well-known in the community, and we refer to \cite{Bertsekas} for a rigorous proof based on dynamic programming.
%
%
%
%
%
\end{proof}

As mentioned in the proof, $\Sigma_E$ represents the Kalman filtering performance (mean square estimation error). 

In the following discussions, we will qualitatively understand the tradeoff between the state distortion and control power by dividing into cases based on the eigenvalue of the system.

\subsection{When $|a|=1$}
\label{subsec:a=1}

\begin{figure*}[htbp]
\begin{center}
\includegraphics[width=0.4\textwidth]{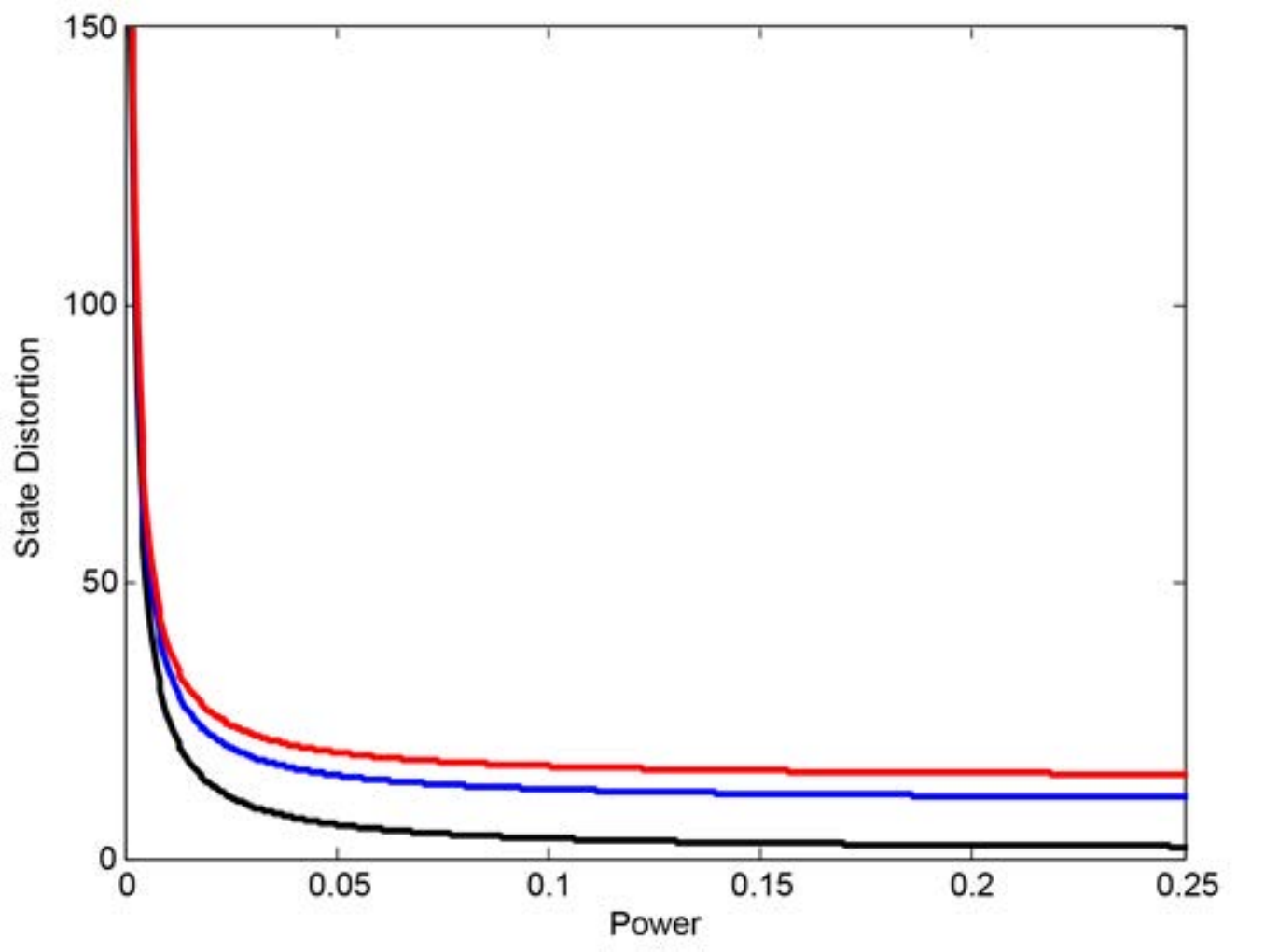}
\caption{The Optimal State Distortion-Input Power Tradeoff: When $a =1$ with different values of $\sigma_v^2$ ( $\sigma_v^2 = 1$ (Black line), $\sigma_v^2 = 100$ (Blue line), $\sigma_v^2 = 200$ (Red line))}
\label{fig:simulaeq11}
\end{center}
\end{figure*}

\begin{figure*}[htbp]
\begin{center}
\includegraphics[width=0.4\textwidth]{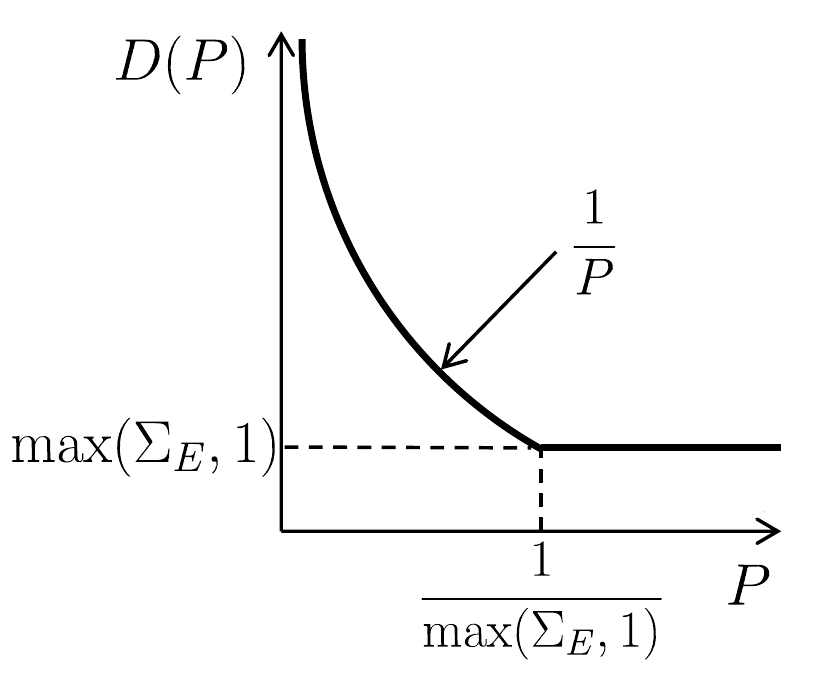}
\caption{Conceptual Plot of State Distortion-Input Power Tradeoff: When $|a| =1$}
\label{fig:aeq11}
\end{center}
\end{figure*}
First, let's consider the case when the magnitude of the eigenvalue is $1$, i.e. $|a|=1$. 

Since the Kalman filtering performance $\Sigma_E$ is the minimum squared error for estimating the states, we can see $D_{\sigma_v}(P) \geq a^2 \Sigma_E +1$ for all $P$. For notational convenience, let's approximate $a^2 \Sigma_E +1$ by $\max(\Sigma_E, 1)$.

To achieve $D_{\sigma_v}(P) \approx \max(\Sigma_E, 1)$, the control power $P$ has to be large enough. As we can see in Figure~\ref{fig:simulaeq11}, the state distortion $D_{\sigma_v}(P)$ inversely proportionally increases as the control power $P$ decreases. 

Therefore, the power-state distortion tradeoff $D_{\sigma_v}(P)$ can be conceptualized as Figure~\ref{fig:aeq11}. When the power $P$ is smaller than $\frac{1}{\max(\Sigma_E, 1)}$, the state distortion behaves like $\frac{1}{P}$. When the power becomes greater than $\frac{1}{\max(\Sigma_E, 1)}$, the state distortion saturates at $\max(\Sigma_E, 1)$. 

Let's write $(a_1, \cdots, a_n) \geq (b_1, \cdots, b_n)$ if and only if $a_1 \geq b_1$, $\cdots$, $a_n \geq b_n$. Then, the power-distortion tradeoff for the centralized LQG problem is characterized as follows. Then, Corollary~\ref{cor:3}  shows the formal statement of this tradeoff.


\begin{corollary}
Consider the centralized LQG problem shown in Problem~\ref{prob:b}. When $|a|=1$, the achievable power-distortion tradeoff $(D_{\sigma_v}(P), P)$ by the strategies of Definition~\ref{def:cen} is upper bounded as follows:
\begin{align}
(D_{\sigma_v}(P),P) \leq (\frac{2}{t},t) \mbox{ for all }  0 < t \leq \frac{1}{\max(1,\Sigma_E)} \label{eqn:tradeoff1}
\end{align}
where the definition of $\Sigma_E$ is given as \eqref{eqn:kalmanperf}.

Especially, when $\sigma_v \geq 16$, we have
\begin{align}
(D_{\sigma_v}(P),P) \leq (\frac{2}{t},t) \mbox{ for all }  0 < t \leq \frac{1}{1.0005 \sigma_v}. \label{eqn:cor:31}
\end{align}
When $\sigma_v \leq 16$, we have
\begin{align}
(D_{\sigma_v}(P),P) \leq (\frac{2}{t},t) \mbox{ for all }  0 < t \leq \frac{1}{15.008}. \label{eqn:cor:32}
\end{align}
\label{cor:3}
\end{corollary}
\begin{proof}
See Appendix~\ref{app:cor1} for the proof.
\end{proof}

As we can see from \eqref{eqn:tradeoff1}, for the power $0 < P \leq \frac{1}{\max(1, \Sigma_E)}$, the tradeoff is inversely proportional. When the power becomes $P=\frac{1}{\max(1, \Sigma_e)}$, the state distortion saturates at the Kalman filtering performance.

In fact, careful inspection of Figure~\ref{fig:simulaeq11} shows that the transition between the interval $P \in [0, \frac{1}{\max(\Sigma_E, 1)}]$ and $P \in [ \frac{1}{\max(\Sigma_E, 1)}, \infty]$ is much smoother than the one suggested in the conceptual plot of Figure~\ref{fig:aeq11}. Therefore, a better approximation of the tradeoff can be $D_{\sigma_v}(P) \approx \frac{1}{P}+\max(\Sigma_E, 1)$ rather than $D_{\sigma_v}(P) \approx \max(\frac{1}{P}, \Sigma_E, 1)$ suggested in Figure~\ref{fig:aeq11}. In fact, since $\max(\frac{1}{P}, \Sigma_E, 1) \leq \frac{1}{P} + \max(\Sigma_E, 1) \leq 2 \max(\frac{1}{P}, \Sigma_E, 1)$, the two approximations are within a constant ratio. Thus, both approximations are enough to prove constant ratio optimality. In this paper, we choose the approximation shown in Figure~\ref{fig:aeq11}, since it is more discrete and thereby easier to compare with the lower bound in Section~\ref{sec:subaeq1} by dividing cases.

Furthermore, we can prove that the optimal tradeoff $(D_{\sigma_v}(P), P)$ can be upper and lower bounded by the approximation of Figure~\ref{fig:aeq11} within a constant ratio. Consider the case~\footnote{Remind that when $|a|=1$, $\Sigma_E \approx \sigma_v$.} when $\sigma_v \geq 16$, then \eqref{eqn:cor:31} of Lemma~\ref{cor:3} gives an achievable upper bound on the tradeoff, $D_{\sigma_v}(P) \leq \frac{2}{P}$ for all $0 < P \leq \frac{1}{1.0005 \sigma_v} \leq \frac{1}{16}$. Corollary~\ref{cor:4} of Section~\ref{sec:subaeq1} gives a lower bound on the tradeoff. By putting the second controller's noise $\sigma_{v2}=\infty$ and considering the first controller as the centralized controller, (b) of Corollary~\ref{cor:4} gives that $D_{\sigma_v}(P) \geq \frac{0.02417}{P}+1$ for all $P \leq \frac{1}{64}$. Therefore, we can notice that the upper and lower bound matches within a constant ratio. Moreover, (d) of Corollary~\ref{cor:4} gives that $D_{\sigma_v}(P) \geq \max(\frac{\sqrt{2}}{2} \sigma_{v1}, 1)$ for all $P$, which justifies the flat part of Figure~\ref{fig:aeq11}. Therefore, increasing input power $P$ more than $\frac{1}{1.0005 \sigma_v}$ will not be greatly helpful, and we can use an achievable upper bound $D_{\sigma_{v}}(P) = {2.001 \sigma_v}$ for all $P \geq \frac{1}{1.0005 \sigma_v}$ to prove a constant ratio optimality. This constant ratio characterization of the tradeoff curve can be easily converted to a constant ratio optimality of average cost problems by applying \cite[Lemma~14]{Park_Approximation_Journal_Parti}.

\subsection{When $1 < |a| \leq 2.5$}
\label{subsec:ageq1}
\begin{figure*}[htbp]
\begin{center}
        \begin{subfigure}[b]{0.4\textwidth}
                \centering
                \includegraphics[width=\textwidth]{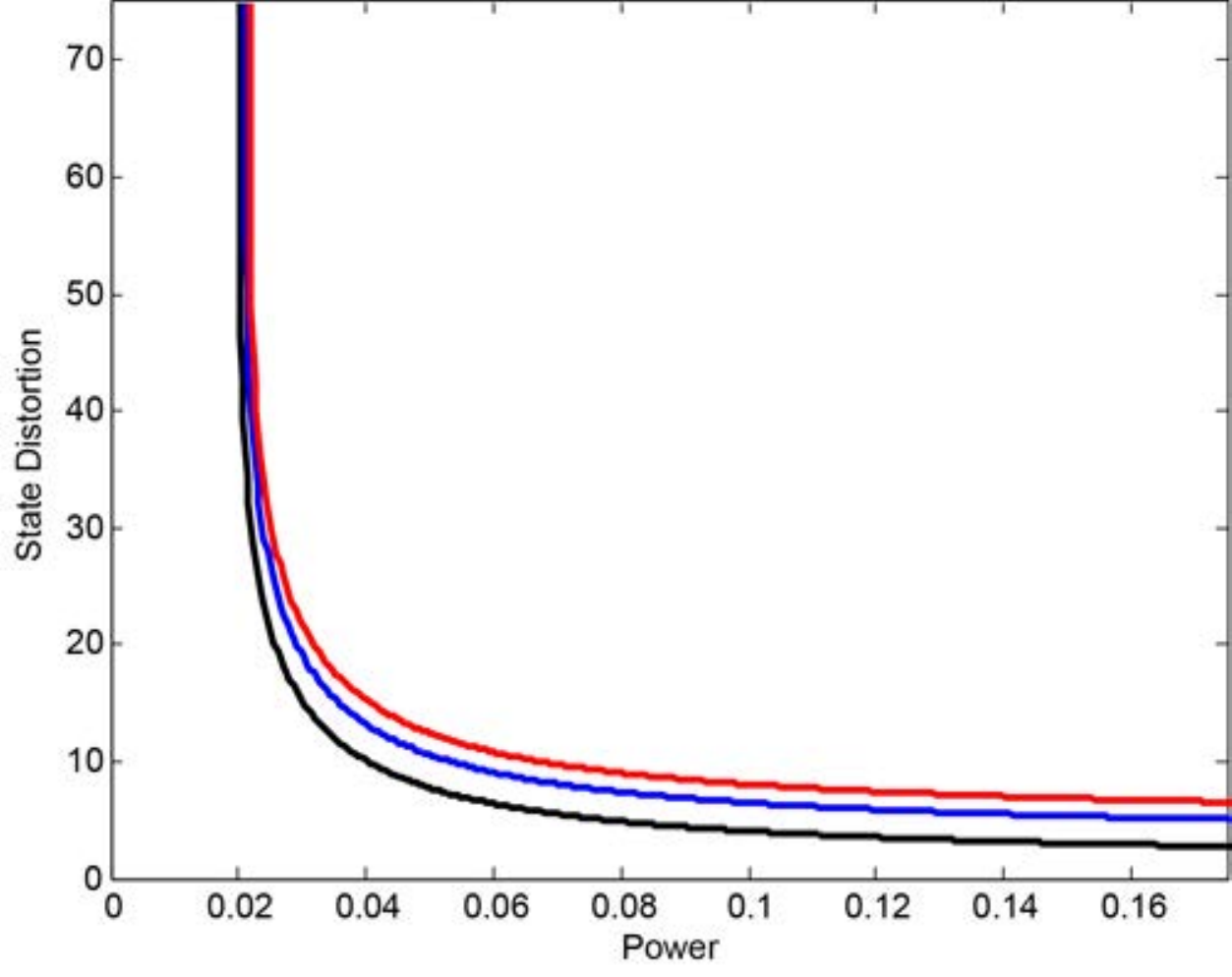}
                \caption{When $\sigma_v^2 = 1$ (Black line), $\sigma_v^2 = 10$ (Blue line), $\sigma_v^2 = 20$ (Red line)}
                \label{fig:simulaless21}
        \end{subfigure}
        \begin{subfigure}[b]{0.4\textwidth}
                \centering
                \includegraphics[width=\textwidth]{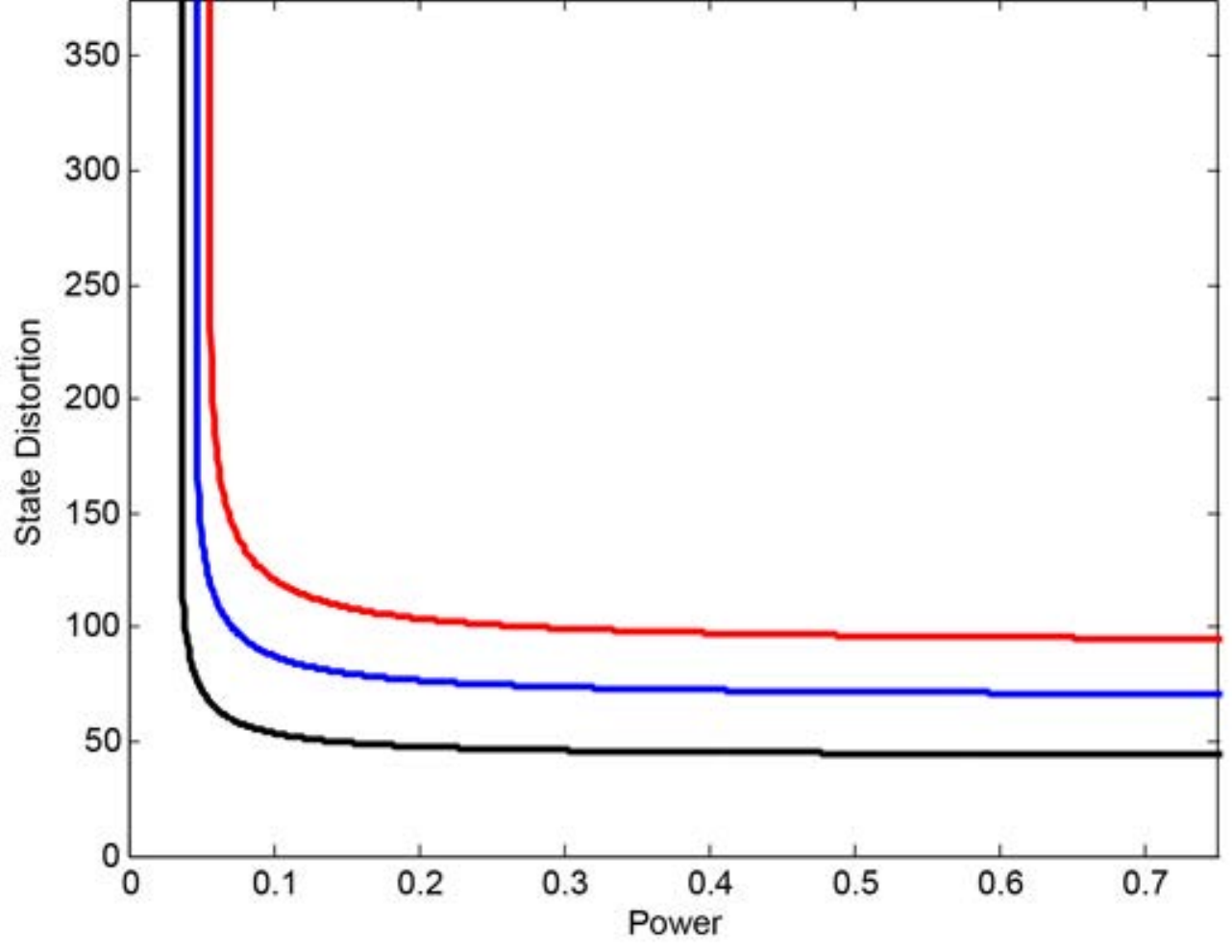}
                \caption{When $\sigma_v^2 = 1000$ (Black line), $\sigma_v^2 = 2000$ (Blue line), $\sigma_v^2 = 3000$ (Red line)}
                \label{fig:simulaless22}
        \end{subfigure}
\caption{The Optimal State Distortion-Input Power Tradeoff: When $a =1.01$ with different values of $\sigma_v^2$}
\label{fig:simulaless2}
\end{center}
\end{figure*}

\begin{figure*}[htbp]
\begin{center}
        \begin{subfigure}[b]{0.4\textwidth}
                \centering
                \includegraphics[width=\textwidth]{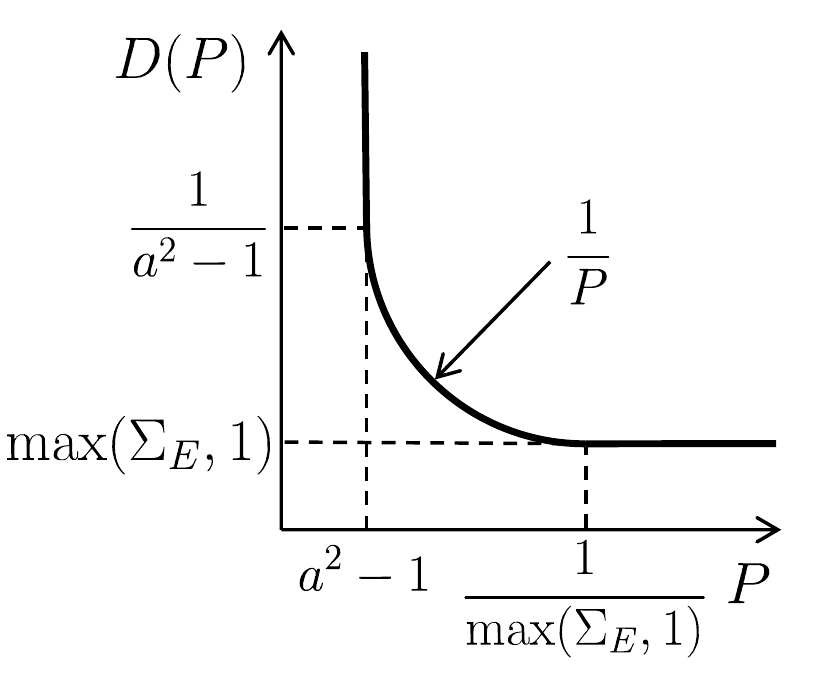}
                \caption{When $\max(\Sigma_E,1) \leq \frac{1}{a^2-1}$}
                \label{fig:aless21}
        \end{subfigure}
        \begin{subfigure}[b]{0.4\textwidth}
                \centering
                \includegraphics[width=\textwidth]{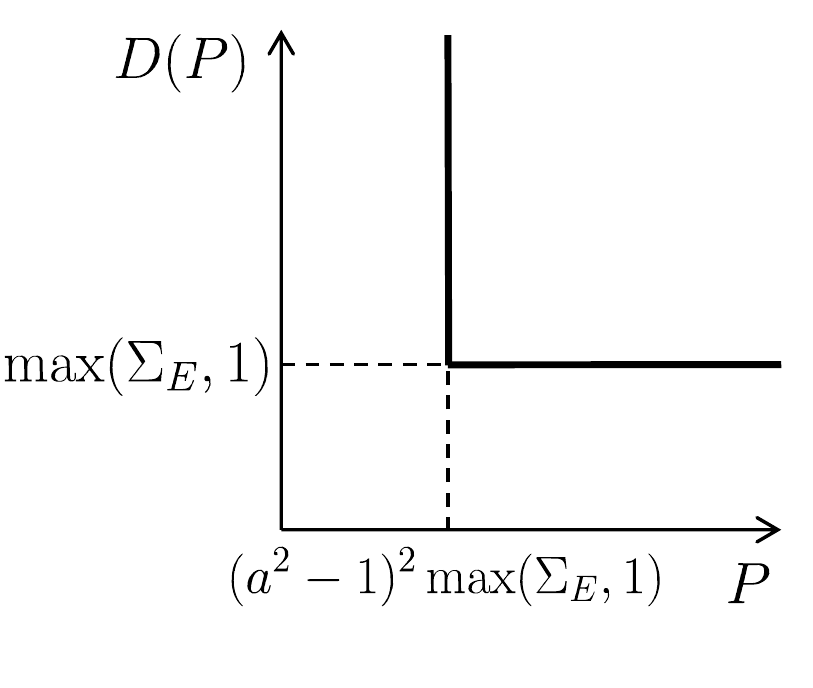}
                \caption{When $\max(\Sigma_E,1) \geq \frac{1}{a^2-1}$}
                \label{fig:aless22}
        \end{subfigure}
\caption{Conceptual Plot of State Distortion-Input Power Tradeoff: When $1 < |a| \leq 2.5$}
\label{fig:aless2}
\end{center}
\end{figure*}
Let's consider the case\footnote{Here, the explicit number $2.5$ does not have to be $2.5$. In fact, we can choose any fixed number like $2,3,5,6,\cdots$.} when $1 < |a| \leq 2.5$. Just like the case of $|a|=1$, the state distortion saturates at $a^2 \Sigma_E +1 \approx \max(\Sigma_E, 1)$ for all $P$, and the state distortion inversely proportionally increases as the power decreases.

However, there is a significant difference from the previous case of $|a|=1$. Since the system is unstable by itself, when the power is too small the state distortion diverges to infinity. Figure~\ref{fig:simulaless21} shows this behavior. Furthermore, it is well known that the minimum capacity to stabilize unstable plants is $\log |a|$. Since the variance of $w[n]$ is $1$, the capacity from the controller to the plant can be thought as of $\frac{1}{2} \log (1+P)$. Therefore, the stabilizability condition $\frac{1}{2} \log(1+P) > \log |a|$ gives $P \geq a^2-1$ to stabilize the system.

Based on the above discussion, we can draw a conceptual power-distortion tradeoff curve as shown in Figure~\ref{fig:aless21}. Like Figure~\ref{fig:aeq11}, when the power is larger than $\frac{1}{\max(\Sigma_E, 1)}$, the state distortion is saturated at $\max(\Sigma_E, 1)$. When the power is between $a^2-1$ and $\frac{1}{\max(\Sigma_E, 1)}$, the state distortion is inversely proportional to the power. However unlike Figure~\ref{fig:aeq11} when the power is smaller than $(a^2-1)$, the controller cannot stabilize the system, so the state distortion diverges to infinity.

Furthermore, Figure~\ref{fig:simulaless22} shows that as $\Sigma_E$ increases, the gap between $(a^2-1)$ and $\frac{1}{\max(\Sigma_E, 1)}$ (the interval where the distortion is inversely proportional to the power) decreases, i.e. the boundary of the optimal tradeoff region shrinks. Eventually, the whole boundary will converge to one point. Figure~\ref{fig:aless22} conceptualize this situation. When $\Sigma_E$ is large enough so that $\max(\Sigma_E, 1) \geq \frac{1}{a^2-1}$, we need at least $(a^2-1)^2 \max(\Sigma_E,1)$ controller power to stabilize the plant, and the corresponding state distortion saturates at the Kalman filtering performance $\max(\Sigma_E, 1)$.

The following corollary shows a formal statement of these conceptual tradeoff curves shown in Figure~\ref{fig:aless2}.

\begin{corollary}
Consider the centralized LQG problem shown in Problem~\ref{prob:b}. When $|a|>1$, the achievable power-distortion tradeoff $(D_{\sigma_v}(P),P)$ by the strategies of Definition~\ref{def:cen} is upper bounded as follows:\\
(i) $(D_{\sigma_v}(P),P) \leq ((a^2+1) \Sigma_E + \frac{a^2}{a^2-1}, (a^2-1)^2 \Sigma_E + (a^2-1))$\\
(ii) $(D_{\sigma_v}(P),P) \leq (\frac{4(|a|+1)^2}{t},t)$ for all $2(|a|+1)^2(1-(\frac{1}{a})^2) \leq t \leq \frac{2(|a|+1)^2}{\max(1,(a^2+1)\Sigma_E)}$\\
where the definition of $\Sigma_E$ is given in \eqref{eqn:kalmanperf}.

Especially, when $1 < |a| \leq 2.5$, $D_{\sigma_v}(P)$ satisfies the following conditions:\\
(i') $(D_{\sigma_v}(P),P) \leq (7.25 \Sigma_E + \frac{6.25}{a^2-1}, (a^2-1)^2 \Sigma_E + (a^2-1))$\\
(ii') $(D_{\sigma_v}(P),P) \leq (\frac{49}{t},t)$ for all $8(a^2-1)\leq t \leq \frac{8}{\max(1,7.25\Sigma_E)} $
\label{cor:1}
\end{corollary}
\begin{proof}
See Appendix~\ref{app:cor1} for the proof.
\end{proof}

When $\max(\Sigma_E, 1) \leq \frac{1}{a^2-1}$, (ii') of the corollary shows that we can achieve the tradeoff curve shown in Figure~\ref{fig:aless21}. More precisely, when $P=8(a^2-1)$, the statement (ii') reduces to $D_{\sigma_v} (P) \leq \frac{49}{8(a^2-1)}$. Therefore, $(D_{\sigma_v}(P), P) \approx (\frac{1}{a^2-1}, a^2-1)$ is achievable.

When $P=\frac{8}{\max(1,7.25 \Sigma_E)}$, the statement (ii') reduces to $D_{\sigma_v} (P) \leq \frac{49}{8} \max(1,7.25 \Sigma_E)$. Thus, $(D_{\sigma_v}(P), P) \approx (\max(\Sigma_E, 1), \frac{1}{\max(\Sigma_E, 1)})$ is also achievable. Between these two values, the tradeoff is inversely proportional.

When $\max(\Sigma_E, 1) \geq \frac{1}{a^2-1}$, (i') of the corollary shows the tradeoff curve in Figure~\ref{fig:aless22} is achievable. More precisely, with the condition $\max(\Sigma_E, 1) \geq \frac{1}{a^2-1}$, the statement (i') implies
\begin{align}
(D_{\sigma_v}(P), P) & \leq  (7.25 \Sigma_E + \frac{6.25}{a^2-1}, (a^2-1)^2 \Sigma_E + (a^2-1)) \\
&\leq (13.5 \max(\Sigma_E,1), 2(a^2-1)^2 \max(\Sigma_E,1)).
\end{align}
Therefore, the corner point of Figure~\ref{fig:aless22} is achievable up to scaling. The whole tradeoff region is also achievable since we can always achieve the points with more state distortion and input power.

Just like Section~\ref{subsec:a=1}, a careful inspection of Figure~\ref{fig:simulaless21} suggests that $D_{\sigma_v}(P) \approx \frac{1}{P-(a^2-1)} + \max(\Sigma_E, 1)$ may be a better approximation than the one shown in Figure~\ref{fig:aless21}. However, just like the discussion in Section~\ref{subsec:a=1}, the approximation of Figure~\ref{fig:simulaless21} is good enough to prove a constant ratio optimality, and easier to compare with a lower bound on the performance since the approximation is divided into multiple regions. 

In fact, by putting $\Sigma_2 = \infty$ and considering the first controller as the centralized controller, (g), (f), (j) of Corollary~\ref{cor:2} in Section~\ref{sec:aless2} respectively reduce to
\begin{align}
&D_{\sigma_v}(P) = \infty \mbox{ for all } P \leq \frac{1}{20}(a^2-1) \nonumber \\
&D_{\sigma_v}(P) \geq \frac{0.0006976}{P_1} +1 \mbox{ for all } P \leq \frac{1}{150} \nonumber \\
&D_{\sigma_v}(P) \geq \max(0.1035 \Sigma_1, 1). \nonumber
\end{align}
By taking the maximum over these three bounds, we can easily check that the resulting lower bound coincide with the approximation of Figure~\ref{fig:simulaless21} up to a constant, and thereby the average cost can also be characterized within a constant by \cite[Lemma~14]{Park_Approximation_Journal_Parti}.

\subsection{When $ 0.9 \leq |a| < 1$}

\begin{figure*}[htbp]
\begin{center}
        \begin{subfigure}[b]{0.4\textwidth}
                \centering
                \includegraphics[width=\textwidth]{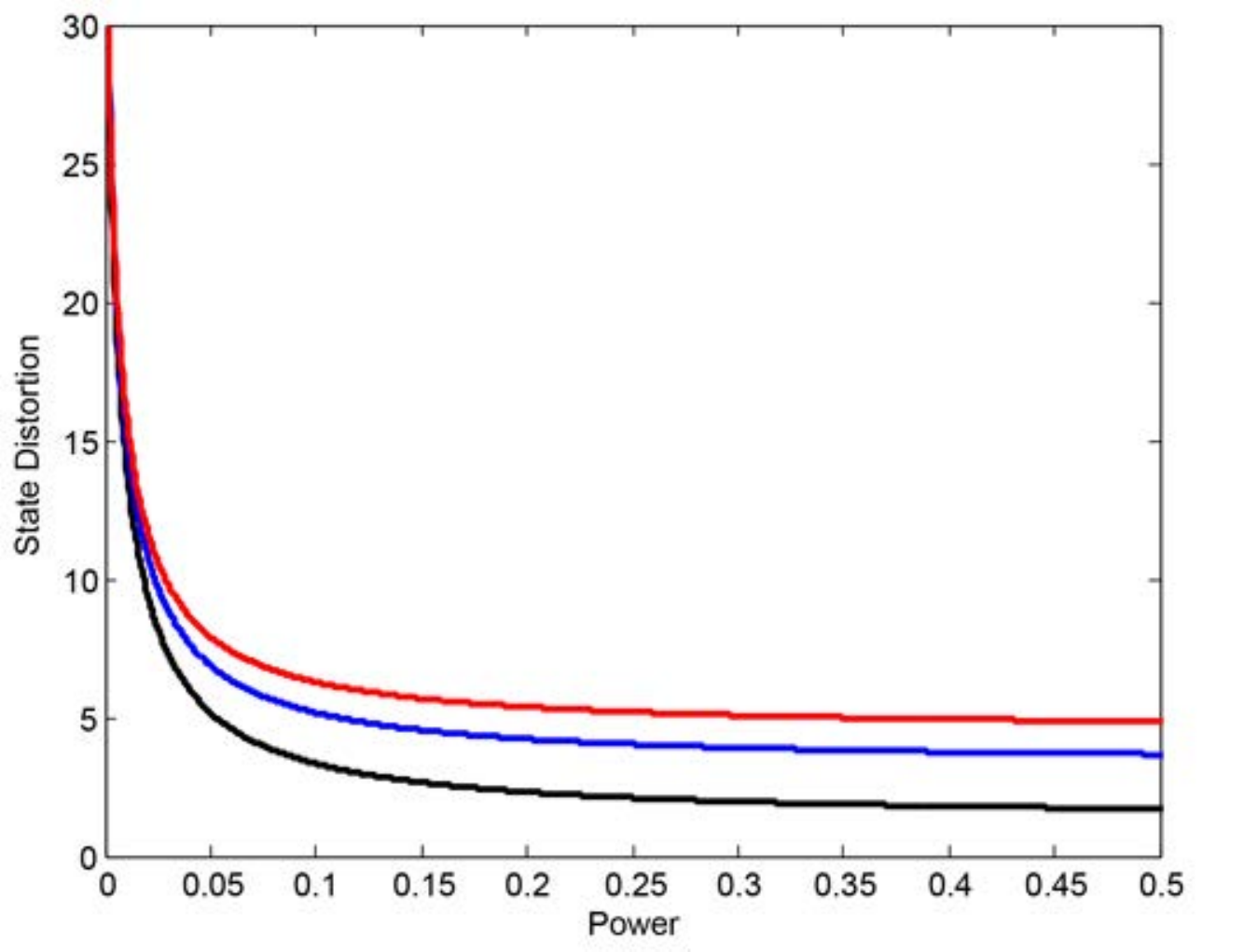}
                \caption{When $\sigma_v^2 = 1$ (Black line), $\sigma_v^2 = 10$ (Blue line), $\sigma_v^2 = 20$ (Red line)}
                \label{fig:simulaless21}
        \end{subfigure}
        \begin{subfigure}[b]{0.4\textwidth}
                \centering
                \includegraphics[width=\textwidth]{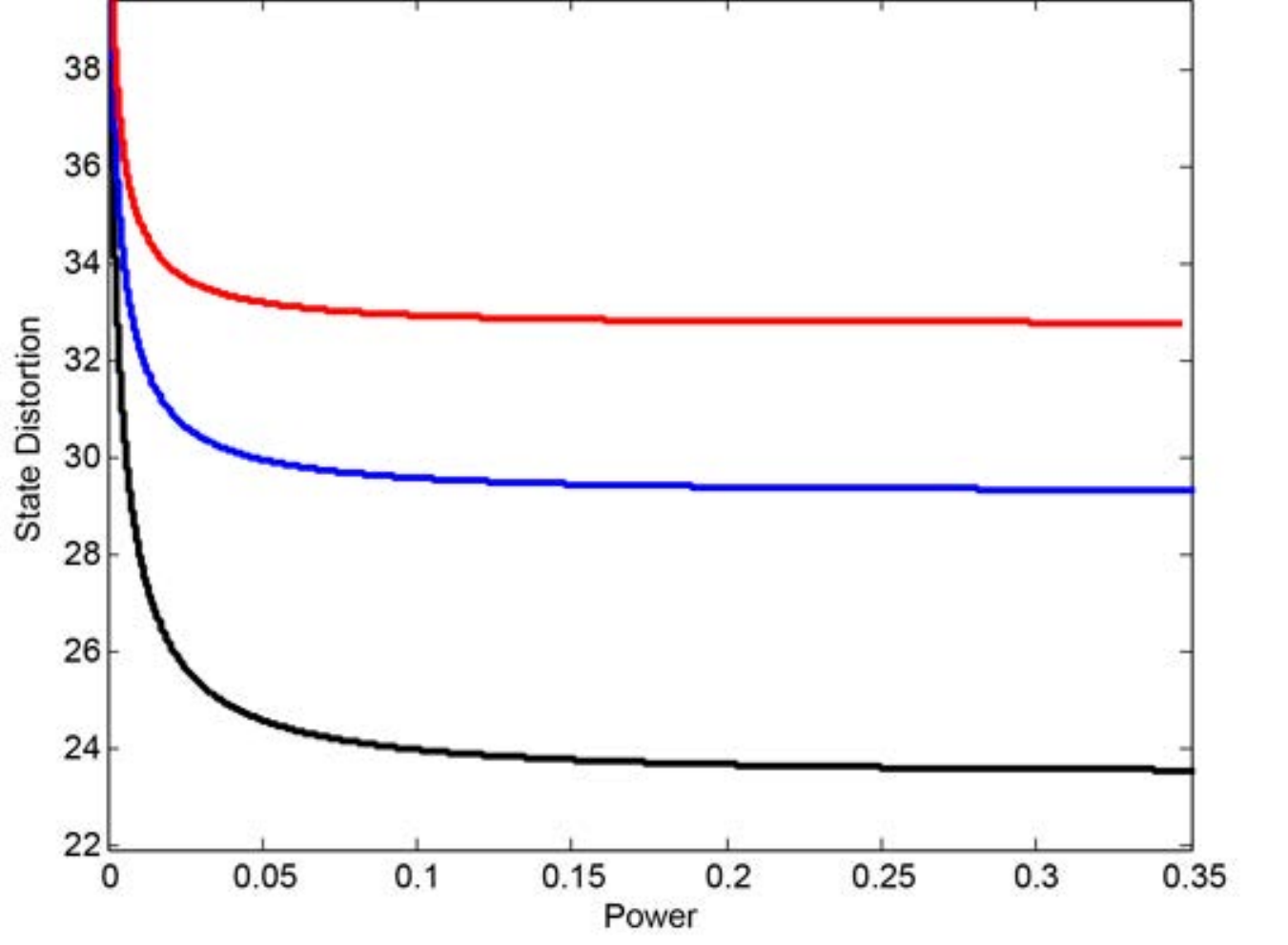}
                \caption{When $\sigma_v^2 = 1000$ (Black line), $\sigma_v^2 = 2000$ (Blue line), $\sigma_v^2 = 3000$ (Red line)}
                \label{fig:simulaless22}
        \end{subfigure}
\caption{The Optimal State Distortion-Input Power Tradeoff: When $a =0.99$ with different values of $\sigma_v^2$}
\label{fig:simulaless2}
\end{center}
\end{figure*}

\begin{figure*}[htbp]
\begin{center}
        \begin{subfigure}[b]{0.4\textwidth}
                \centering
                \includegraphics[width=\textwidth]{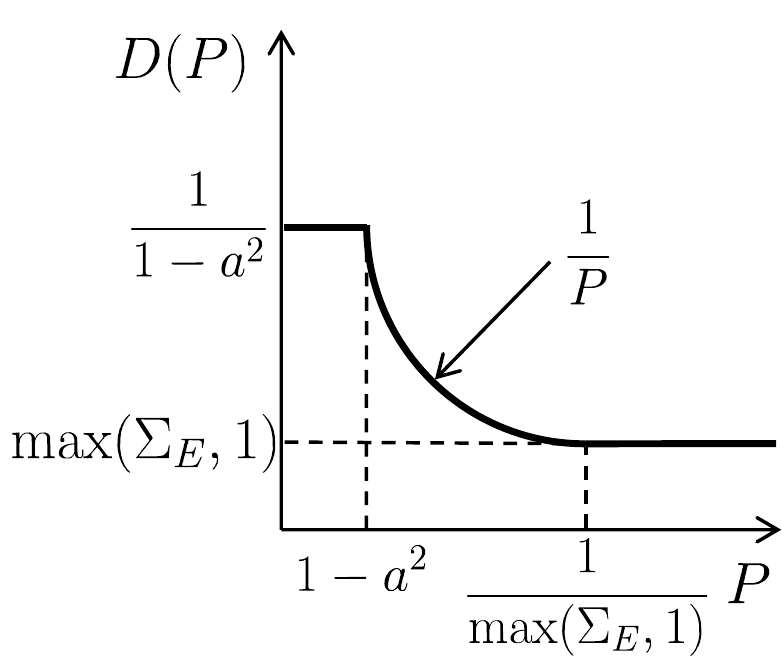}
                \caption{When $\max(\Sigma_E,1) \leq \frac{1}{1-a^2}$}
                \label{fig:aless11}
        \end{subfigure}
        \begin{subfigure}[b]{0.4\textwidth}
                \centering
                \includegraphics[width=\textwidth]{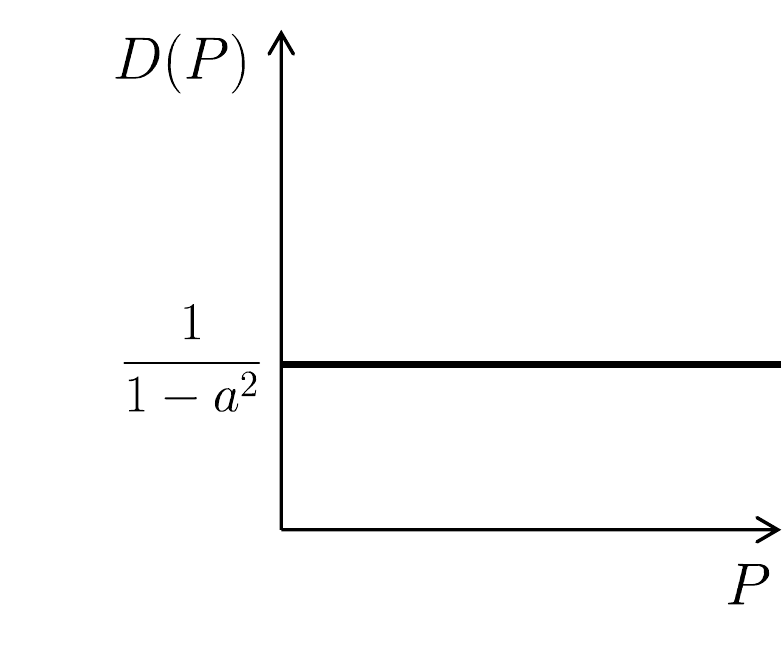}
                \caption{When $\max(\Sigma_E,1) \geq \frac{1}{1-a^2}$}
                \label{fig:aless12}
        \end{subfigure}
\caption{State Distortion-Input Power Tradeoff: When $0.9 \leq  |a| < 1$}
\label{fig:aless1}
\end{center}
\end{figure*}

Let's consider the case when $0.9 \leq |a| < 1$. In contrast to the case of $1 <|a| \leq 2.5$, the system is table by itself in this case. Therefore, the state distortion never increases above $\frac{1}{1-a^2}$. 

As we can see in Figure~\ref{fig:simulaless21}, the essential tradeoff curve is similar to the case of $|a|=1$. For all control power $P$, the state distortion saturates at the Kalman filtering performance $\max(\Sigma_E, 1)$. For the control power between $1-a^2$ and $\frac{1}{\max(\Sigma_E, 1)}$, the state distortion is inversely proportional to the control power.

However, when the power becomes smaller than $1-a^2$, the state distortion becomes larger than $\frac{1}{1-a^2}$ which is achievable even without any control. Therefore, for the power smaller than $1-a^2$, the state distortion stays at $\frac{1}{1-a^2}$. Therefore, a conceptual tradeoff curve is given as Figure~\ref{fig:aless21}.

Furthermore, Figure~\ref{fig:simulaless22} shows that as $\Sigma_E$ increases, the state distortion without control ($\frac{1}{1-a^2}$) and the Kalman filtering performance ($\max(\Sigma_E, 1)$)  becomes similar. Eventually, when $\max(\Sigma_E, 1) \geq \frac{1}{1-a^2}$, as depicted in Figure~\ref{fig:aless22} the minimum state distortion becomes $\frac{1}{1-a^2}$ which is achievable even without any control.

Corollary~\ref{cor:5}  gives formal statements of these observations.

\begin{corollary}
Consider the centralized LQG problem shown in Problem~\ref{prob:b}. When $|a| <1$, the achievable power-distortion tradeoff $(D_{\sigma_v}(P), P)$ by the strategies of Definition~\ref{def:cen} is upper bounded as follows:
\begin{align}
(D_{\sigma_v}(P),P) \leq (\frac{1}{1-a^2},0),
\label{eqn:tradeoff2}
\end{align}
and especially when\footnote{Since $|a| < 1$, the condition $\Sigma_E \leq \frac{1}{1-a^2}$ is equivalent to the condition $\max(1,\Sigma_E) \leq \frac{1}{1-a^2}$.} $\Sigma_E \leq \frac{1}{1-a^2}$ we also have 
\begin{align}
(D_{\sigma_v}(P),P) \leq (\frac{2}{t}, t) \mbox{ for all } 1-a^2 \leq t \leq \frac{1}{\max(1,\Sigma_E)}
\label{eqn:tradeoff3}
\end{align}
where the definition of $\Sigma_E$ is given as \eqref{eqn:kalmanperf}.
\label{cor:5}
\end{corollary}
\begin{proof}
See Appendix~\ref{app:cor1} for the proof.
\end{proof}

When $\max(\Sigma_E, 1) \geq \frac{1}{1-a^2}$, \eqref{eqn:tradeoff2} shows the tradeoff curve shown in Figure~\ref{fig:aless12} is achievable. 

When $\max(\Sigma_E, 1) \leq \frac{1}{1-a^2}$, \eqref{eqn:tradeoff3} shows the inversely proportional tradeoff curve shown in Figure~\ref{fig:aless11} when the power is between $1-a^2$ and $\frac{1}{\max(1,\Sigma_E)}$.

In fact, in Figure~\ref{fig:simulaless21} we cannot find a flat region for the power between $0$ and $1-a^2$ which is shown in the approximation of Figure~\ref{fig:aless11}. Therefore, like Section~\ref{subsec:a=1}, \ref{subsec:ageq1}, a better approximation of the tradeoff might be $D_{\sigma_v}(P) \approx \frac{1}{P-(a^2-1)}+\max(\Sigma_E, 1)$ and worth to explore. However, the approximation of Figure~\ref{fig:simulaless21} is good enough to give a constant ratio optimality result. For example, if we compute the distortion for $P \in [0, {1-a^2}]$ with this new approximation, we get $D_{\sigma_v}(P) \in [\frac{1}{2(1-a^2)}+\max(\Sigma_E, 1), \frac{1}{1-a^2}+\max(\Sigma_E, 1) ]$. Especially, for $\max(\Sigma_E, 1)\leq \frac{1}{1-a^2}$ which is the case of Figure~\ref{fig:aless11}, this interval is included in
\begin{align}
[\frac{1}{2(1-a^2)}+\max(\Sigma_E, 1), \frac{1}{1-a^2}+\max(\Sigma_E, 1) ]  \subseteq [\frac{1}{2(1-a^2)}, \frac{2}{1-a^2}]. \nonumber
\end{align}
Therefore, the approximation is essentially the same as the one of Figure~\ref{fig:aless11}, $\frac{1}{1-a^2}$, up to a constant.

Furthermore, Corollary~\ref{cor:6} of Section~\ref{sec:subless1} gives a matching lower bound to the approximation of Figure~\ref{fig:aless11}. First notice that as $\sigma_{v2}$ goes to infinity, the Kalman filtering performance $\Sigma_2$ converges to $\frac{1}{1-a^2}$ which is the disturbance of the stable system without any control. Thus, by putting $\Sigma_2 = \frac{1}{1-a^2}$, thinking the second controller as the centralized controller, and considering the case of $\frac{1}{1-a^2} \geq 40$, the conditions (a), (b), (e) of Corollary~\ref{cor:6} respectively reduce to
\begin{align} 
&D_{\sigma_v}(P) \geq \frac{0.009131}{1-a^2}+1 \mbox{ for all } P \leq 1-a^2 \nonumber\\
&D_{\sigma_v}(P) \geq \frac{0.009131}{P} + 1 \mbox{ for all } 1-a^2 \leq P  \leq \frac{1}{40} \nonumber\\
&D_{\sigma_v}(P) \geq \max(0.2636\Sigma_1, 1) \mbox{ for all } P. \nonumber
\end{align}
Therefore, we can easily observe that by taking the maximum of these bounds, we get the matching lower bound to Figure~\ref{fig:aless11} up to a constant. Therefore, by \cite[Lemma~14]{Park_Approximation_Journal_Parti}, we can also characterize the average cost within a constant ratio using the approximation of Figure~\ref{fig:aless11}.

\subsection{When $|a| \leq 0.9$}

In this case, the state distortion $\frac{1}{1-a^2}$ which can be obtained without any control input, is already small enough (smaller than $5.27$). Therefore, the tradeoff curve is essentially the same as Figure~\ref{fig:aless12}, which is achievable with zero control input.

\section{Lower bounds and Constant Ratio Results for the decentralized LQG problems}
\label{sec:lower}

Now, we intuitively understood the power-distortion tradeoff of the centralized LQG problems with scalar plants. Based on this understanding, we will prove that the single controller linear strategies are enough to achieve the optimal decentralized LQG performance within a constant ratio. In other words, $(D(P_1, P_2), P_1, P_2)$ of Problem~\ref{prob:c} is essentially $(\min(D_{\sigma_{v1}}(P_1), D_{\sigma_{v2}}(P_2)), P_1, P_2)$ where the definition of $D_{\sigma_v}(P)$ is given in Problem~\ref{prob:b}.

For the upper bound on the optimal cost of the decentralized LQG problems, we can simply use the centralized controller's performance shown in Corollary~\ref{cor:1}, \ref{cor:3}, \ref{cor:5}. However, we still need a lower bound on the cost of the decentralized LQG problems, and it turns out the naive lower bound we can obtain by merging two decentralized controllers to a centralized controller is too loose to prove a constant ratio optimality.

Therefore, in this section, we will give a non-trivial lower bound based on information theory~\cite{Cover} and prove that the proposed lower bounds are tight within a constant ratio.

\subsection{When $1 < |a| \leq 2.5$}
\label{sec:aless2}

The ideas for the lower bounds are essentially the same as ones shown in \cite{Park_Approximation_Journal_Parti}. The main idea is the geometric slicing, which can be thought as a counterpart of cutset bounds in information theory~\cite{Cover}. We refer \cite{Park_Approximation_Journal_Parti} for a detailed description of the idea. The only different from the geometric slicing lemma shown in \cite[Lemma 8]{Park_Approximation_Journal_Parti} is that here we use allow arbitrary sequences for slicing the problem since we will use arithmetic sequences to slice the problem for the $|a| =1$ case.

As we did in \cite{Park_Approximation_Journal_Parti}, we first introduce sliced finite-horizon problems.
\begin{problem}[Sliced Finite-horizon LQG problem for Problem~\ref{prob:a}]
Let the system equations, the problem parameters, the underlying random variables, and the restrictions on the controllers be given exactly the same as Problem~\ref{prob:a}. However, now for given $k, k_1, k_2 \in \mathbb{N} (k_1 \leq k, k_2 \leq k)$ and positive sequences $\alpha_{k_1}, \alpha_{k_1+1}, \cdots, \alpha_{k-1}$ and $\beta_{k_1}, \beta_{k_1+1}, \cdots, \beta_{k-1}$, the control objective is 
\begin{align}
\inf_{u_1, u_2} q \mathbb{E}[x^2[k]] + r_1 \sum_{k_1 \leq i \leq k-1} \alpha_i \mathbb{E}[u_1^2[i]] + r_2 \sum_{k_2 \leq i \leq k-1} \beta_i \mathbb{E}[u_2^2[n]].
\end{align}
\label{prob:sliced}
\end{problem}

\begin{lemma}[Geometric Slicing]
Let the system equations, the problem parameters, the underlying random variables, and the restrictions on the controllers be given as in Problem~\ref{prob:a}. When $\sigma_0^2=0$, for all $k, k_1, k_2 \in \mathbb{N} (k_1 \leq k, k_2 \leq k)$ and positive sequences $\alpha_{k_1}, \alpha_{k_1+1}, \cdots, \alpha_{k}$ and $\beta_{k_1}, \beta_{k_1+1}, \cdots, \beta_{k}$, the infinite-horizon cost of Problem~\ref{prob:a} is lower bounded by the finite-horizon cost of Problem~\ref{prob:sliced}, i.e.
\begin{align}
&\inf_{u_1, u_2} \limsup_{N \rightarrow \infty} \frac{1}{N} \sum_{0 \leq n \leq N-1} (q \mathbb{E}[x^2[n]] + r_1 \mathbb{E}[u_1^2[n]] + r_2 \mathbb{E}[u_2^2[n]])\\
&\geq \inf_{u_1, u_2} q \mathbb{E}[x^2[k]] + r_1 \sum_{k_1 \leq i \leq k-1} \alpha_i \mathbb{E}[u_1^2[i]] + r_2 \sum_{k_2 \leq i \leq k-1} \beta_i \mathbb{E}[u_2^2[n]].
\end{align}
Furthermore, both costs are increasing functions of $\sigma_0^2$ and when $\sigma_0^2 = 0$, $u_1[0] = 0$ and $u_2[0] = 0$ are optimal for both problems.
\label{lem:slicing2}
\end{lemma}
\begin{proof}
The proof essentially the same as the proof of \cite[Lemma 8]{Park_Approximation_Journal_Parti}. The only difference is that the geometric sequences in \cite[Lemma 8]{Park_Approximation_Journal_Parti} have to be replaced by $\alpha_n$ and $\beta_n$.
\end{proof}

Using this lemma, we can lower bound on the cost of the decentralized LQG problems as follows.

\begin{lemma}
Define $S_{L}$ as the set of $(k_1,k_2,k)$ such that $k_1 , k_2, k \in \mathbb{N}$ and $1 \leq k_1 \leq k_2 \leq k$. We also define $D_{L,1}(\widetilde{P_1}, \widetilde{P_2}, k_1, k_2, k)$ as follows:
\begin{align}
D_{L,1}(\widetilde{P_1},\widetilde{P_2}; k_1 ,k_2, k)&:=
(
\sqrt{
\frac{ \Sigma + a^{2(k-k_1)} \frac{1-a^{-2(k_2-k_1)}}{1-a^{-2}}  }{2^{2I'(\widetilde{P_1})}}
+a^{2(k-k_2)} \frac{1-a^{-2(k-k_2)}}{1-a^{-2}}
}\\
&
-
\sqrt{a^{2(k-k_1-1)} \frac{(1-a^{-(k-k_1)} )^2}{(1-a^{-1})^2} \widetilde{P_1}}
-
\sqrt{a^{2(k-k_2-1)} \frac{(1-a^{-(k-k_2)} )^2}{(1-a^{-1})^2} \widetilde{P_2}}
)_+^2 + 1
\end{align}
where
\begin{align}
\Sigma&=\frac{ a^{2(k-1)}\frac{1-a^{-2(k_1-1)}}{1-a^{-2}} }{2^{2I}}\\
I&=\frac{k_1-1}{2} \log( 1+ \frac{1}{(k_1-1)\sigma_{v1}^2}
\frac{a^{2(k_1-2)}(1-a^{-2(k_1-1)})}{1-a^{-2}}
\frac{1-a^{-2(k_1-1)}}{1-a^{-2}})\\
&+
\frac{k_1-1}{2} \log( 1+ \frac{1}{(k_1-1)\sigma_{v2}^2}
\frac{a^{2(k_1-2)}(1-a^{-2(k_1-1)})}{1-a^{-2}}
\frac{1-a^{-2(k_1-1)}}{1-a^{-2}})\\
I'(\widetilde{P_1})&= \frac{k_2-k_1}{2} \log(
1+\frac{1}{(k_2-k_1)\sigma_{v2}^2}(2a^{2(k_2-1-k)} \frac{1-a^{-2(k_2-k_1)}}{1-a^{-2}} \Sigma
+ 2a^{2(k_2-1-k_1)} \frac{1-a^{-2(k_2-k_1)}}{1-a^{-2}} \frac{1-a^{-2(k_2-k_1)}}{1-a^{-2}} \\
&+ 2a^{2(k_2-k_1-2)} \frac{1-a^{-2(k_2-k_1)}}{1-a^{-2}} \frac{(1-a^{-(k_2-1-k_1)})(1-a^{-(k-k_1)})}{(1-a^{-1})^2} \widetilde{P_1} )).
\end{align}
Here, when $k_1-1=0$, $I=0$ and when $k_2-k_1=0$, $I'(\widetilde{P_1})=0$.

Let $|a| > 1$. Then, for all $q, r_1, r_2, \sigma_0, \sigma_{v1}, \sigma_{v2} \geq 0$, the minimum cost \eqref{eqn:part11} of Problem~\ref{prob:a} is lower bounded as follows:
\begin{align}
&\inf_{u_1, u_2} \limsup_{N \rightarrow \infty} \frac{1}{N} \sum_{0 \leq n < N}
q \mathbb{E}[x^2[n]] + r_1 \mathbb{E}[u_1^2[n]] + r_2 \mathbb{E}[u_2^2[n]] \\
&\geq  \sup_{(k_1, k_2, k) \in S_L} \min_{\widetilde{P_1}, \widetilde{P_2} \geq 0}
q D_{L,1}(\widetilde{P_1}, \widetilde{P_2}; k_1 , k_2, k) + r_1 \widetilde{P_1} + r_2 \widetilde{P_2}.
\end{align}
\label{lem:aless22}
\end{lemma}
\begin{proof}
For simplicity, we assume $a >1$, $1< k_1 < k_2 <k$. The remaining cases when $a < -1$ or $k_1=1$ or $k_2-k_1=0$ or $k=k_2$ easily follow with minor modifications.

$\bullet$ Geometric Slicing: We apply the geometric slicing idea of Lemma~\ref{lem:slicing2} to get a finite-horizon problem. By putting $\alpha_{k_1}=( \frac{1-a^{-1}}{1-a^{-(k-k_1)}} )$, $\alpha_{k_1+1}=( \frac{1-a^{-1}}{1-a^{-(k-k_1)}} ) a^{-1}$, $\cdots$, $\alpha_k=( \frac{1-a^{-1}}{1-a^{-(k-k_1)}} )a^{-k+1+k_1}$ and $\beta_{k_2}=( \frac{1-a^{-1}}{1-a^{-(k-k_2)}} )$, $\beta_{k_2+1}=( \frac{1-a^{-1}}{1-a^{-(k-k_2)}} )a^{-1}$, $\cdots$, $\beta_{k-1}=( \frac{1-a^{-1}}{1-a^{-(k-k_2)}} )a^{-k+1+k_2}$ the average cost is lower bounded by
\begin{align}
&\inf_{u_1,u_2}(
q \mathbb{E}[x^2[k]] \\
&+r_1
\underbrace{(( \frac{1-a^{-1}}{1-a^{-(k-k_1)}} )\mathbb{E}[u_1^2[k_1]]+( \frac{1-a^{-1}}{1-a^{-(k-k_1)}} )a^{-1}\mathbb{E}[u_1^2[k_1+1]]+\cdots+
( \frac{1-a^{-1}}{1-a^{-(k-k_1)}} )a^{-k+1+k_1} \mathbb{E}[u_1^2[k-1]])}_{:=\widetilde{P_1}}\\
&+r_2
\underbrace{(( \frac{1-a^{-1}}{1-a^{-(k-k_2)}} )\mathbb{E}[u_2^2[k_2]]+( \frac{1-a^{-1}}{1-a^{-(k-k_2)}} )a^{-1}\mathbb{E}[u_2^2[k_2+1]]+\cdots+
( \frac{1-a^{-1}}{1-a^{-(k-k_2)}} )a^{-k+1+k_2} \mathbb{E}[u_2^2[k-1]])}_{:=\widetilde{P_2}})
\end{align}
Here, we denote the second and third terms as $\widetilde{P_1}$ and $\widetilde{P_2}$ respectively. 

$\bullet$ Three stage division: 
As we did \cite{Park_Approximation_Journal_Parti}, we will divide the finite-horizon problem into three time intervals --- information-limited interval, Witsenhausen's interval, power-limited interval. Define
\begin{align}
&W_1:=a^{k-1}w[0]+ \cdots + a^{k-k_1+1}w[k_1-2]\\
&W_2 := a^{k-k_1}w[k_1-1]+\cdots+a^{k-k_2+1}w[k_2-2] \\
&W_3 := a^{k-k_2} w[k_2-1]+\cdots + a w[k-2] \\
&U_{11}:=a^{k-2}u_1[1]+ \cdots+ a^{k-k_1} u_1[k_1-1] \\
&U_{21}:=a^{k-2}u_2[1]+ \cdots+ a^{k-k_1} u_2[k_1-1] \\
&U_1 := a^{k-k_1-1}u_1[k_1]+ \cdots + u_1[k-1] \\
&U_{22} := a^{k-k_1-1}u_2[k_1]+\cdots+a^{k-k_2}u_2[k_2-1])\\
&U_2 := a^{k-k_2-1}u_2[k_2]+ \cdots + u_2[k-1]\\
&X_1 := W_1 + U_{11} + U_{21} \\
&X_2 := W_2 + U_{22}
\end{align}
$W_1$, $W_2$, $W_3$ represent the distortions of three intervals respectively. 
$U_{11}$ and $U_{21}$ respectively represent the first and second controller inputs in the information-limited interval. $U_{1}$ represents the remaining input of the first controller. $U_{22}$ and $U_{2}$ represent the second controller's inputs in Witsenhausen's and power-limited intervals respectively.

The goal of this proof is grouping control inputs and expanding $x[n]$, so that we reveal the effects of the controller inputs on the state and isolate their effects according to their characteristics.

$\bullet$ Power-Limited Inputs: We will first isolate the power-limited inputs of the controllers, i.e. the first controller's input in the Witsenhausen's and power-limited intervals, and the second controller's input in the power-limited interval. Notice that
\begin{align}
x[k]&=w[k-1]+aw[k-2]+ \cdots + a^{k-1}w[0] \\
&\quad +u_1[k-1]+au_1[k-2]+ \cdots + a^{k-1}u_1[0] \\
&\quad +u_2[k-1]+au_2[k-2]+ \cdots + a^{k-1}u_2[0] \\
&=
(a^{k-1}w[0]+ \cdots + a^{k-k_1+1}w[k_1-2] \\
&\quad+a^{k-2}u_1[1]+ \cdots+ a^{k-k_1} u_1[k_1-1] \\
&\quad+a^{k-2}u_2[1]+ \cdots+ a^{k-k_1} u_2[k_1-1]) \\
&\quad +(
a^{k-k_1}w[k_1-1]+\cdots+a^{k-k_2+1}w[k_2-2] \\
&\quad+a^{k-k_1-1}u_2[k_1]+\cdots+a^{k-k_2}u_2[k_2-1])\\
&\quad + (
a^{k-k_2} w[k_2-1]+\cdots + a w[k-2]
)\\
&\quad + (
a^{k-k_1-1}u_1[k_1]+ \cdots + u_1[k-1]
) \\
&\quad + (
a^{k-k_2-1}u_2[k_2]+ \cdots + u_2[k-1]
) \\
&\quad+w[k-1].
\end{align}
Therefore, by \cite[Lemma~1]{Park_Approximation_Journal_Parti} we have
\begin{align}
\mathbb{E}[x^2[k]]&=
\mathbb{E}[(X_1+X_2+W_3+U_1+U_2+w[k-1])^2] \\
&=\mathbb{E}[(X_1+X_2+W_3+U_1+U_2)^2]+\mathbb{E}[w^2[k-1]] \\
&\geq (\sqrt{\mathbb{E}[(X_1+X_2+W_3)^2]}-\sqrt{\mathbb{E}[U_1^2]}-\sqrt{\mathbb{E}[U_2^2]})_+^2 + 1 \\
&= (\sqrt{\mathbb{E}[(X_1+X_2)^2]+\mathbb{E}[W_3^2]}-\sqrt{\mathbb{E}[U_1^2]}-\sqrt{\mathbb{E}[U_2^2]})_+^2 + 1.  \label{eqn:aless28}
\end{align}
where the last equality comes from the causality.
Here, we can see that $\mathbb{E}[(X_1+X_2)^2]$ does not depend on the power-limited inputs.

$\bullet$ Information-Limited Interval:  We will bound the remaining state distortion after the information-limited interval. Define $y_1'$ and $y_2'$ as follows. 
\begin{align}
&y_1'[k]= a^{k-1}w[0]+ a^{k-2}w[1]+ \cdots + w[k-1]+ v_1[k]\\
&y_2'[k]= a^{k-1}w[0]+ a^{k-2}w[1]+ \cdots + w[k-1]+ v_2[k]
\end{align}
Here, $y_1'[k]$, $y_2'[k]$ can be obtained by removing $u_1[1:k-1]$, $u_2[1:k-1]$ from $y_1[k]$, $y_2[k]$, and $u_1[k]$ and $u_2[k]$ are functions of $y_1[1:k]$ and $y_2[1:k]$ respectively.
Therefore, we can see that $y_1[1:k], y_2[1:k]$ are functions of $y_1'[1:k],y_2'[1:k]$. Moreover  $W_1,y_1'[1:k_1-1],y_2'[1:k_1-1]$ are jointly Gaussian.

Let
\begin{align}
&W_1' := W_1 - \mathbb{E}[W_1|y_1'[1:k_1-1],y_2'[1:k_1-1]]\\
&W_1'' := \mathbb{E}[W_1|y_1'[1:k_1-1],y_2'[1:k_1-1]].
\end{align}
Then, $W_1'$, $W_1''$, $W_2$ are independent Gaussian random variables. Moreover, $W_1', W_2$ are independent from $y_1'[1:k_1-1],y_2'[1:k_1-1]$. $W_1''$ is a function of $y_1'[1:k_1-1],y_2'[1:k_1-1]$.

Now, let's lower bound $\mathbb{E}[(X_1+X_2)^2]$. Since Gaussian maximizes the entropy, we have
\begin{align}
&\frac{1}{2}\log( 2 \pi e \mathbb{E}[(X_1+X_2)^2] \\
&\geq h(X_1+X_2)\\
&\geq h(X_1+X_2|y_1'[1:k_1-1],y_2'[1:k_1-1],y_2[k_1:k_2-1])\\
&= h(W_1'+W_1''+U_{11}+U_{12}+W_2+U_{22}|y_1'[1:k_1-1],y_2'[1:k_1-1],y_2[k_1:k_2-1])\\
&= h(W_1'+W_2|y_1'[1:k_1-1],y_2'[1:k_1-1],y_2[k_1:k_2-1]). \label{eqn:14converse5}
\end{align}
We will first lower bound the variance of $W_1'$. Notice that
\begin{align}
\mathbb{E}[y_1'[k]^2]&=a^{2(k-1)}+ a^{2(k-2)} + \cdots + 1 + \sigma_{v1}^2 \\
&=a^{2(k-1)} \frac{1-a^{-2k}}{1-a^{-2}} + \sigma_{v1}^2
\end{align}
and
\begin{align}
\mathbb{E}[y_2'[k]^2]&=a^{2(k-1)}+ a^{2(k-2)} + \cdots + 1 + \sigma_{v2}^2 \\
&=a^{2(k-1)} \frac{1-a^{-2k}}{1-a^{-2}} + \sigma_{v2}^2.
\end{align}
Thus, we have
\begin{align}
&I(W_1;y_1'[1:k_1-1],y_2'[1:k_1-1]) \\
&= h(y_1'[1:k_1-1],y_2'[1:k_1-1])-h(y_1'[1:k_1-1],y_2'[1:k_1-1]| W_1) \\
&\leq \sum_{1 \leq i \leq k_1-1}h(y_1'[i]) + \sum_{1 \leq i \leq k_1-1} h(y_2'[i])
-\sum_{1 \leq i \leq k_1-1}h(v_1[i]) - \sum_{1 \leq i \leq k_1-1} h(v_2[i]) \\
&\leq \sum_{1 \leq k \leq k_1 -1} \frac{1}{2} \log(\frac{ a^{2(k-1)} \frac{1-a^{-2k}}{1-a^{-2}} + \sigma_{v1}^2}{\sigma_{v1}^2}) +
\sum_{1 \leq k \leq k_1 -1} \frac{1}{2} \log(\frac{ a^{2(k-1)} \frac{1-a^{-2k}}{1-a^{-2}} + \sigma_{v2}^2}{\sigma_{v2}^2}) \\
& = \frac{1}{2} \log(\prod_{1 \leq k \leq k_1 -1} \frac{ a^{2(k-1)} \frac{1-a^{-2k}}{1-a^{-2}} + \sigma_{v1}^2}{\sigma_{v1}^2} )+
\frac{1}{2} \log(\prod_{1 \leq k \leq k_1 -1} \frac{ a^{2(k-1)} \frac{1-a^{-2k}}{1-a^{-2}} + \sigma_{v2}^2}{\sigma_{v2}^2} )\\
& \overset{(A)}{\leq}
\frac{k_1-1}{2} \log( \frac{1}{k_1-1}\sum_{1 \leq k \leq k_1 -1} \frac{ a^{2(k-1)} \frac{1-a^{-2k}}{1-a^{-2}} + \sigma_{v1}^2}{\sigma_{v1}^2} )+
\frac{k_1-1}{2} \log( \frac{1}{k_1-1}\sum_{1 \leq k \leq k_1 -1} \frac{ a^{2(k-1)} \frac{1-a^{-2k}}{1-a^{-2}} + \sigma_{v2}^2}{\sigma_{v2}^2} )\\
& =
\frac{k_1-1}{2} \log( 1+ \frac{1}{k_1-1}\sum_{1 \leq k \leq k_1 -1} \frac{ a^{2(k-1)} \frac{1-a^{-2k}}{1-a^{-2}}}{\sigma_{v1}^2} )+
\frac{k_1-1}{2} \log( 1+ \frac{1}{k_1-1}\sum_{1 \leq k \leq k_1 -1} \frac{ a^{2(k-1)} \frac{1-a^{-2k}}{1-a^{-2}}}{\sigma_{v2}^2} )\\
&\leq
\frac{k_1-1}{2} \log( 1+ \frac{1}{k_1-1}\sum_{1 \leq k \leq k_1 -1} \frac{ a^{2(k-1)} \frac{1-a^{-2(k_1-1)}}{1-a^{-2}}}{\sigma_{v1}^2} )+
\frac{k_1-1}{2} \log( 1+ \frac{1}{k_1-1}\sum_{1 \leq k \leq k_1 -1} \frac{ a^{2(k-1)} \frac{1-a^{-2(k_1-1)}}{1-a^{-2}}}{\sigma_{v2}^2} )\\
&=
\frac{k_1-1}{2} \log( 1+ \frac{1}{(k_1-1)\sigma_{v1}^2}
\frac{a^{2(k_1-2)}(1-a^{-2(k_1-1)})}{1-a^{-2}}
\frac{1-a^{-2(k_1-1)}}{1-a^{-2}})\\
&+
\frac{k_1-1}{2} \log( 1+ \frac{1}{(k_1-1)\sigma_{v2}^2}
\frac{a^{2(k_1-2)}(1-a^{-2(k_1-1)})}{1-a^{-2}}
\frac{1-a^{-2(k_1-1)}}{1-a^{-2}})  \label{eqn:aless22}
\end{align}
(A): Arithmetic-Geometric mean.

Let's denote the last equation as $I$. We also have
\begin{align}
\mathbb{E}[W_1^2]&=a^{2(k-1)} + \cdots + a^{2(k-k_1 +1)})\\
&=a^{2(k-1)} ( 1 + \cdots + a^{-2(k_1-2)})\\
&=a^{2(k-1)}\frac{1 - a^{-2(k_1-1)}}{1- a^{-2}}. \label{eqn:aless23}
\end{align}
Now, we can bound the variance of the Gaussian random variable $W_1'$ as follows:
\begin{align}
&\frac{1}{2}\log(2\pi e \mathbb{E}[W_1'^2]) = h(W_1') \\
&\geq h(W_1'|y_1'[1:k_1-1],y_2'[1:k_1-1]) \\
&= h(W_1|y_1'[1:k_1-1],y_2'[1:k_1-1]) \\
&= h(W_1) - I(W_1;y_1'[1:k_1-1],y_2'[1:k_1-1])\\
&\geq \frac{1}{2} \log( 2 \pi e a^{2(k-1)}\frac{1-a^{-2(k_1-1)}}{1-a^{-2}}) - I
\end{align}
where the last inequality follows from \eqref{eqn:aless22} and \eqref{eqn:aless23}.

Thus, 
\begin{align}
\mathbb{E}[W_1'^2] \geq \frac{ a^{2(k-1)}\frac{1-a^{-2(k_1-1)}}{1-a^{-2}} }{2^{2I}} \label{eqn:aless2300}
\end{align}
and denote the last term as $\Sigma$. Since $W_1'$ is Gaussian, we can write $W_1'=W_1'''+W_1''''$ where $W_1'''\sim \mathcal{N}(0,\Sigma)$, and $W_1''',W_1''''$ are independent.

Moreover, we also have
\begin{align}
\mathbb{E}[W_2^2] &= a^{2(k-k_1)} + \cdots + a^{2(k-k_2+1)}\\
&= a^{2(k-k_1)} \frac{1-a^{-2(k_2-k_1)}}{1-a^{-2}}. \label{eqn:aless2301}
\end{align}

By \eqref{eqn:14converse5} we have
\begin{align}
&\frac{1}{2} \log(2 \pi e \mathbb{E}[(X_1+X_2)^2]) \\
&\geq h(W_1'+W_2| y_1'[1:k_1-1],y_2'[1:k_1-1],y_2[k_1:k_2-1])\\
&\geq h(W_1'+W_2| W_1'''', y_1'[1:k_1-1],y_2'[1:k_1-1],y_2[k_1:k_2-1])\\
&= h(W_1'''+W_2| W_1'''', y_1'[1:k_1-1],y_2'[1:k_1-1],y_2[k_1:k_2-1])\\
&= h(W_1'''+W_2| W_1'''', y_1'[1:k_1-1],y_2'[1:k_1-1])\\
& - I(W_1'''+W_2 ; y_2[k_1:k_2-1] | W_1'''', y_1'[1:k_1-1],y_2'[1:k_1-1])\\
&= h(W_1'''+W_2)\\
& - I(W_1'''+W_2 ; y_2[k_1:k_2-1] | W_1'''', y_1'[1:k_1-1],y_2'[1:k_1-1]) \\
&\geq \frac{1}{2} \log( 2 \pi e( \Sigma + a^{2(k-k_1)} \frac{1-a^{-2(k_2-k_1)}}{1-a^{-2}})) \\
& - I(W_1'''+W_2 ; y_2[k_1:k_2-1] | W_1'''', y_1'[1:k_1-1],y_2'[1:k_1-1]) \label{eqn:14convere8}
\end{align}
where the last inequality follows from the fact that $W_1'''$ and $W_2$ are independent Gaussian, and \eqref{eqn:aless2300}, \eqref{eqn:aless2301}.

Now, the question boils down to the upper bound of the last mutual information term, which can be understood as the information contained in the second controller's observation in Witsenhausen's interval.

$\bullet$ Second controller's observation in Witsenhausen's interval: We will bound the amount of information contained in the second controller's observation in Witsenhausen's interval. For $n \geq k_1$, define
\begin{align}
y_2''[n]&:=a^{n-k} W_1''' + a^{n-k_1} w[k_1-1] + a^{n-k_1-1}w[k_1]+ \cdots + w[n-1]\\
&+a^{n-k_1-1}u_1[k_1]+ \cdots + u_1[n-1]\\
&+v_2[n].
\end{align}
Notice that the relationship between $y_2[n]$ and $y_2''[n]$ is
\begin{align}
y_2[n]&=y_2''[n]+a^{n-k_1-1}u_2[k_1]+ \cdots + u_2[n-1]\\
&+a^{n-k}W_1'''' + a^{n-k} \mathbb{E}[W_1 | y_1'[1:k_1-1], y_2'[1:k_1-1]]. \label{eqn:14converse100}
\end{align}

The mutual information of \eqref{eqn:14convere8} is bounded as follows:
\begin{align}
&I(W_1'''+W_2; y_2[k_1:k_2-1] | W_1'''',y_1'[1:k_1-1], y_2'[1:k_1-1])\\
&= h(y_2[k_1:k_2-1]|  W_1'''',y_1'[1:k_1-1], y_2'[1:k_1-1])\\
&- h(y_2[k_1:k_2-1]|  W_1'''+W_2,W_1'''',y_1'[1:k_1-1], y_2'[1:k_1-1]) \\
&= \sum_{k_1 \leq i \leq k_2-1} h(y_2[i]|y_2[k_1:i-1],W_1'''',y_1'[1:k_1-1], y_2'[1:k_1-1])\\
&- \sum_{k_1 \leq i \leq k_2-1} h(y_2[i]|y_2[k_1:i-1],W_1'''+W_2,W_1'''',y_1'[1:k_1-1], y_2'[1:k_1-1]) \\
&\overset{(A)}{=} \sum_{k_1 \leq i \leq k_2-1} h(y_2''[i]|y_2[k_1:i-1],W_1'''',y_1'[1:k_1-1], y_2'[1:k_1-1])\\
&- \sum_{k_1 \leq i \leq k_2-1} h(y_2[i]|y_2[k_1:i-1],W_1'''+W_2,W_1'''',y_1'[1:k_1-1], y_2'[1:k_1-1]) \\
&\overset{(B)}{\leq} \sum_{k_1 \leq i \leq k_2-1}  h(y_2''[i]) - \sum_{k_1 \leq i \leq k_2-1} h(v_2[i])\\
&\leq \sum_{k_1 \leq i \leq k_2-1} \frac{1}{2} \log (2 \pi e \mathbb{E}[y_2''[i]^2])- \sum_{k_1 \leq i \leq k_2-1} \frac{1}{2} \log (2 \pi e \sigma_{v2}^2) \label{eqn:14converse6}
\end{align}
(A): Since $y_2[1:k_1-1]$ is a function of $y_2'[1:k_1-1]$, $u_2[k_1], \cdots, u_2[i]$ are functions of $y_2[k_1:i-1]$, $y_2'[1:k_1-1]$. Thus, all the terms in \eqref{eqn:14converse100} except $y_2''[i]$ can be vanished by the conditioning.\\
(B): By causality, $v_2[i]$ is independent from all conditioning random variables.

First, let's bound the variance of $y_2''[n]$. By \cite[Lemma~1]{Park_Approximation_Journal_Parti}, we have
\begin{align}
\mathbb{E}[y_2''[n]^2] &\leq 
2 \mathbb{E}[(a^{n-k}W_1'''+a^{n-k_1} w[k_1-1] + a^{n-k_1-1}w[k_1]+ \cdots + w[n-1])^2] \\
&+ 2 \mathbb{E}[(a^{n-k_1-1}u_1[k_1]+ \cdots + u_1[n-1])^2] + \sigma_{v2}^2 \\
&= 2 (a^{2(n-k)} \Sigma + a^{2(n-k_1)} + \cdots + 1 )  \\
&+ 2 \mathbb{E}[(a^{n-k_1-1}u_1[k_1]+ \cdots + u_1[n-1])^2] + \sigma_{v2}^2.
\end{align}
Here, by putting $a=a$ and $b=a^{-1}$ to \cite[Lemma~10]{Park_Approximation_Journal_Parti} we have
\begin{align}
&\mathbb{E}[(a^{n-k_1-1}u_1[k_1]+ \cdots + u_1[n-1])^2] \\
&\leq a^{2(n-k_1-1)} \frac{1-a^{-(n-k_1)}}{1-a^{-1}}
(\mathbb{E}[u_1^2[k_1]]+ a^{-1}\mathbb{E}[u_1^2[k_1+1]]+ \cdots + a^{-(n-k_1-1)} \mathbb{E}[u_1^2[n-1]]) \\
&\leq a^{2(n-k_1-1)} \frac{1-a^{-(n-k_1)}}{1-a^{-1}} \frac{1-a^{-(k-k_1)}}{1-a^{-1}} \widetilde{P_1} \\
&=a^{2(n-k_1-1)} \frac{(1-a^{-(n-k_1)})(1-a^{-(k-k_1)})}{(1-a^{-1})^2} \widetilde{P_1}.
\end{align}
Thus, the variance of $y_2''[n]$ is bounded as:
\begin{align}
\mathbb{E}[y_2''[n]^2] \leq 2 a^{2(n-k)} \Sigma + 2 a^{2(n-k_1)} \frac{1-a^{-2(n-k_1+1)}}{1-a^{-2}}
+ 2 a^{2(n-k_1-1)} \frac{(1-a^{-(n-k_1)})(1-a^{-(k-k_1)})}{(1-a^{-1})^2} \widetilde{P_1}+ \sigma_{v2}^2.
\end{align}
Therefore, we have
\begin{align}
&\sum_{k_1 \leq n \leq k_2-1} \mathbb{E}[y_2''[n]^2] \\
&\leq \sum_{k_1 \leq n \leq k_2-1}
2 a^{2(n-k)} \Sigma + 2 a^{2(n-k_1)} \frac{1-a^{-2(n-k_1+1)}}{1-a^{-2}}
+ 2 a^{2(n-k_1-1)} \frac{(1-a^{-(n-k_1)})(1-a^{-(k-k_1)})}{(1-a^{-1})^2} \widetilde{P_1}+ \sigma_{v2}^2 \\
&\leq
2(a^{2(k_1-k)}+ \cdots + a^{2(k_2-1-k)})\Sigma
+\sum_{k_1 \leq n \leq k_2 -1} 2a^{2(n-k_1)} \frac{1-a^{-2(k_2-k_1)}}{1-a^{-2}} \\
&+\sum_{k_1 \leq n \leq k_2 -1} 2a^{2(n-k_1-1)} \frac{(1-a^{-(k_2-1-k_1)})(1-a^{-(k-k_1)})}{(1-a^{-1})^2} \widetilde{P_1}
+(k_2-k_1) \sigma_{v2}^2 \\
&\leq
2a^{2(k_2-1-k)} \frac{1-a^{-2(k_2-k_1)}}{1-a^{-2}} \Sigma
+ 2a^{2(k_2-1-k_1)} \frac{1-a^{-2(k_2-k_1)}}{1-a^{-2}} \frac{1-a^{-2(k_2-k_1)}}{1-a^{-2}} \\
&+ 2a^{2(k_2-k_1-2)} \frac{1-a^{-2(k_2-k_1)}}{1-a^{-2}} \frac{(1-a^{-(k_2-1-k_1)})(1-a^{-(k-k_1)})}{(1-a^{-1})^2} \widetilde{P_1}
+ (k_2 - k_1) \sigma_{v2}^2 \label{eqn:14converse7}
\end{align}
Therefore, by \eqref{eqn:14converse6} and \eqref{eqn:14converse7} we conclude
\begin{align}
&I(W_1'''+W_2; y_2[k_1:k_2-1] | W'''', y_1'[1:k_1-1], y_2'[1:k_1-1])\\
&\leq \sum_{k_1 \leq n \leq k_2-1} \frac{1}{2} \log( \frac{ \mathbb{E}[y_2''[n]^2] }{\sigma_{v2}^2} ) \\
&= \frac{1}{2} \log( \prod_{k_1 \leq n \leq k_2-1} \frac{ \mathbb{E}[y_2''[n]^2] }{\sigma_{v2}^2} ) \\
&\overset{(A)}{\leq} \frac{k_2-k_1}{2} \log( \frac{1}{k_2-k_1} \sum_{k_1 \leq n \leq k_2-1} \frac{ \mathbb{E}[y_2''[n]^2] }{\sigma_{v2}^2} )\\
&\leq \frac{k_2-k_1}{2} \log(
1+\frac{1}{(k_2-k_1)\sigma_{v2}^2}(2a^{2(k_2-1-k)} \frac{1-a^{-2(k_2-k_1)}}{1-a^{-2}} \Sigma
+ 2a^{2(k_2-1-k_1)} \frac{1-a^{-2(k_2-k_1)}}{1-a^{-2}} \frac{1-a^{-2(k_2-k_1)}}{1-a^{-2}} \\
&+ 2a^{2(k_2-k_1-2)} \frac{1-a^{-2(k_2-k_1)}}{1-a^{-2}} \frac{(1-a^{-(k_2-1-k_1)})(1-a^{-(k-k_1)})}{(1-a^{-1})^2} \widetilde{P_1} )
)
\end{align}
(A): Arithmetic-Geometric mean

Denote the last equation as $I'(\widetilde{P_1})$. By \eqref{eqn:14convere8}, we can conclude
\begin{align}
&\frac{1}{2}\log( 2 \pi e \mathbb{E}[(X_1+X_2)^2] \geq \frac{1}{2} \log( 2 \pi e( \Sigma + a^{2(k-k_1)} \frac{1-a^{-2(k_2-k_1)}}{1-a^{-2}})) - I'(\widetilde{P_1})
\end{align}
which implies
\begin{align}
\mathbb{E}[(X_1+X_2)^2] \geq \frac{ \Sigma + a^{2(k-k_1)} \frac{1-a^{-2(k_2-k_1)}}{1-a^{-2}}  }{2^{2I'(\widetilde{P_1})}}. \label{eqn:aless24}
\end{align}

$\bullet$ Final lower bound: Now, we can merge the inequalities to prove the lemma. The variance of $W_3$ is given as follows:
\begin{align}
\mathbb{E}[W_3^2] &= a^{2(k-k_2)} + \cdots + a^2 = a^{2(k-k_2)} \frac{1-a^{-2(k-k_2)}}{1-a^{-2}}. \label{eqn:aless25}
\end{align}
By putting $a=a$ and $b=a^{-1}$ to \cite[Lemma~10]{Park_Approximation_Journal_Parti}, the variance of $U_1$ is bounded as follows:
\begin{align}
\mathbb{E}[U_1^2]
& \leq a^{2(k-k_1-1)} \frac{1-a^{-(k-k_1)}}{1-a^{-1}} (\mathbb{E}[u_1^2[k_1]]+a^{-1}\mathbb{E}[u_1^2[k_1+1]]+\cdots+a^{-(k-k_1-1)}\mathbb{E}[u_1^2[k-1]]) \\
&= a^{2(k-k_1-1)} \frac{(1-a^{-(k-k_1)} )^2}{(1-a^{-1})^2} \widetilde{P_1}. \label{eqn:aless26}
\end{align}
Likewise, the variance of $U_2$ can be bounded as
\begin{align}
\mathbb{E}[U_2^2] \leq a^{2(k-k_2-1)} \frac{(1-a^{-(k-k_2)} )^2}{(1-a^{-1})^2} \widetilde{P_2}. \label{eqn:aless27}
\end{align}
Finally, by plugging \eqref{eqn:aless24}, \eqref{eqn:aless25}, \eqref{eqn:aless26}, \eqref{eqn:aless27} into \eqref{eqn:aless28}, we prove the lemma.
\end{proof}

\begin{corollary}
Consider the decentralized LQG problem of Problem~\ref{prob:a}. Define
\begin{align}
&\Sigma_1 := \frac{(a^2-1)\sigma_{v1}^2 -1  + \sqrt{((a^2-1)\sigma_{v1}^2 - 1)^2 + 4 a^2 \sigma_{v1}^2}}{2a^2} \label{eqn:cor3stat:1}\\
&\Sigma_2 := \frac{(a^2-1)\sigma_{v2}^2 -1  + \sqrt{((a^2-1)\sigma_{v2}^2 - 1)^2 + 4 a^2 \sigma_{v2}^2}}{2a^2}. \label{eqn:cor3stat:2}
\end{align}
Let $1 < |a| \leq 2.5$. Then, for all $q, r_1, r_2, \sigma_0, \sigma_{v1}, \sigma_{v2} >0$, the minimum cost \eqref{eqn:part11} of Problem~\ref{prob:a} is lower bounded as follows:
\begin{align}
&\inf_{u_1, u_2} \limsup_{N \rightarrow \infty} \frac{1}{N} \sum_{0 \leq n < N}
q \mathbb{E}[x^2[n]] + r_1 \mathbb{E}[u_1^2[n]] + r_2 \mathbb{E}[u_2^2[n]] \\
&\geq \min_{\widetilde{P_1}, \widetilde{P_2} \geq 0} q D_L(\widetilde{P_1},\widetilde{P_2}) + r_1 \widetilde{P_1} + r_2 \widetilde{P_2}
\end{align}
where $D_L(\widetilde{P_1},\widetilde{P_2})$ satisfies the following conditions.

(a) If $\Sigma_1 \geq 150$, $\Sigma_2 \geq 150$, $\widetilde{P_1} \leq \frac{(a^2-1)^2 \Sigma_1}{40000}$, $\widetilde{P_2} \leq \frac{(a^2-1)^2 \Sigma_2}{40000}$ then $D_L(\widetilde{P_1},\widetilde{P_2})= \infty$.

(b) If $\Sigma_1 \geq 150$, $\Sigma_2 \geq 150$, $\widetilde{P_1} \leq \frac{(a^2-1)^2 \Sigma_1}{40000}$ then $D_L(\widetilde{P_1},\widetilde{P_2}) \geq 0.002774 \Sigma_2 + 1$.

(c) If $\widetilde{P_1} \leq \frac{1}{20}(a^2-1)$, $\widetilde{P_2} \leq \frac{1}{20}(a^2-1)$ then $D_L(\widetilde{P_1},\widetilde{P_2})=\infty$.

(d) If $\widetilde{P_1} \leq \frac{1}{75}$ and $\widetilde{P_2} \leq \frac{1}{75}$ then $D_L(\widetilde{P_1},\widetilde{P_2}) \geq 0.00389 \frac{1}{\max(P_1,P_2)} + 1$.

(e) If $\Sigma_2 \geq 150$, $\widetilde{P_1} \leq \frac{1}{\Sigma_2}$ then $D_L(\widetilde{P_1},\widetilde{P_2}) \geq 0.0006976\Sigma_2 + 1$.

(f) If $\Sigma_2 \geq 150$, $\frac{1}{\Sigma_2} \leq \widetilde{P_1} \leq \frac{1}{150}$ then  $D_L(\widetilde{P_1},\widetilde{P_2}) \geq \frac{0.0006976}{\widetilde{P_1}} +1 $.

(g) If $\Sigma_2 \geq 150$, $\widetilde{P_1} \leq \frac{1}{20}(a^2-1)$, $\widetilde{P_2} \leq \frac{(a^2-1)^2 \Sigma_2}{40000}$  then $D_L(\widetilde{P_1},\widetilde{P_2})=\infty$.

(h) If $\Sigma_1 \geq 150$, $\widetilde{P_1} \leq \frac{(a^2-1)^2 \Sigma_1}{40000}$, $\widetilde{P_2} \leq \frac{1}{20}(a^2-1)$ then $D_L(\widetilde{P_1},\widetilde{P_2})=\infty$.

(i) If $\Sigma_2 \geq 150$, $\widetilde{P_1} \leq \frac{1}{20}(a^2-1)$ then $D_L(\widetilde{P_1},\widetilde{P_2}) \geq 0.0002732 \Sigma_2 + 1$.

(j) For all $\widetilde{P_1}$ and $\widetilde{P_2}$, $D_L(\widetilde{P_1},\widetilde{P_2}) \geq \max( 0.1035 \Sigma_1 ,1)$.
\label{cor:2}
\end{corollary}
\begin{proof}
See Appendix~\ref{sec:ageq1} for the proof.
\end{proof}

In this corollary, $\Sigma_1$ and $\Sigma_2$ are the Kalman filtering performance of the first and second controllers respectively.

Now, we have lower bounds on the average decentralized control cost of Problem~\ref{prob:a}. Furthermore, by inspecting the form of the lower bounds, the term $D_L(\widetilde{P_1}, \widetilde{P_2})$ can be speculated as a lower bound on the power-distortion tradeoff $D(P_1, P_2)$ of Problem~\ref{prob:c}. 

Furthermore, \cite[Lemma~14]{Park_Approximation_Journal_Parti} shows the average cost problem in Problem~\ref{prob:a} and the power-distortion tradeoff problem in Problem~\ref{prob:c} are closely related, i.e. if we can characterize the power-distortion tradeoff within a constant ratio, then we can characterize the average cost within a constant ratio.
Therefore, in the following discussion, we will focus on the power-distortion tradeoff and justify that why it can be characterized within a constant. Throughout the discussion, we will consider $D_L(\widetilde{P_1}, \widetilde{P_2})$ as if it is a lower bound on $D(P_1, P_2)$ without rigorous justification. 

By comparing the achievable cost shown in Corollary~\ref{cor:1}, we will prove that they are within a constant ratio. In other words, we will prove the power-distortion tradeoff $(D(P_1, P_2), P_1, P_2)$ is essentially the better performance between two single controllers, i.e. $(\min(D_{\sigma_1}(P_1), D_{\sigma_2}(P_2)), P_1, P_2)$.

To justify this, we will divide the cases. As discussed in Section~\ref{sec:central}, the centralized controller's performances behave qualitatively differently depending on $\max(\Sigma_1, 1)$, $\max(\Sigma_2, 1)$, $\frac{1}{a^2-1}$. Therefore, we will divide into three cases\footnote{Since $\sigma_{v1} \leq \sigma_{v2}$, $\Sigma_1$ is always smaller than $\Sigma_2$.} depending on these values. Then, we will further divide the cases by $P_1$ and $P_2$.

\subsubsection{When $\max(\Sigma_1, 1) \leq \max(\Sigma_2, 1) \leq \Theta(\frac{1}{a^2-1})$} We will again divide the cases based on $P_1, P_2$.

$\bullet$ When $P_1 \leq \Theta(a^2-1)$ and $P_2 \leq \Theta(a^2-1)$. As we can see from Figure~\ref{fig:aless21}, each controller does not have enough power to stabilize the system. The statement (c) of Corollary~\ref{cor:2} tells that the system is unstable even in decentralized control problems.

$\bullet$ When $P_1 \leq \Theta(a^2-1)$ and $\Theta(a^2-1) \leq P_2 \leq \Theta(\frac{1}{\max(\Sigma_2, 1)})$. In this case, the control performance is determined by the second controller. From Figure~\ref{fig:aless21} we can see that $D(P_1, P_2) = O(\frac{1}{P_2})$ is achievable. The statement (d) of Corollary~\ref{cor:2} tells it is tight up to a constant ratio.

$\bullet$ When $P_1 \leq \Theta(a^2-1)$ and $\Theta(\frac{1}{\max(\Sigma_2, 1)}) \leq P_2$. Like above the second controller dominates the performance, and Figure~\ref{fig:aless21} shows $D(P_1, P_2) = O(\max(\Sigma_2, 1))$. The statement (i) of Corollary~\ref{cor:2} shows its tightness.

%
%

$\bullet$ When $\Theta(a^2-1) \leq P_1 \leq \Theta(\frac{1}{\max(\Sigma_2, 1)})$ and $ P_2 \leq \Theta(\frac{1}{\max(\Sigma_2, 1)})$. In this case, the control performance is determined by the controller with a larger power, and Figure~\ref{fig:aless21} shows $D(P_1, P_2) = O(\frac{1}{\max(P_1, P_2)})$ is achievable. The statement (d) of Corollary~\ref{cor:2} gives a matching lower bound.

$\bullet$ When $\Theta(a^2-1) \leq P_1 \leq \Theta(\frac{1}{\max(\Sigma_2, 1)})$ and $\Theta(\frac{1}{\max(\Sigma_2, 1)}) \leq P_2$. In this case, the second controller dominates the performance, and Figure~\ref{fig:aless21} shows $D(P_1, P_2) = O(\Sigma_2, 1)$ is achievable. The statement (e) of Corollary~\ref{cor:2} gives a matching lower bound.

$\bullet$ When $\Theta(\frac{1}{\max(\Sigma_2, 1)}) \leq P_1 \leq \Theta(\frac{1}{\max(\Sigma_1, 1)})$. In this case, the first controller dominates the performance, and Figure~\ref{fig:aless21} shows $D(P_1, P_2) = O(\frac{1}{P_1})$ is achievable. The statement (f) of Corollary~\ref{cor:2} gives a matching lower bound.

$\bullet$ When $\Theta(\frac{1}{\max(\Sigma_1, 1)}) \leq P_1$. In this case, the first controller dominates the performance, and Figure~\ref{fig:aless21} shows $D(P_1, P_2) = O(\max(\Sigma_1, 1))$ is achievable. The statement (j) of Corollary~\ref{cor:2} gives a matching lower bound.

\subsubsection{When $\max(\Sigma_1, 1) \leq \Theta(\frac{1}{a^2-1}) \leq \max(\Sigma_2, 1) $} We will further divide the cases based on $P_1, P_2$.

$\bullet$ When $P_1 \leq \Theta(a^2-1)$ and $P_2 \leq \Theta((a^2-1)^2 \max(\Sigma_2,1))$.
As we can see from Figure~\ref{fig:aless21}, each controller does not have enough power to stabilize the system by itself. The statement (g) of Corollary~\ref{cor:2} shows that the system is unstable indeed for decentralized control problems.

$\bullet$ When $P_1 \leq \Theta(a^2-1)$ and $\Theta((a^2-1)^2 \max(\Sigma_2,1)) \leq P_2$. 
In this case, the second controller dominates the performance, and Figure~\ref{fig:aless22} shows $D(P_1, P_2)=O(\max(\Sigma_2, 1))$. The statement (i) of Corollary~\ref{cor:2} give a matching lower bound up to a constant ratio.

$\bullet$ When $\Theta(a^2-1) \leq P_1 \leq \Theta(\frac{1}{\max(\Sigma_2, 1)})$. Since we assume $\Theta(\frac{1}{a^2-1})\leq \max(\Sigma_2, 1)$, this case never happens.

$\bullet$ When $\Theta(\frac{1}{\max(\Sigma_2, 1)}) \leq P_1 \leq \Theta(\frac{1}{\max(\Sigma_1, 1)})$.
In this case, the first controller dominates the performance, and Figure~\ref{fig:aless21} shows $D(P_1, P_2)=O(\frac{1}{P_1})$ is achievable. The statement (f) of Corollary~\ref{cor:2} gives a matching lower bound.

$\bullet$ When $\Theta(\frac{1}{\max(\Sigma_1, 1)}) \leq P_1$. The first controller dominates the performance, but as we can see in Figure~\ref{fig:aless21} its performance is saturated by the Kalman filtering and $D(P_1, P_2)=O(\max(\Sigma_1, 1))$. The statement (j) of Corollary~\ref{cor:2} gives a matching lower bound.

\subsubsection{When $\Theta(\frac{1}{a^2-1}) \leq \max(\Sigma_1, 1) \leq \max(\Sigma_2, 1)$} We will divide the cases based on $P_1, P_2$.

$\bullet$ When $P_1 \leq \Theta((a^2-1)^2 \max(\Sigma_1, 1))$ and $P_2 \leq \Theta((a^2-1)^2 \max(\Sigma_1, 1))$. In this case, as shown in Figure~\ref{fig:aless22} each controller cannot stabilize the system by itself. The statement (a) of Corollary~\ref{cor:2} shows that the decentralized system is indeed unstable.

$\bullet$ When $P_1 \leq \Theta((a^2-1)^2 \max(\Sigma_1, 1))$ and $\Theta((a^2-1)^2 \max(\Sigma_1, 1)) \leq P_2$. In this case, the second controller dominates the performance, and Figure~\ref{fig:aless22} shows $D(P_1, P_2) \leq O(\frac{1}{\max(\Sigma_2, 1)})$ is achievable. The statement (b) of Corollary~\ref{cor:2} gives a matching lower bound.

$\bullet$ When $\Theta((a^2-1)^2 \max(\Sigma_1, 1)) \leq P_1$. In this case, the first controller dominates the performance, and Figure~\ref{fig:aless22} shows $D(P_1, P_2) \leq O(\max(\Sigma_1, 1))$ is achievable. The statement (j) of Corollary~\ref{cor:2} gives a matching lower bound.

Formally, the average cost can be characterized within a constant ratio as follows.
\begin{proposition}
Consider the decentralized LQG control of Problem~\ref{prob:a}. There exists $c \leq 2 \times 10^6$ such that for all $1 < |a| \leq 2.5$, $q$, $r_1$, $r_2$, $\sigma_{v1}$ and $\sigma_{v2}$,
\begin{align}
\frac{
\underset{u_1,u_2 \in L_{lin,kal}}{\inf}
\underset{N \rightarrow \infty}{\limsup}
\frac{1}{N}
\underset{0 \leq n < N}{\sum}
\mathbb{E}[qx^2[n]+r_1 u_1^2[n]+r_2 u_2^2[n]]
}
{
\underset{u_1,u_2}{\inf}
\underset{N \rightarrow \infty}{\limsup}
\frac{1}{N}
\underset{0 \leq n < N}{\sum}
\mathbb{E}[qx^2[n]+r_1 u_1^2[n]+r_2 u_2^2[n]]
}
\leq c. \nonumber
\end{align}
\label{prop:1}
\end{proposition}
\begin{proof}
See Appendix~\ref{sec:ageq1}.
\end{proof}

\subsection{When $|a| = 1$}
\label{sec:subaeq1}

In this case, we can prove the following lemmas which parallel Lemma~\ref{lem:aless22} and Corollary~\ref{cor:2} of the case when $1 < |a| \leq 2.5$.

\begin{lemma}
We use the definition of $S_L$ shown in Lemma~\ref{lem:aless22}, i.e. the set of $(k_1, k_2, k)$ such that $k_1,k_2,k \in \mathbb{N}$ and $1 \leq k_1 \leq k_2 \leq k$. We define $D_{L,2}(\widetilde{P_1},\widetilde{P_2},k_1,k_2,k)$ as follows:
\begin{align}
D_{L,2}(\widetilde{P_1},\widetilde{P_2},k_1,k_2,k) \geq
(
\sqrt{
\frac{ \Sigma + k_2-k_1  }{2^{2I'(P_1)}}
+k-k_2
}
-
\sqrt{(k-k_1)^2 \widetilde{P_1}}
-
\sqrt{(k-k_2)^2 \widetilde{P_2}}
)_+^2 + 1
\end{align}
where
\begin{align}
\Sigma&=\frac{ k_1-1 }{2^{2I}}\\
I&=\frac{k_1-1}{2} \log( 1+ \frac{k_1-1}{\sigma_{v1}^2})+\frac{k_1-1}{2} \log( 1+ \frac{k_1-1}{\sigma_{v2}^2})\\
I'(\widetilde{P_1})&= \frac{k_2-k_1}{2} \log(1+\frac{1}{\sigma_{v2}^2}(2 \Sigma + 2(k_2-k_1) + 2(k_2-k_1-1)(k-k_1) \widetilde{P_1})).
\end{align}
Here, when $k_1-1=0$, $I=0$ and when $k_2-k_1=0$, $I'(\widetilde{P_1})=0$.

Let $|a|=1$. Then, for all $q, r_1, r_2, \sigma_0, \sigma_{v1}, \sigma_{v2} \geq 0$, the minimum cost \eqref{eqn:part11} of Problem~\ref{prob:a} is lower bounded as follows:
\begin{align}
&\inf_{u_1, u_2} \limsup_{N \rightarrow \infty} \frac{1}{N} \sum_{0 \leq n < N}
q \mathbb{E}[x^2[n]] + r_1 \mathbb{E}[u_1^2[n]] + r_2 \mathbb{E}[u_2^2[n]] \\
&\geq  \sup_{(k_1, k_2, k) \in S_L} \min_{\widetilde{P_1}, \widetilde{P_2} \geq 0}
q D_{L,2}(\widetilde{P_1}, \widetilde{P_2}; k_1 , k_2, k) + r_1 \widetilde{P_1} + r_2 \widetilde{P_2}.
\end{align}
\label{lem:eq11}
\end{lemma}
\begin{proof}
See Appendix~\ref{sec:aeq1}.
\end{proof}

\begin{corollary}
Consider the decentralized LQG problem of Problem~\ref{prob:a}. 
Let $|a|=1$. Then, for all $q, r_1, r_2 >0$, the minimum cost \eqref{eqn:part11} of Problem~\ref{prob:a} is lower bounded as follows:
\begin{align}
&\inf_{u_1, u_2} \limsup_{N \rightarrow \infty} \frac{1}{N} \sum_{0 \leq n < N}
q \mathbb{E}[x^2[n]] + r_1 \mathbb{E}[u_1^2[n]] + r_2 \mathbb{E}[u_2^2[n]] \\
&\geq \min_{\widetilde{P_1}, \widetilde{P_2} \geq 0} q D_L(\widetilde{P_1},\widetilde{P_2}) + r_1 \widetilde{P_1} + r_2 \widetilde{P_2}
\end{align}
where $D_L(\widetilde{P_1},\widetilde{P_2})$ satisfies the following conditions.\\
(a) If $\sigma_{v2} \geq 16$ and $\widetilde{P_1} \leq \frac{1}{4\sigma_{v2}}$ then $D_L(\widetilde{P_1}, \widetilde{P_2}) \geq  0.09168\sigma_{v2} + 1$.\\
(b) If $\sigma_{v2} \geq 16$ and $\frac{1}{4\sigma_{v2}} \leq \widetilde{P_1} \leq \frac{1}{64}$ then $D_L(\widetilde{P_1}, \widetilde{P_2}) \geq  \frac{0.02417}{\widetilde{P_1}} + 1$.\\
(c) If $\widetilde{P_1} \leq \frac{1}{50}$, $\widetilde{P_2} \leq \frac{1}{50}$ then $D_L(\widetilde{P_1}, \widetilde{P_2}) \geq \frac{0.003772}{\max(\widetilde{P_1}, \widetilde{P_2})}+1$. \\
(d) For all $\widetilde{P_1}, \widetilde{P_2}$, $D_L(\widetilde{P_1}, \widetilde{P_2}) \geq \max(\frac{\sqrt{2}}{2}\sigma_{v1},1)$.
\label{cor:4}
\end{corollary}
\begin{proof}
See Appendix~\ref{sec:aeq1}.
\end{proof}

Like Section~\ref{sec:aless2}, we will intuitively argue why the power-distortion tradeoff can be characterized within a constant ratio by considering $D_L(\widetilde{P_1}, \widetilde{P_2})$ as if it is a lower bound on $D(P_1, P_2)$.


Notice that by \eqref{eqn:kalmanperf}, when $|a|=1$ the Kalman filtering performance of the controllers are given as $\Sigma_1 = \frac{-1 + \sqrt{1+4 \sigma_v^2}}{2}$ and $\Sigma_2 = \frac{-1 + \sqrt{1+4 \sigma_v^2}}{2}$ respectively. Therefore, we can see $\Sigma_1 \approx \sigma_1$ and $\Sigma_2 \approx \sigma_1$ and so we can think of $\sigma_1$, $\sigma_2$ shown in Corollary~\ref{cor:4} as if they are $\Sigma_1$, $\Sigma_2$.

As we discussed in Section~\ref{subsec:a=1}, when $|a|=1$ there are only one case for the power-distortion tradeoff. Thus, we will only divide the cases by $P_1$ and $P_2$.

$\bullet$ When $P_1 \leq \Theta(\frac{1}{\max(\Sigma_2, 1)})$ and $P_2 \leq \Theta(\frac{1}{\max(\Sigma_2, 1)})$. The controller with a larger power dominates the performance, and Figure~\ref{fig:aeq11} shows $D(P_1, P_2) = O(\frac{1}{\max(P_1, P_2)})$ is achievable. The statement (c) of Corollary~\ref{cor:4} gives a matching lower bound.

$\bullet$ When $P_1 \leq \Theta(\frac{1}{\max(\Sigma_2, 1)})$ and $\Theta(\frac{1}{\max(\Sigma_2, 1)}) \leq P_2$. In this case, the second controller dominates the performance, but its performance is saturated by the Kalman filtering. Figure~\ref{fig:aeq11} shows $D(P_1, P_2) = O(\max(\Sigma_2, 1))$. The statement (a) of Corollary~\ref{cor:4} gives a matching lower bound.

$\bullet$ When $\Theta(\frac{1}{\max(\Sigma_2, 1)}) \leq P_1 \leq \Theta(\frac{1}{\max(\Sigma_1, 1)})$. In this case, the first controller dominates the performance, and Figure~\ref{fig:aeq11} shows $D(P_1, P_2)=O(\frac{1}{P_1})$. The statement (b) of Corollary~\ref{cor:4} gives a matching lower bound.

$\bullet$ When $P_1 \geq \Theta(\frac{1}{\max(\Sigma_1, 1)})$. 
In this case, the first controller dominates the performance, but its performance is saturated by the Kalman filtering. Figure~\ref{fig:aeq11} shows $D(P_1, P_2)=O(\max(\Sigma_1, 1))$ is achievable. The statement (d) of Corollary~\ref{cor:4} gives a matching lower bound.

Formally, the constant ratio result for the average cost LQG problems can be written as follows.
\begin{proposition}
Consider the decentralized LQG control of Problem~\ref{prob:a}. There exists $c \leq 540$ such that for all $|a|=1$, $q$, $r_1$, $r_2$, $\sigma_{v1}$ and $\sigma_{v2}$,
\begin{align}
\frac{
\underset{u_1,u_2 \in L_{lin,kal}}{\inf}
\underset{N \rightarrow \infty}{\limsup}
\frac{1}{N}
\underset{0 \leq n < N}{\sum}
\mathbb{E}[qx^2[n]+r_1 u_1^2[n]+r_2 u_2^2[n]]
}
{
\underset{u_1,u_2}{\inf}
\underset{N \rightarrow \infty}{\limsup}
\frac{1}{N}
\underset{0 \leq n < N}{\sum}
\mathbb{E}[qx^2[n]+r_1 u_1^2[n]+r_2 u_2^2[n]]
}
\leq c. \nonumber
\end{align}
\label{prop:2}
\end{proposition}
\begin{proof}
See Appendix~\ref{sec:aeq1}.
\end{proof}

\subsection{When $0.9 \leq |a| < 1$}
\label{sec:subless1}
\begin{lemma}
We use the definition of $S_L$ shown in Lemma~\ref{lem:aless22}, i.e. the set of $(k_1, k_2, k)$ such that $k_1,k_2,k \in \mathbb{N}$ and $1 \leq k_1 \leq k_2 \leq k$. We define $D_{L,3}(\widetilde{P_1},\widetilde{P_2},k_1,k_2,k)$ as follows:
\begin{align}
D_{L,3}(\widetilde{P_1},\widetilde{P_2}) &:=
(
\sqrt{
\frac{ \Sigma + a^{2(k-k_2+1)}\frac{1-a^{2(k_2-k_1)}}{1-a^2}  }{2^{2I'(\widetilde{P_1})}}
+a^2 \frac{1-a^{2(k-k_2)}}{1-a^2}
}
-
\sqrt{(\frac{1-a^{k-k_1}}{1-a} )^2 \widetilde{P_1}}
-
\sqrt{(\frac{1-a^{k-k_2}}{1-a} )^2 \widetilde{P_2}}
)_+^2 + 1
\end{align}
where
\begin{align}
\Sigma&=\frac{ a^{2(k-k_1+1)}\frac{1-a^{2(k_1-1)}}{1-a^2} }{2^{2I}}\\
I&=\frac{1}{2} \log( 1+
\frac{1}{\sigma_{v1}^2} \frac{1-a^{2(k_1-1)}}{1-a^2})^{k_1-1} +
\frac{1}{2} \log( 1+
\frac{1}{\sigma_{v2}^2} \frac{1-a^{2(k_1-1)}}{1-a^2})^{k_1-1}\\
I'(\widetilde{P_1})&= \frac{1}{2} \log(
1+\frac{1}{(k_2-k_1)\sigma_{v2}^2}(2a^{2(k_1-k)} \frac{1-a^{2(k_2-k_1)}}{1-a^{2}} \Sigma
+ 2 (k_2-k_1)\frac{1-a^{2(k_2-1-k_1+1)}}{1-a^2} \\
&+ 2 a^{k_1-k} \frac{1-a^{k_2-k_1}}{1-a}\frac{(1-a^{k_2-1-k_1})(1-a^{k-k_1})}{(1-a)^2} \widetilde{P_1} )
)^{k_2-k_1}
\end{align}
Here, when $k_1-1=0$, $I=0$ and when $k_2-k_1=0$, $I'(\widetilde{P_1})=0$.

Let $0 \leq |a| < 1$. Then, for all $q, r_1, r_2, \sigma_{0}, \sigma_{v1}, \sigma_{v2} \geq 0$, the minimum cost \eqref{eqn:part11} of Problem~\ref{prob:a} is lower bounded as follows:
\begin{align}
&\inf_{u_1, u_2} \limsup_{N \rightarrow \infty} \frac{1}{N} \sum_{0 \leq n < N}
q \mathbb{E}[x^2[n]] + r_1 \mathbb{E}[u_1^2[n]] + r_2 \mathbb{E}[u_2^2[n]] \\
&\geq  \sup_{(k_1, k_2, k) \in S_L} \min_{\widetilde{P_1}, \widetilde{P_2} \geq 0}
q D_{L,3}(\widetilde{P_1}, \widetilde{P_2}; k_1 , k_2, k) + r_1 \widetilde{P_1} + r_2 \widetilde{P_2}.
\end{align}
\label{lem:less11}
\end{lemma}
\begin{proof}
See Appendix~\ref{sec:aless1} for the proof.
\end{proof}

\begin{corollary}
Consider the decentralized LQG problem of Problem~\ref{prob:a}. Define
\begin{align}
&\Sigma_1 := \frac{(a^2-1)\sigma_{v1}^2 -1  + \sqrt{((a^2-1)\sigma_{v1}^2 - 1)^2 + 4 a^2 \sigma_{v1}^2}}{2a^2} \\
&\Sigma_2 := \frac{(a^2-1)\sigma_{v2}^2 -1  + \sqrt{((a^2-1)\sigma_{v2}^2 - 1)^2 + 4 a^2 \sigma_{v2}^2}}{2a^2}.
\end{align}
Let $0.9 \leq |a| < 1$. Then, for all $q, r_1, r_2 > 0$, the minimum cost \eqref{eqn:part11} of Problem~\ref{prob:a} is lower bounded as follows:
\begin{align}
&\inf_{u_1, u_2} \limsup_{N \rightarrow \infty} \frac{1}{N} \sum_{0 \leq n < N}
q \mathbb{E}[x^2[n]] + r_1 \mathbb{E}[u_1^2[n]] + r_2 \mathbb{E}[u_2^2[n]] \\
&\geq \min_{\widetilde{P_1}, \widetilde{P_2} \geq 0} q D_L(\widetilde{P_1},\widetilde{P_2}) + r_1 \widetilde{P_1} + r_2 \widetilde{P_2}
\end{align}
where $D_L(\widetilde{P_1}, \widetilde{P_2})$ satisfies the following conditions.

Then, we have a lower bound $D_L(\widetilde{P_1}, \widetilde{P_2})$ on $D(P_1, P_2)$ where $D_L(\widetilde{P_1}, \widetilde{P_2})$ satisfies the followings:\\
(a) If $\Sigma_2 \geq 40$, $\widetilde{P_1} \leq \frac{1}{\Sigma_2}$ then $D_L(\widetilde{P_1}, \widetilde{P_2}) \geq 0.009131 \Sigma_2 +1$.\\
(b) If $\Sigma_2 \geq 40$, $\frac{1}{\Sigma_2} \leq \widetilde{P_1} \leq \frac{1}{40}$ then $D_L(\widetilde{P_1}, \widetilde{P_2}) \geq \frac{0.009131}{\widetilde{P_1}}+1$.\\
(c) If $\frac{1-a^2}{20} \leq \max(\widetilde{P_1}, \widetilde{P_2}) \leq \frac{1}{40}$ then $D_L(P_1, P_2)\geq \frac{0.001201}{\max(\widetilde{P_1}, \widetilde{P_2})}+1$.\\
(d) If $\max(\widetilde{P_1}, \widetilde{P_2}) \leq \frac{1-a^2}{20}$ then $D_L(\widetilde{P_1}, \widetilde{P_2}) \geq \frac{0.0869}{1-a^2}+1$.\\
(e) For all $\widetilde{P_1}, \widetilde{P_2}$, $D_L(\widetilde{P_1}, \widetilde{P_2}) \geq \max(0.2636\Sigma_1,1)$.
\label{cor:6}
\end{corollary}
\begin{proof}
See Appendix~\ref{sec:aless1} for the proof.
\end{proof}

Like Section~\ref{sec:aless2}, we will intuitively argue why the power-distortion tradeoff can be characterized within a constant ratio by considering $D_L(\widetilde{P_1}, \widetilde{P_2})$ as if it is a lower bound on $D(P_1, P_2)$. The characterization of the power-distortion tradeoff is equivalent to the characterization of the average cost. Thus, we will intuitively argue how we can characterize the power-distortion tradeoff within a constant ratio. For this, we will first divide the cases by $\Sigma_1, \Sigma_2$, then we will further divide the cases by $P_1, P_2$.

\subsubsection{When $\max(\Sigma_1, 1) \leq \max(\Sigma_2, 1) \leq \Theta(\frac{1}{1-a^2})$} We will further divide the cases based on $\widetilde{P_1}, \widetilde{P_2}$.

$\bullet$ When $P_1 \leq \Theta(1-a^2)$ and $P_2 \leq \Theta(1-a^2)$. 
From Figure~\ref{fig:aless11} we can see that $D(P_1, P_2)= \frac{1}{1-a^2}$ is achievable without any control input.
The statement (d) of Corollary~\ref{cor:6} gives a lower bound tight up to a constant ratio.

$\bullet$ When $P_1 \leq \Theta(1-a^2)$ and $\Theta(1-a^2) \leq P_2 \leq \Theta(\frac{1}{\max(\Sigma_2,1)})$. The second controller dominates the control performance. 
Figure~\ref{fig:aless11} shows $D(P_1, P_2)=O(\frac{1}{P_2})$ is achievable in this case. 
The statement (c) of Corollary~\ref{cor:6} gives a lower bound tight up to a constant ratio.

$\bullet$ When $P_1 \leq \Theta(1-a^2)$ and $\Theta(\frac{1}{\max(\Sigma_2,1)}) \leq P_2$. 
In this case, the second controller's performance is saturated by the Kalman filtering, and Figure~\ref{fig:aless11} shows $D(P_1, P_2)= O(\max(\Sigma_2, 1))$ is achievable. The statement (a) of Corollary~\ref{cor:6} gives a lower  bound tight up to a constant ratio.

$\bullet$ When $\Theta(1-a^2) \leq P_1 \leq \Theta(\frac{1}{\max(\Sigma_2,1)})$ and $P_2 \leq \Theta(1-a^2)$. The first controller dominates the control performance. Figure~\ref{fig:aless11} shows $D(P_1, P_2)=O(\frac{1}{P_1})$ is achievable in this case. The statement (c) of Corollary~\ref{cor:6} gives a lower bound tight up to a constant ratio.

$\bullet$ When $\Theta(1-a^2) \leq P_1 \leq \Theta(\frac{1}{\max(\Sigma_2,1)})$ and $\Theta(1-a^2) \leq P_2 \leq \Theta(\frac{1}{\max(\Sigma_2,1)})$. The controller with larger power dominates the control performance. Figure~\ref{fig:aless11} shows $D(P_1, P_2)= O(\frac{1}{\max(P_1, P_2)})$ is achievable with the controller with a lager power. The statement (c) of Corollary~\ref{cor:6} gives a lower bound tight up to a constant ratio.

$\bullet$ When $\Theta(1-a^2) \leq P_1 \leq \Theta(\frac{1}{\max(\Sigma_2,1)})$ and $\Theta(\frac{1}{\max(\Sigma_2,1)}) \leq P_2$. The second controller dominates the control performance. Figure~\ref{fig:aless11} shows $D(P_1, P_2) = O(\max(\Sigma_2, 1))$ is achievable. The statement (a) of Corollary~\ref{cor:6} gives a lower bound tight up to a constant ratio.

$\bullet$ When $\Theta(\frac{1}{\max(\Sigma_2,1)}) \leq P_1 \leq \Theta(\frac{1}{\max(\Sigma_1,1)})$. 
The first controller dominates the performance. Figure~\ref{fig:aless11} shows $D(P_1, P_2)=O(\frac{1}{P_1})$ is achievable. The statement (b) of Corollary~\ref{cor:6} gives a lower bound tight up to a constant ratio.

$\bullet$ When $\Theta(\frac{1}{\max(\Sigma_1,1)}) \leq P_1$. The first controller dominates the performance. Figure~\ref{fig:aless11} shows $D(P_1, P_2)= O(\max(\Sigma_1, 1))$ is achievable. 
The statement (e) of Corollary~\ref{cor:6} gives a lower bound tight up to a constant ratio.


\subsubsection{When $\max(\Sigma_1, 1) \leq \Theta(\frac{1}{1-a^2}) = \max(\Sigma_2, 1)$} We will further divide the cases based on $P_1, P_2$.

$\bullet$ When $P_1 \leq \Theta(1-a^2)$. From Figure~\ref{fig:aless11} we can see that $D(P_1, P_2)= \frac{1}{1-a^2}$ is achievable without any control input. The statement (b) of Corollary~\ref{cor:6} give a matching lower bound. More precisely, since $\Theta(\frac{1}{1-a^2}) = \max(\Sigma_2, 1)$, for a large value of $\Sigma_2$ we can put $P_1 = \Theta(1-a^2)$ in the statement (b). Then, the bound reduces to $D(P_1, P_2) = \Omega(\frac{1}{1-a^2})$.

$\bullet$ When $\Theta(1-a^2) \leq P_1 \leq \Theta(\frac{1}{\max(\Sigma_1,1)})$. 
In this case, the first controller dominates the performance, and Figure~\ref{fig:aless11} shows $D(P_1, P_2) = O(\frac{1}{P_1})$ is achievable. The statement (b) of Corollary~\ref{cor:6} gives a matching lower bound.

$\bullet$ When $\Theta(\frac{1}{\max(\Sigma_1,1)}) \leq P_1$. 
In this case, the first controller dominates the performance, and Figure~\ref{fig:aless11} shows $D(P_1, P_2) = O(\max(\Sigma_1, 1))$ is achievable. The statement (e) of Corollary~\ref{cor:6} gives a matching lower bound. 

\subsubsection{When $\Theta(\frac{1}{1-a^2}) = \max(\Sigma_1, 1) \approx \max(\Sigma_2, 1)$}
In this case, the Kalman filtering noise $\Sigma_1$ and $\Sigma_2$ is already compatible with $\frac{1}{1-a^2}$, the state distortion attainable without any control inputs. Therefore, we cannot expect a significant control gain, and the optimal state distortion is $\Theta(\frac{1}{1-a^2})$. Since $\Theta(\frac{1}{1-a^2}) = \max(\Sigma_1, 1)$, the statement (e) of Corollary~\ref{cor:6} gives a matching lower bound.

Formally, the average LQG cost for $0.9 \leq |a| < 1$ can be characterized as follows.
\begin{proposition}
Consider the decentralized LQG control of Problem~\ref{prob:a}. There exists $c \leq 1700$ such that for all $0.9 \geq |a|<1$, $q$, $r_1$, $r_2$, $\sigma_{v1}$ and $\sigma_{v2}$,
\begin{align}
\frac{
\underset{u_1,u_2 \in L_{lin,kal}}{\inf}
\underset{N \rightarrow \infty}{\limsup}
\frac{1}{N}
\underset{0 \leq n < N}{\sum}
\mathbb{E}[qx^2[n]+r_1 u_1^2[n]+r_2 u_2^2[n]]
}
{
\underset{u_1,u_2}{\inf}
\underset{N \rightarrow \infty}{\limsup}
\frac{1}{N}
\underset{0 \leq n < N}{\sum}
\mathbb{E}[qx^2[n]+r_1 u_1^2[n]+r_2 u_2^2[n]]
}
\leq c. \nonumber
\end{align}
\label{prop:3}
\end{proposition}
\begin{proof}
See Appendix~\ref{sec:aless1} for the proof.
\end{proof}

\subsection{When $|a| \leq 0.9$}

\begin{proposition}
Consider the decentralized LQG control of Problem~\ref{prob:a}. There exists $c \leq 6$ such that for all $ |a|< 0.9$, $q$, $r_1$, $r_2$, $\sigma_{v1}$ and $\sigma_{v2}$,
\begin{align}
\frac{
\underset{u_1,u_2 \in L_{lin,kal}}{\inf}
\underset{N \rightarrow \infty}{\limsup}
\frac{1}{N}
\underset{0 \leq n < N}{\sum}
\mathbb{E}[qx^2[n]+r_1 u_1^2[n]+r_2 u_2^2[n]]
}
{
\underset{u_1,u_2}{\inf}
\underset{N \rightarrow \infty}{\limsup}
\frac{1}{N}
\underset{0 \leq n < N}{\sum}
\mathbb{E}[qx^2[n]+r_1 u_1^2[n]+r_2 u_2^2[n]]
}
\leq c \nonumber
\end{align}
\label{prop:4}
\end{proposition}
\begin{proof}
By \cite[Lemma~14]{Park_Approximation_Journal_Parti}, it is enough to show that there exists $c \in \mathbb{R}$ such that $D_U(c P_1, c P_2) \leq c \cdot D_L(P_1, P_2)$.

Upper bound: Putting $k=0$ to Lemma~\ref{lem:aless21} gives
\begin{align}
(D_U(P_1),P_1) \leq (\frac{1}{1-a^2},0) \leq (\frac{1}{1-0.9^2},0)
\end{align}

Lower bound: By Lemma~\ref{lem:less11},
\begin{align}
D_L(P_1,P_2) \geq 1
\end{align}

Ratio: $c$ is upper bounded by
\begin{align}
c \leq \frac{1}{1-0.9^2} \leq 6
\end{align}

Therefore, the lemma is proved.
\end{proof}

\subsection{Proof of Theorem~\ref{thm:2}}
\label{subsec:mainthm}
Now, by combining the results of Proposition~\ref{prop:1}, \ref{prop:2}, \ref{prop:3} we can prove the main theorem of the paper.
\begin{proof}[Proof of Theorem~\ref{subsec:mainthm}]
The proof immediately follows from Proposition~\ref{prop:1}, \ref{prop:2}, \ref{prop:3}.
\end{proof}

\bibliographystyle{IEEEtran}
\bibliography{IEEEabrv,seyongbib}

\appendix

\subsection{Proof of Corollary~\ref{cor:3}, \ref{cor:1}, \ref{cor:5}}
\label{app:cor1}

\begin{proof}[Proof of Corollary~\ref{cor:3} of Page~\pageref{cor:3}]
For simplicity, we will only proof for the case when $a=1$. The proof for the case of $a=-1$ follows similarly by replacing $a$ with $-a$.

In this case, Lemma~\ref{lem:aless21} reduces to that for all $|1-k| < 1$,
\begin{align}
D_{\sigma_v}(P) &\leq \frac{(2k-k^2)\Sigma_E + 1}{1-(1-k)^2}=\frac{(2k-k^2)\Sigma_E+1}{2k-k^2}=\frac{1}{2k-k^2}+\Sigma_E \label{eqn:cor3:1}
\end{align}
\begin{align}
P\leq k^2( \frac{(2k-k^2)\Sigma_E+1}{1-(1-k)^2} - \Sigma_E)=k^2 ( \frac{1}{2k-k^2}+\Sigma_E - \Sigma_E) = \frac{k^2}{2k-k^2} \label{eqn:cor3:2}
\end{align}
where
\begin{align}
\Sigma_E &= \frac{-1 + \sqrt{4 \sigma_{v}^2 + 1}}{2}.
\end{align}

Let $k^\star \in (0,1]$ be a constant such that $\max(1,\Sigma_E)= \frac{1}{2k^\star - k^{\star2}}$. Here, we can see that such $k^\star$ always exists since $\max(1,\Sigma_E) \geq 1$ and $\frac{1}{2k-k^2}$ is a decreasing function on $k$. Let $k \in (0,k^\star]$. Then, we can see since $0 < k^\star \leq 1$, $|1-k| < 1$. Then, \eqref{eqn:cor3:1} and \eqref{eqn:cor3:2} are again upper bounded as follows:
\begin{align}
D_{\sigma_v}(P) &= \frac{1}{2k-k^2} + \Sigma_E \\
&\leq \frac{1}{2k-k^2} + \max(1,\Sigma_E) \\
&= \frac{1}{2k-k^2} + \frac{1}{2k^\star - k^{\star 2}} \\
&\leq \frac{2}{2k-k^2}
\end{align}
where the last inequality follows from $0 < k \leq k^\star$.
\begin{align}
P &= \frac{k^2}{2k-k^2} \\
& \leq \frac{k^2 (2-k)^2}{2k-k^2} \\
& = 2k-k^2
\end{align}
where the inequality follows from $0 < k \leq k^\star \leq 1$.

Let's put $t=2k-k^2$. Then, we have $(D_{\sigma_v}(P),P) \leq (\frac{2}{t},t)$ where $t \in (0, 2k^\star - k^{\star 2}]$. Therefore, $t \in (0, \frac{1}{\max(1,\Sigma_E)}]$. This finishes the proof of the first claim.

When $\sigma_v \geq 16$, we have
\begin{align}
\Sigma_E &= \frac{-1 + \sqrt{4 \sigma_v^2 + 1 }}{2} \leq \frac{\sqrt{4\sigma_v^2 +1}}{2} \\
&\leq \frac{\sqrt{4 \sigma_v^2 + \frac{1}{16^2}  \sigma_v^2}}{2} = 1.000488... \sigma_v\\
&\leq 1.0005 \sigma_v.
\end{align}
Therefore, the range of $t$ at least includes $(0, \frac{1}{1.0005 \sigma_v}]$ and the second claim is true.

When $\sigma_v \leq 16$, we have
\begin{align}
\Sigma_E &= \frac{-1 + \sqrt{4 \cdot 16^2 +1 }}{2} =15.0078105... \leq 15.008.
\end{align}
Therefore,the range of $t$ at least includes $(0, \frac{1}{15.008}]$ and the third claim is true.
\end{proof}

\begin{proof}[Proof of Corollary~\ref{cor:1} of Page~\pageref{cor:1}]
For simplicity, we prove only for the case when $a>1$. The proof for the case of $a < -1$ follows similarly by replacing $a$ with $-a$.

Proof of (i): Let's put $k=a-\frac{1}{a}$ in Lemma~\ref{lem:aless21}. Since $|a-a+\frac{1}{a}|=|\frac{1}{a}| < 1$, the power-distortion tradeoff in \eqref{eqn:powerdis} still holds. Thus, we can see that
\begin{align}
D_U(P)&\leq \frac{(2a(a-\frac{1}{a})-(a-\frac{1}{a})^2)\Sigma_E + 1}{1-(\frac{1}{a})^2} \\
&=\frac{a^2-\frac{1}{a^2}}{1-\frac{1}{a^2}} \Sigma_E + \frac{1}{1-(\frac{1}{a})^2}\\
&=(a^2+1)\Sigma_E + \frac{a^2}{a^2-1} 
\end{align}
and
\begin{align}
P&\leq (\frac{a^2-1}{a})^2 (a^2 \Sigma_E + \frac{a^2}{a^2-1}) - \Sigma_E \\
&\leq (a^2-1)^2 \Sigma_E + (a^2-1),
\end{align}
which finishes the proof of (i).

Proof of (ii): We will divide into two cases depending on $\Sigma_E$. 

Case 1) When $\max(1, (1+a^2)\Sigma_E) > \frac{1}{1-(\frac{1}{a})^2}$.

In this case, the domain for $t$ is an empty set and we do not have to prove anything.

Case 2) When $\max(1, (1+a^2)\Sigma_E) \leq \frac{1}{1-(\frac{1}{a})^2}$.

Since $\max(1, (1+a^2)\Sigma_E) \leq \frac{1}{1-(\frac{1}{a})^2}$, there exists $\Delta^\star \in [0, \frac{1}{a}]$ such that
\begin{align}
\max(1,(a^2+1) \Sigma_E)= \frac{1}{1-(\frac{1}{a}-\Delta^\star)^2}.
\end{align}

Let's put $k=a-\frac{1}{a}+\Delta$ in Lemma~\ref{lem:aless21} where $\Delta \in [0,\Delta^\star]$. Then, we have the following upper bound on $D_{\sigma_v}(P)$  and $P$.

\begin{align}
D_{\sigma_v}(P)&\leq  \frac{ (2ak-k^2)\Sigma_E + 1}{1-(a-k)^2} \label{eqn:evalkal1} \\
&= \frac{2ak-k^2}{1-(a-k)^2}\Sigma_E + \frac{1}{1-(a-k)^2}\\
&= \frac{a^2-1+1-(a-k)^2}{1-(a-k)^2}\Sigma_E + \frac{1}{1-(a-k)^2}\\
&= (\frac{a^2-1}{1-(a-k)^2}+1)\Sigma_E + \frac{1}{1-(a-k)^2}\\
&= (\frac{a^2-1}{1-(\frac{1}{a}-\Delta)^2}+1)\Sigma_E + \frac{1}{1-(a-k)^2}\\
&\overset{(A)}{\leq} (\frac{a^2-1}{1-(\frac{1}{a})^2}+1)\Sigma_E + \frac{1}{1-(a-k)^2}\\
&= (a^2+1)\Sigma_E + \frac{1}{1-(a-k)^2} \\
&= (a^2+1)\Sigma_E + \frac{1}{1-(\frac{1}{a}-\Delta)^2} \\
&\overset{(B)}{\leq} \max(1,(a^2+1)\Sigma_E) + \frac{1}{1-(\frac{1}{a}-\Delta)^2} \\
&= \frac{1}{1-(\frac{1}{a}-\Delta^\star)^2}  + \frac{1}{1-(\frac{1}{a}-\Delta)^2} \\
&\leq \frac{1}{1-(\frac{1}{a}-\Delta)^2} + \frac{1}{1-(\frac{1}{a}-\Delta)^2} \\
&= \frac{2}{1-(\frac{1}{a}-\Delta)^2} \label{eqn:evalkal2}
\end{align}
(A): $0 \leq \Delta \leq \frac{1}{a}$\\
(B): $0 \leq \Delta \leq \Delta^\star \leq \frac{1}{a}$
\begin{align}
P&\leq k^2( \frac{(2ak-k^2) \Sigma_E + 1}{1-(a-k)^2} - \Sigma_E ) \\
&\leq k^2( \frac{(2ak-k^2) \Sigma_E + 1}{1-(a-k)^2}) \\
&\overset{(A)}{\leq} k^2 \frac{2}{1-(\frac{1}{a}-\Delta)^2}\\
&= (a-\frac{1}{a}+\Delta)^2 \frac{2}{1-(\frac{1}{a}-\Delta)^2} \\
&\overset{(B)}{\leq} (a+\Delta a  - \frac{1}{a}+ \Delta)^2 \frac{2}{1-(\frac{1}{a}-\Delta)^2} \\
&= ((a+1)(1 - \frac{1}{a}+ \Delta))^2 \frac{2}{1-(\frac{1}{a}-\Delta)^2} \\
&= \frac{2(a+1)^2 (1-\frac{1}{a}+\Delta)^2}{1-(\frac{1}{a}-\Delta)^2} \\
&= \frac{2(a+1)^2 (1-\frac{1}{a}+\Delta)}{1+\frac{1}{a}-\Delta} \\
&\overset{(C)}{\leq} 2(a+1)^2 (1-\frac{1}{a}+\Delta)(1+\frac{1}{a}-\Delta)\\
&= 2(a+1)^2 (1-(\frac{1}{a}-\Delta)^2)
\end{align}
(A): This comes from the comparison of \eqref{eqn:evalkal1} and \eqref{eqn:evalkal2}.\\
(B): Since $\Delta \geq 0$, $a >1$, we have $a - \frac{1}{a} + \Delta > 0$. Moreover, $\Delta a \geq 0$.\\
(C): $0 \leq \Delta \leq \frac{1}{a}$, $(1+\frac{1}{a}-\Delta) \geq 1$.

Therefore, by putting $t=2(a+1)^2 (1-(\frac{1}{a}-\Delta)^2)$ we can conclude
\begin{align}
(D_{\sigma_v}(P), P) \leq (\frac{4(a+1)^2}{t}, t).
\end{align}
Since $\Delta \in [0,\Delta^\star]$, we have $t \in [2(a+1)^2(1-(\frac{1}{a})^2), 2(a+1)^2 (1-(\frac{1}{a}-\Delta^\star)^2)]$. Moreover, by the definition of $\Delta^\star$, it is equivalent to $t \in [2(a+1)^2(1-(\frac{1}{a})^2), \frac{2(a+1)^2}{\max(1,(a^2+1)\Sigma_E)}]$. 

This finishes the proof of (ii).

When $1 < |a| \leq 2.5$, (i) is upper bounded as
\begin{align}
(D_{\sigma_v}(P), P)  &\leq ((a^2+1)\Sigma_E + \frac{a^2}{a^2-1}, (a^2-1)^2 \Sigma_E + (a^2-1)) \\
& \leq (7.25\Sigma_E + \frac{6.25}{a^2-1}, (a^2-1)^2 \Sigma_E + (a^2-1)).
\end{align}
Thus, we get (i').

When $1 < |a| \leq 2.5$, (ii) is also upper bounded as
\begin{align}
(D_{\sigma_v}(P), P) &\leq (\frac{4(|a|+1)^2}{t}, t) \\
&\leq  (\frac{49}{t}, t).
\end{align}
Moreover, for $1 < |a| \leq 2.5$
\begin{align}
&2(|a|+1)^2 (1-(\frac{1}{a})^2) \leq t \leq \frac{2(|a|+1)^2}{\max(1, (a^2+1)\Sigma_E)} \\
&(\Leftrightarrow)
2(1+\frac{1}{|a|})^2 (a^2-1) \leq t \leq  \frac{2(|a|+1)^2}{\max(1, (a^2+1)\Sigma_E)} \\
&(\Rightarrow)  8(a^2-1) \leq t \leq \frac{8}{\max(1, 7.25 \Sigma_E)}.
\end{align}
Therefore, we get (ii').
\end{proof}

\begin{proof}[Proof of Corollary~\ref{cor:5} of Page~\pageref{cor:5}]
For simplicity, we will only proof for the case when $0 \leq a < 1$. The proof for $-1< a \leq 0$ follows similarly by replacing $a$ with $-a$.

First part of the lemma easily follows by putting $k=0$ in Lemma~\ref{lem:aless21}.

Let's prove the second part of the lemma. Since the second part of the lemma assumes $\Sigma_E \leq \frac{1}{1-a^2}$, there always exists $k^\star \in [0,a]$ such that $\max(1,\Sigma_E)=\frac{1}{1-(a-k^\star)^2}$.

Since $0 \leq k^\star \leq a$ and $0 \leq a <1$, for all $k \in [0,k^\star]$ we have $|a-k| < 1$. Thus, by Lemma~\ref{lem:aless21}, for all $k \in [0,k^\star]$ we have the following upper bounds on $D_{\sigma_v}(P), P$.
\begin{align}
D_{\sigma_v}(P)&\leq \frac{(2ak-k^2)\Sigma_E +1 }{1-(a-k)^2} \label{eqn:cor3:10}\\
&\overset{(A)}{\leq} \frac{(1-a^2+2ak-k^2)\Sigma_E + 1}{1-(a-k)^2} \\
&= \frac{(1-(a-k)^2)\Sigma_E + 1}{1-(a-k)^2} \\
&= \Sigma_E + \frac{1}{1-(a-k)^2} \label{eqn:cor3:11}\\
&\overset{(B)}{\leq} \frac{1}{1-(a-k^\star)^2} + \frac{1}{1-(a-k)^2} \\
&\overset{(C)}{\leq} \frac{2}{1-(a-k)^2}
\end{align}
(A): $0 \leq a < 1$, $0 < k \leq a$, $\Sigma_E \geq 0$.\\
(B): $\max(1, \Sigma_E) = \frac{1}{1-(a-k^\star)^2}$.\\
(C): $0 \leq k \leq k^\star \leq a$.

\begin{align}
P&\leq k^2(\frac{(2ak-k^2)\Sigma_E +1 }{1-(a-k)^2}-\Sigma_E)\\
&\overset{(A')}{\leq} k^2(  \Sigma_E + \frac{1}{1-(a-k)^2} - \Sigma_E ) \\
&= \frac{k^2}{1-(a-k)^2} \\
&\overset{(B')}{\leq} \frac{(1-a+k)^2}{1-(a-k)^2} \\
&= \frac{1-a+k}{1+a-k} \\
&\overset{(C')}{\leq} (1-a+k)(1+a-k) \\
&= 1-(a-k)^2
\end{align}
(A'): $\eqref{eqn:cor3:10} \leq \eqref{eqn:cor3:11}$.\\
(B'): $0 \leq a < 1$ and $0 < k \leq a$.\\
(C'): $0 \leq a < 1$ and $0 < k \leq a$.

Let's put $t=1-(a-k)^2$. Then, we have $(D_{\sigma_1}(P), P) \leq (\frac{2}{t}, t)$. Moreover, since $0 \leq k \leq k^\star \leq a$, $t \in [1-a^2, 1-(a-k^\star)^2]$. Furthermore, since $t$, $t \in [1-a^452, \frac{1}{\max(1,\Sigma_E)}]$. This finishes the proof of the lemma.
\end{proof}

\subsection{Proof of Corollary~\ref{cor:2} and Proposition~\ref{prop:1}}
\label{sec:ageq1}
\begin{proof}[Proof of Corollary~\ref{cor:2} of page~\pageref{cor:2}]
For simplicity, we first prove for the case when $1 < a \leq 2.5$. The proof for the case when $-2.5 \leq a < -1$ follows similarly.

First, let's upper bound $\Sigma_1$ and $\Sigma_2$ of \eqref{eqn:cor3stat:1} and \eqref{eqn:cor3stat:2}. When $|(a^2-1)\sigma_{v1}^2 -1 | \geq |2 a \sigma_{v1}|$,  we have
\begin{align}
\Sigma_1 &\leq \frac{(a^2-1)\sigma_{v1}^2 -1  + \sqrt{2 ((a^2-1)\sigma_{v1}^2 - 1)^2}}{2a^2} \\
&\leq \frac{(1+\sqrt{2})|(a^2-1)\sigma_{v1}^2 - 1 |}{2a^2} \\
&\leq \frac{(1+\sqrt{2})\max(1,(a^2-1)\sigma_{v1}^2)}{2a^2} \label{eqn:cor2:1}
\end{align}
When $|(a^2-1)\sigma_{v1}^2 -1 | \leq |2 a \sigma_{v1}|$, we have
\begin{align}
\Sigma_1 & \leq \frac{|2a\sigma_{v1}|+\sqrt{(2a\sigma_{v1})^2 + 4 a^2 \sigma_{v1}^2}}{2a^2} \\
&= \frac{(1+\sqrt{2})2a\sigma_{v1}}{2a^2} \label{eqn:cor2:2}
\end{align}
Therefore, by \eqref{eqn:cor2:1} and \eqref{eqn:cor2:2}, we can conclude
\begin{align}
\Sigma_1 &\leq \frac{(1+\sqrt{2}) \max(1,(a^2-1)\sigma_{v1}^2, 2 a \sigma_{v1}) }{2a^2}. \label{eqn:eval3}
\end{align}
Likewise, we also have
\begin{align}
\Sigma_2 &\leq \frac{(1+\sqrt{2}) \max(1,(a^2-1)\sigma_{v2}^2, 2 a \sigma_{v2}) }{2a^2}.
\end{align}

We also have for all $k \geq 3$
\begin{align}
&\frac{a^2(1-a^{-2(k-1)})}{1-a^{-2(k-2)}}=\frac{a^{2(k-1)}-1}{a^{2(k-2)}-1}=
\frac{(a-1)(1+\cdots+ a^{(2k-4)} + a^{(2k-3)} )}
{(a-1)(1+\cdots+ a^{(2k-5)} )} \label{eqn:eval0}\\
&=
\frac{1+\cdots+ a^{(2k-4)} + a^{(2k-3)}}
{1+\cdots+ a^{(2k-5)}} \\
&= 1 + \frac{a^{(2k-4)}+a^{(2k-3)}}{1+\cdots+a^{(2k-5)}} \\
&\overset{(A)}{\leq} 1 + \frac{a^{(2k-4)}+a^{(2k-3)}}{a^{(2k-6)}+a^{(2k-5)}} \\
&= 1 + a^2 \leq 1+2.5^2 = 7.25. \label{eqn:eval1}
\end{align}
(A): This comes from $k \geq 3$.

Then, let's prove the statements of the lemma.

Proof of (a):

Since $\Sigma_1 \geq 150 $ and $\Sigma_2 \geq 150$, there exist $k_1 \geq 3$ and $k_2 \geq 3$ such that
\begin{align}
\frac{a^{2(k_1-2)}-1}{1-a^{-2}} \leq \frac{\Sigma_1}{24} < \frac{a^{2(k_1-1)}-1}{1-a^{-2}} \label{eqn:eval2}\\
\frac{a^{2(k_2-2)}-1}{1-a^{-2}} \leq \frac{\Sigma_2}{24} < \frac{a^{2(k_2-1)}-1}{1-a^{-2}}
\end{align}
We will evaluate Lemma~\ref{lem:aless22} with these $k_1$ and $k_2$, and increase $k$ arbitrary large.

Moreover, since $\Sigma_1 \geq 150$ implies $\sigma_{v1} \geq 1$, \eqref{eqn:eval3} further reduces to
\begin{align}
\Sigma_1 \leq \frac{(1+\sqrt{2}) \max((a^2-1)\sigma_{v1}^2, 2 a \sigma_{v1}) }{2a^2}. \label{eqn:eval103}
\end{align}

Let's upper bound $I$ of Lemma~\ref{lem:aless22}. First, we have
\begin{align}
&\frac{a^{2(k_1-2)}(1-a^{-2(k_1-1)})^2 }{(1-a^{-2})^2} \\
&\overset{(A)}{\leq} \frac{a^{2(k_1-2)}(1-a^{-2(k_1-1)}) }{(1-a^{-2})^2} \\
&\overset{(B)}{\leq} \frac{a^{2(k_1-2)}(7.25 a^{-2}(1-a^{-2(k_1-2)}))}{(1-a^{-2})^2} \\
&= 7.25 a^{-2}( \frac{a^{2(k_1-2)}-1}{1-a^{-2}} ) \frac{1}{1-a^{-2}} \\
&\leq \frac{7.25 \Sigma_1}{24} \frac{a^{-2}}{1-a^{-2}}=\frac{7.25 \Sigma_1}{24} \frac{1}{a^2-1}. \label{eqn:eval4}
\end{align}
(A): For $k_1 \geq 3$, $1- a^{-2(k_1-1)}\leq 1$.\\
(B): By comparing \eqref{eqn:eval0} and \eqref{eqn:eval1}, we get $7.25a^{-2}(1-a^{-2(k_1-2)}) \geq (1-a^{-2(k_1-1)})$.\\
(C): This comes from \eqref{eqn:eval2}.

Moreover, we also have
\begin{align}
&\frac{a^{2(k_1-2)}(1-a^{-2(k_1-1)})^2}{(1-a^{-2})^2} \\
&\overset{(A)}{\leq} \frac{ a^{2(k_1-2)}(7.25 a^{-2}(1-a^{-2(k_1-2)}))^2 }{(1-a^{-2})^2}\\
&\overset{(B)}{\leq} 7.25^2 (\frac{a^{2(k_1-2)}(1-a^{-2(k_1-2)})}{1-a^{-2}} )^2 \\
&\overset{(C)}{\leq} (\frac{7.25 \Sigma_1}{24})^2. \label{eqn:eval5}
\end{align}
(A): By comparing \eqref{eqn:eval0} and \eqref{eqn:eval1}, we get $7.25a^{-2}(1-a^{-2(k_1-2)}) \geq (1-a^{-2(k_1-1)})$.\\
(B): $a >1$ and $k_1 \geq 3$.\\
(C): This comes from \eqref{eqn:eval2}.

By merging the results so far, we can conclude
\begin{align}
&\frac{a^{2(k_1-2)}(1-a^{-2(k_1-1)})^2 }{(1-a^{-2})^2} \label{eqn:eval:100} \\
&\overset{(A)}{\leq} \min( \frac{7.25 \Sigma_1}{24} \frac{1}{a^2-1} ,(\frac{7.25 \Sigma_1}{24})^2) \label{eqn:eval:101}\\
&\overset{(B)}{\leq} \max( \frac{7.25}{24} \frac{1}{a^2-1} \frac{(1+\sqrt{2})(a^2-1)\sigma_{v1}^2}{2a^2} , (\frac{7.25}{24})^2 (\frac{(1+\sqrt{2})2a \sigma_{v1}}{2a^2})^2 ) \\
&= \max( \frac{7.25}{24} \frac{1+\sqrt{2}}{2a^2} , (\frac{7.25}{24})^2 (\frac{1+\sqrt{2}}{a})^2 ) \sigma_{v1}^2 \\
&\overset{(C)}{\leq} \max( \frac{7.25}{24} \frac{1+\sqrt{2}}{2} , (\frac{7.25}{24})^2 (1+\sqrt{2})^2 ) \sigma_{v1}^2 \\
&\leq 0.5319 \sigma_{v1}^2 \label{eqn:14converse2}
\end{align}
(A): This comes from \eqref{eqn:eval4} and \eqref{eqn:eval5}.\\
(B): When $(a^2-1) \sigma_{v1}^2 \geq 2a \sigma_{v1}$, by \eqref{eqn:eval103} we have $\Sigma_1 \leq \frac{(1+\sqrt{2})(a^2-1) \sigma_{v1}^2}{2a^2}$. Thus, by plugging it into \eqref{eqn:eval:101}, we get 
\begin{align}
\eqref{eqn:eval:100} \leq \frac{7.25}{24} \frac{1}{a^2-1} \frac{(1+\sqrt{2})(a^2-1)\sigma_{v1}^2}{2a^2}.
\end{align}
Likewise, when $(a^2-1) \sigma_{v1}^2 \leq 2 a \sigma_{v1}$, by \eqref{eqn:eval103} we have $\Sigma_1 \leq 
\frac{(1+\sqrt{2})2a \sigma_{v1}}{2a^2}$. Therefore, by plugging it into \eqref{eqn:eval:101}, we get
\begin{align}
\eqref{eqn:eval:100} \leq (\frac{7.25}{24})^2 (\frac{(1+\sqrt{2})2a \sigma_{v1}}{2a^2})^2.
\end{align}
(C): Because $a > 1$.

In the same ways, we can also prove that
\begin{align}
&\frac{a^{2(k_2-2)}(1-a^{-2(k_2-1)})^2 }{(1-a^{-2})^2} \leq 0.5319 \sigma_{v2}^2.
 \label{eqn:14converse1}
\end{align}

Therefore, by plugging \eqref{eqn:14converse2} and \eqref{eqn:14converse1} into $I$ of Lemma~\ref{lem:aless22}, we can upper bound $I$ by
\begin{align}
I &\leq (k_1-1) \log(1+\frac{1}{k_1-1}0.5319) \\
&\leq \log e^{0.5319}. \label{eqn:eval9}
\end{align}

Let's upper bound $I'(\widetilde{P_1})$. First, we have
\begin{align}
&2a^{2(k_2-1-k)} \frac{1-a^{-2(k_2-k_1)}}{1-a^{-2}} \Sigma
+ 2a^{2(k_2-1-k_1)} \frac{1-a^{-2(k_2-k_1)}}{1-a^{-2}} \frac{1-a^{-2(k_2-k_1)}}{1-a^{-2}} \\
&\overset{(A)}{\leq} 2a^{2(k_2-1-k)} \frac{1-a^{-2(k_2-k_1)}}{1-a^{-2}} \frac{a^{2(k-1)}(1-a^{-2(k_1-1)})}{1-a^{-2}}\\
&+
2a^{2(k_2-1-k_1)} \frac{1-a^{-2(k_2-k_1)}}{1-a^{-2}} \frac{1-a^{-2(k_2-k_1)}}{1-a^{-2}}\\
&=
2a^{2(k_2-2)}(\frac{1-a^{-2(k_2-k_1)}}{1-a^{-2}})
(
\frac{(1-a^{-2(k_1-1)})}{1-a^{-2}} + a^{2(-k_1 + 1)} \frac{1-a^{-2(k_2-k_1)}}{1-a^{-2}}
)\\
&=
2a^{2(k_2-2)}(\frac{1-a^{-2(k_2-k_1)}}{1-a^{-2}})
(\frac{1-a^{-2(k_1-1)}+a^{-2(k_1-1)}-a^{-2(k_2-1)} }{1-a^{-2}}) \\
&\overset{(B)}{\leq}
2a^{2(k_2-2)}(\frac{1-a^{-2(k_2-1)}}{1-a^{-2}})^2 \\
&\overset{(C)}{\leq}
2 \cdot 0.5319\sigma_{v2}^2. \label{eqn:14converse9}
\end{align}
(A): Since $I \geq 0$, $\Sigma \leq a^{2(k-1)} \frac{1-a^{-2(k_1-1)}}{1-a^{-2}}$.\\
(B): $k_1 \geq 1$.\\
(C): It comes from \eqref{eqn:14converse1}.

We also have
\begin{align}
&2a^{2(k_2-k_1-2)}
\frac{1-a^{-2(k_2-k_1)}}{1-a^{-2}} \frac{(1-a^{-(k_2-1-k_1)})(1-a^{-(k-k_1)})}{(1-a^{-1})^2} 
\widetilde{P_1}
\\
&\overset{(A)}{\leq}
2a^{2(k_2-k_1-2)}
\frac{1-a^{-2(k_2-k_1)}}{1-a^{-2}} \frac{(1-a^{-(k_2-1-k_1)})(1-a^{-(k-k_1)})}{(1-a^{-1})^2}
\frac{24(a^2-1)^2}{40000} \frac{a^{2(k_1-1)}-1}{1-a^{-2}}
\\
&\overset{(B)}{\leq}
2a^{2(k_2-2)}(\frac{1-a^{-2(k_2-1)}}{1-a^{-2}})^2 \frac{24 a^{-2}}{40000} \frac{(a^2-1)^2}{(1-a^{-1})^2} \\
&= 2a^{2(k_2-2)}(\frac{1-a^{-2(k_2-1)}}{1-a^{-2}})^2 \frac{24(a+1)^2}{40000} \\
&\leq \frac{48(2.5+1)^2}{40000} 0.5319\sigma_{v2}^2 \\
&= 0.00781893 \sigma_{v2}^2. \label{eqn:14converse11}
\end{align}
(A): Since we have $\widetilde{P_1} \leq \frac{(a^2-1)^2 \Sigma_1}{40000}$
and $\Sigma_1 \leq 24 \frac{a^{2(k_1-1)}-1}{1-a^{-2}}$ by \eqref{eqn:eval2}. \\
(B): Since $k_2-1 \geq k_2-k_1$ and $2(k_2-1) \geq (k_2-1-k_1)$. \\
(C): By \eqref{eqn:14converse1} and $0 \leq a \leq 2.5$.

Therefore, by \eqref{eqn:14converse9} and \eqref{eqn:14converse11}, we can bound $I'(\widetilde{P_1})$ of Lemma~\ref{lem:aless22} by
\begin{align}
I'(\widetilde{P_1}) & \leq  \frac{k_2-k_1}{2}\log(1+\frac{1}{k_2-k_1}(
2 \cdot 0.5319 + 0.00781893
)) \\
&\leq  \frac{k_2-k_1}{2}\log(1+\frac{1}{k_2-k_1}(
1.07161893
)) \\
&\leq \frac{1}{2} \log e^{1.0717}. \label{eqn:eval8}
\end{align}

Moreover, we have
\begin{align}
&a^{2(k-k_1-1)} \frac{(1-a^{-(k-k_1)})^2}{(1-a^{-1})^2} \widetilde{P_1} \\
&\overset{(A)}{\leq} a^{2(k-k_1-1)} \frac{(1-a^{-(k-k_1)})^2}{(1-a^{-1})^2} \frac{24(a^2-1)^2}{40000} \frac{a^{2(k_1-1)}-1}{1-a^{-2}}\\
&= a^{2(k-k_1-1)} \frac{1-2a^{-(k-k_1)}+a^{-2(k-k_1)}}{(1-a^{-1})^2} \frac{24(a^2-1)^2}{40000} \frac{a^{2(k_1-1)}-1}{1-a^{-2}}\\
&\overset{(B)}{\leq} a^{2(k-k_1-1)} \frac{1-a^{-2(k-k_1)}}{(1-a^{-1})^2} \frac{24(a^2-1)^2}{40000} \frac{a^{2(k_1-1)}-1}{1-a^{-2}}\\
&= \frac{a^{2(k-2)} (1-a^{-2(k-k_1)})(1-a^{-2(k_1-1)})}{(1-a^{-2})} \cdot \frac{24(a^2-1)^2}{40000(1-a^{-1})^2} \\
&\leq \frac{a^{2(k-2)} (1-a^{-2(k-1)})}{(1-a^{-2})} \cdot \frac{24(a^2-1)^2}{40000(1-a^{-1})^2} \\
&= \frac{a^{2(k-1)} (1-a^{-2(k-1)})}{(1-a^{-2})} \cdot \frac{24(a+1)^2}{40000} \\
&\overset{(C)}{\leq} \frac{a^{2(k-1)} (1-a^{-2(k-1)})}{(1-a^{-2})} \cdot \frac{24(2.5+1)^2}{40000} \\
&= \frac{a^{2(k-1)} (1-a^{-2(k-1)})}{(1-a^{-2})} \cdot \frac{147}{20000} \label{eqn:eval6}
\end{align}
(A): By \eqref{eqn:eval2} and $\widetilde{P_1} \leq \frac{(a^2-1)^2 \Sigma_1}{40000}$. \\
(B): Since $k \geq k_1$. \\
(C): Since $1 \leq a \leq 2.5$.

Likewise, we can also prove that
\begin{align}
&a^{2(k-k_2-1)} \frac{(1-a^{-(k-k_2)})^2}{(1-a^{-1})^2} \widetilde{P_2} \leq \frac{a^{2(k-1)} (1-a^{-2k-1)})}{(1-a^{-2})} \cdot \frac{147}{20000}.  \label{eqn:eval7}
\end{align}

Finally, by plugging \eqref{eqn:eval9}, \eqref{eqn:eval8}, \eqref{eqn:eval6}, \eqref{eqn:eval7} into Lemma~\ref{lem:aless22} we have
\begin{align}
&D_L(\widetilde{P_1}, \widetilde{P_2})\\
&\geq
(
\sqrt{
\frac{a^{2(k-1)}\frac{1-a^{-2(k_1-1)}}{1-a^{-2}} + a^{2(k-k_1)} \frac{1-a^{-2(k_2-k_1)}}{1-a^{-2}} +
a^{2(k-k_2)} \frac{1-a^{-2(k-k_2)}}{1-a^{-2}}
}{2^{2(I+I'(\widetilde{P_1}))}}
} \\
&-
\sqrt{
a^{2(k-k_1-1)}
\frac{(1-a^{-(k-k_1)})^2}{(1-a^{-1})^2} \widetilde{P_1}
}
-
\sqrt{
a^{2(k-k_2-1)}
\frac{(1-a^{-(k-k_2)})^2}{(1-a^{-1})^2} \widetilde{P_2}
}
)_+^2 +1
\\
&\geq
\frac{a^{2(k-1)}(1-a^{-2(k-1)})}{1-a^{-2}}
(\sqrt{ \frac{1}{e^{2 \cdot 0.5319 + 1.0717} }}-\sqrt{\frac{147}{20000}} -\sqrt{\frac{147}{20000}}  )_+^2+1\\
&\geq
\frac{a^{2(k-1)}(1-a^{-2(k-1)})}{1-a^{-2}} 0.02969 + 1.
\end{align}

Therefore, by choosing $k$ arbitrary large, we have $D_L(\widetilde{P_1},\widetilde{P_2})=\infty$.

Proof of (b):

Like (a), since $\Sigma_1 \geq 150$ and $\Sigma_2 \geq 150$, there exist $k_1 \geq 3$ and $k_2 \geq 3$ such that
\begin{align}
\frac{a^{2(k_1-2)}-1}{1-a^{-2}} \leq \frac{\Sigma_1}{24} < \frac{a^{2(k_1-1)}-1}{1-a^{-2}}, \\
\frac{a^{2(k_2-2)}-1}{1-a^{-2}} \leq \frac{\Sigma_2}{24} < \frac{a^{2(k_2-1)}-1}{1-a^{-2}}.
\end{align}
We put the parameters of Lemma~\ref{lem:aless22} as such $k_1$, $k_2$ and $k=k_2$. Then, the lower bound of Lemma~\ref{lem:aless22} reduces to
\begin{align}
D_L(\widetilde{P_1}, \widetilde{P_2})
&\geq
(
\sqrt{
\frac{ \Sigma + a^{2(k-k_1)} \frac{1-a^{-2(k-k_1)}}{1-a^{-2}}  }{2^{2I'(\widetilde{P_1})}}
}-
\sqrt{a^{2(k-k_1-1)} \frac{(1-a^{-(k-k_1)} )^2}{(1-a^{-1})^2} \widetilde{P_1}}
)_+^2 + 1.
\end{align}

Since we choose $k_1$ and $k_2$ in the same way as (a) and have the same bound on $\widetilde{P_1}$, we still have \eqref{eqn:eval9}, \eqref{eqn:eval8}, \eqref{eqn:eval6} which are
\begin{align}
&I \leq \log e^{0.5319},\\
&I'( \widetilde{P_1}) \leq \frac{1}{2} \log e^{1.0717}, \\
&a^{2(k-k_1-1)} \frac{(1-a^{-(k-k_1)})^2}{(1-a^{-1})^2} \widetilde{P_1}
\leq  \frac{a^{2(k-1)}(1-a^{-2(k-1)})}{(1-a^{-2})} \cdot \frac{147}{20000}.
\end{align}
Therefore, we can conclude
\begin{align}
&D_L(\widetilde{P_1}, \widetilde{P_2})\\
&\geq
(
\sqrt{
\frac{a^{2(k-1)}\frac{1-a^{-2(k_1-1)}}{1-a^{-2}} + a^{2(k-k_1)} \frac{1-a^{-2(k-k_1)}}{1-a^{-2}}
}{2^{2(I+I'(P_1))}}
} -
\sqrt{
a^{2(k-k_1-1)}
\frac{(1-a^{-(k-k_1)})^2}{(1-a^{-1})^2} P_1
}
)_+^2 +1
\\
&\geq
\frac{a^{2(k-1)}(1-a^{-2(k-1)})}{1-a^{-2}}
(\sqrt{ \frac{1}{e^{2 \cdot 0.5319 + 1.0717} }}-\sqrt{\frac{147}{20000}}  )_+^2+1\\
&\geq 
\frac{\Sigma_2}{24}
(\sqrt{ \frac{1}{e^{2 \cdot 0.5319 + 1.0717} }}-\sqrt{\frac{147}{20000}}  )_+^2+1\\
&\geq
 0.002774 \Sigma_2 + 1.
\end{align}

Proof of (c):

We will put $k_1 = k_2 = 1$ in Lemma~\ref{lem:aless22} and increase $k$ arbitrary large. First, the lower bound in Lemma~\ref{lem:aless22} reduces to
\begin{align}
D_L(\widetilde{P_1}, \widetilde{P_2}) \geq (\sqrt{a^{2(k-1)}\frac{1-a^{-2(k-1)}}{1-a^{-2}}}-\sqrt{\frac{a^{2(k-2)}(1-a^{-(k-1)})^2 }{(1-a^{-1})^2} \widetilde{P_1}}
-\sqrt{\frac{a^{2(k-2)}(1-a^{-(k-1)})^2 }{(1-a^{-1})^2} \widetilde{P_2}})_+^2 + 1. \label{eqn:14converse333}
\end{align}
Here, we have
\begin{align}
&\frac{a^{2(k-2)}(1-a^{-(k-1)})^2 }{(1-a^{-1})^2} \widetilde{P_1}\\
&\leq
\frac{a^{2(k-2)}(1-2a^{-(k-1)}+a^{-2(k-1)})}{(1-a^{-1})^2} \frac{1}{20}(a^{2}-1) \\
&\overset{(A)}{\leq}
\frac{a^{2(k-1)}(1-a^{-2(k-1)})}{(1-a^{-2})}(1+a^{-1})^2 \frac{1}{20} \\
&\overset{(B)}{\leq}
\frac{1}{5}\frac{a^{2(k-1)}(1-a^{-2(k-1)})}{(1-a^{-2})} \label{eqn:14converse3}
\end{align}
(A): Since $k \geq 1$.\\
(B): Since $a \geq 1$.

Likewise, we can also prove that
\begin{align}
&\frac{a^{2(k-2)}(1-a^{-(k-1)})^2 }{(1-a^{-1})^2} \widetilde{P_2} \leq 
\frac{1}{5}\frac{a^{2(k-1)}(1-a^{-2(k-1)})}{(1-a^{-2})} \label{eqn:14converse33}
\end{align}

Finally, by plugging \eqref{eqn:14converse3}, \eqref{eqn:14converse33} into \eqref{eqn:14converse333}, we get
\begin{align}
D_L(\widetilde{P_1},\widetilde{P_2}) \geq a^{2(k-1)} \frac{1-a^{-2(k-1)}}{1-a^{-2}} (1 - \sqrt{\frac{1}{5}} - \sqrt{\frac{1}{5}})^2_+ + 1.
\end{align}
Therefore, by choosing $k$ arbitrary large, we have $D_L(\widetilde{P_1},\widetilde{P_2})=\infty$.

Proof of (d):

Let $k_1 = k_2 =1$ and $P=\max(\widetilde{P_1}, \widetilde{P_2})$. Since $P \leq \frac{1}{75}$, we can find $k \geq 2$ such that
\begin{align}
\frac{a^{(k-1)}-1}{1-a^{-1}} \leq \frac{1}{30P} <  \frac{a^{k}-1}{1-a^{-1}} \label{eqn:14converseadd2}
\end{align}
By setting the parameters of Lemma~\ref{lem:aless22} to such $k_1, k_2, k$, the lower bound in Lemma~\ref{lem:aless22} reduces to
\begin{align}
D_L(\widetilde{P_1}, \widetilde{P_2}) \geq (\sqrt{a^{2(k-1)}\frac{1-a^{-2(k-1)}}{1-a^{-2}}}-\sqrt{\frac{a^{2(k-2)}(1-a^{-(k-1)})^2 }{(1-a^{-1})^2} \widetilde{P_1}}
-\sqrt{\frac{a^{2(k-2)}(1-a^{-(k-1)})^2 }{(1-a^{-1})^2} \widetilde{P_2}})_+^2 + 1. \label{eqn:14converseadd1}
\end{align}
The first term of \eqref{eqn:14converseadd1} is lower bounded as follows:
\begin{align}
&a^{2(k-1)} \frac{1-a^{-2(k-1)}}{1-a^{-2}}\\
&\overset{(A)}{\geq} \frac{a^{k}-1}{1-a^{-1}} \frac{1}{1+a^{-1}} \\
&\overset{(B)}{\geq} \frac{1}{60P} \label{eqn:14converseadd3}
\end{align}
(A): $k \geq 2$.\\
(B): \eqref{eqn:14converseadd2} and $a \geq 1$.

The second term of \eqref{eqn:14converseadd1} is upper bounded as follows:
\begin{align}
\frac{a^{2(k-2)}(1-a^{-(k-1)})^2}{(1-a^{-1})^2} \widetilde{P_1} \overset{(A)}{\leq} \frac{(a^{k-1}-1)^2}{(1-a^{-1})^2} \widetilde{P_1} \overset{(B)}{\leq} \frac{1}{900P^2} \widetilde{P_1} \leq \frac{1}{900P} \label{eqn:14converseadd4}
\end{align}
(A): $a \geq 1$.\\
(B): \eqref{eqn:14converseadd2}.

Likewise, the third term of \eqref{eqn:14converseadd1} is upper bounded as
\begin{align}
\frac{a^{2(k-2)}(1-a^{-(k-1)})^2}{(1-a^{-1})^2} \widetilde{P_1} \leq \frac{1}{900P}. \label{eqn:14converseadd5}
\end{align}

Therefore, by plugging \eqref{eqn:14converseadd3}, \eqref{eqn:14converseadd4}, \eqref{eqn:14converseadd5} into \eqref{eqn:14converseadd1}, we conclude
\begin{align}
D_L(\widetilde{P_1}, \widetilde{P_2}) &\geq \frac{1}{P}(\sqrt{\frac{1}{60}}-\sqrt{\frac{1}{900}}-\sqrt{\frac{1}{900}})^2+1 \\
&\geq 0.00389 \frac{1}{P} + 1
\end{align}

Proof of (e):

Since $\Sigma_2 \geq 150$, we can find $k \geq 3$ such that
\begin{align}
\frac{a^{2(k-2)}-1}{1-a^{-2}} \leq \frac{\Sigma_2}{24} < \frac{a^{2(k-1)}-1}{1-a^{-2}}. \label{eqn:14converseadd6}
\end{align}
Let $k_2 = k$ and $k_1 =1$. By putting the parameters of Lemma~\ref{lem:aless22} with these parameters, the lower bound of Lemma~\ref{lem:aless22} reduces to 
\begin{align}
D_L(\widetilde{P_1}, \widetilde{P_2})
\geq
(
\sqrt{\frac{a^{2(k-1)}\frac{1-a^{-2(k-1)}}{1-a^{-2}} }{2^{2I'(\widetilde{P_1})}}}
-
\sqrt{\frac{a^{2(k-2)}(1-a^{-(k-1)})^2}{(1-a^{-1})^2}\widetilde{P_1}}
)_+^2+1. \label{eqn:14converseadd7}
\end{align}

We will upper bound $I'(\widetilde{P_1})$. 
First, since we chose $k$ in the same ways as $k_2$ of (a), \eqref{eqn:14converse9} still holds, i.e. 
\begin{align}
&2a^{2(k-2)} \frac{1-a^{-2(k-1)}}{1-a^{-2}} \Sigma
+ 2a^{2(k-2)} \frac{1-a^{-2(k-1)}}{1-a^{-2}} \frac{1-a^{-2(k-1)}}{1-a^{-2}}
\leq
2 \cdot 0.5319\sigma_{v2}^2. \label{eqn:14conversee1}
\end{align}
Moreover, we have
\begin{align}
&2a^{2(k-3)}
\frac{1-a^{-2(k-1)}}{1-a^{-2}}
\frac{(1-a^{-(k-2)})(1-a^{-(k-1)})}{(1-a^{-1})^2} \widetilde{P_1}  \\
&\overset{(A)}{\leq}
2a^{2(k-3)}
\frac{1-a^{-2(k-1)}}{1-a^{-2}}
\frac{(1-a^{-(k-2)})(1-a^{-(k-1)})}{(1-a^{-1})^2}
\frac{1}{24} \frac{1-a^{-2}}{a^{2(k-2)}-1}
\\
&\overset{(B)}{\leq}
\frac{1}{12} \frac{a^{-4}(a^{2(k-1)}-1)}{a^{2(k-2)}-1} \frac{(1-a^{-(k-1)})^2}{(1-a^{-1})^2} \\
&\overset{(C)}{\leq}
\frac{1}{12} (1+a^{-1})^2 a^{-2} \frac{a^{2(k-2)}(1-a^{-2(k-1)})^2}{(1-a^{-2})^2} 
\\
&\overset{(D)}{\leq}
\frac{1}{3} 0.5319 \sigma_{v2}^2 \label{eqn:14conversee2}
\end{align}
(A): $\widetilde{P_1} \leq \frac{1}{\Sigma_2} \leq \frac{1}{24} \frac{1-a^{-2}}{a^{2(k-2)}-1}$.\\
(B): $1-a^{-(k-2)} \leq 1 - a^{-(k-1)}$.\\
(C): Since $k \geq 3$ and $1 \leq a \leq 2.5$, we have
\begin{align}
&(a^{2(k-1)}-1) \leq (a^{4(k-2)}-1)\\
(\Leftrightarrow) &(a^{2(k-1)}-1) \leq (a^{2(k-2)}-1)(a^{2(k-2)}+1) \\
(\Rightarrow) &(a^{2(k-1)}-1) \leq (a^{2(k-2)}-1)(a^{k-1}+1)^2 \\
(\Leftrightarrow) & (a^{2(k-1)}-1) \leq a^{2(k-1)} (a^{2(k-2)}-1)(1+a^{-(k-1)})^2\\
(\Leftrightarrow) & \frac{a^{-4}(a^{2(k-1)}-1)(1-a^{-(k-1)})^2}{(a^{2(k-2)}-1)(1-a^{-1})^2}
\leq (1+a^{-1})^2 a^{-2} \frac{a^{2(k-2)}(1-a^{-2(k-1)})^2}{(1-a^{-2})^2}.
\end{align}
(D): This comes from $1 \leq a \leq 2.5$ and \eqref{eqn:14converse1}.

Therefore, by \eqref{eqn:14conversee1} and \eqref{eqn:14conversee2}, we have
\begin{align}
I'(\widetilde{P_1})& \leq \frac{1}{2} \log (1 + \frac{1}{k-1}(
2 \cdot 0.5319 + \frac{1}{3} 0.5319
)^{k-1} \\
&\leq \frac{1}{2} \log (1 + \frac{1}{k-1}(
1.2411
))^{k-1} \\
&\leq \frac{1}{2} \log  e^{1.2411}. \label{eqn:14converseadd8}
\end{align}
We also have
\begin{align}
&\frac{a^{2(k-2)}(1-a^{-(k-1)})^2}{(1-a^{-1})^2}\widetilde{P_1} \\
&=
\frac{a^{-2}(a^{k-1}-1)^2}{(1-a^{-1})^2} \widetilde{P_1} \\
&\overset{(A)}{\leq}
\frac{a^{-2}(a^{2(k-2)}-1)^2}{(1-a^{-2})^2}(1+a^{-1})^2 \widetilde{P_1} \\
&\overset{(B)}{\leq}
(\frac{\Sigma_2}{24})^2 4 \widetilde{P_1} \\
&\overset{(C)}{\leq} \frac{\Sigma_2}{144} \label{eqn:14converseadd9}
\end{align}
(A): This comes from $2(k-2) \geq (k-1)$.\\
(B): By \eqref{eqn:14converseadd6} and $1 \leq a \leq 2.5$.\\
(C): Since $P_1 \leq \frac{1}{\Sigma_2}$.

Therefore, by plugging \eqref{eqn:14converseadd6}, \eqref{eqn:14converseadd8}, \eqref{eqn:14converseadd9} into \eqref{eqn:14converseadd7}, we get
\begin{align}
D_L(\widetilde{P_1},\widetilde{P_2})
&\geq
(
\sqrt{\frac{a^{2(k-1)} \frac{1-a^{-2(k-1)}}{1-a^{-2}} }{2^{2I'(\widetilde{P_1})}}}
-\sqrt{\frac{\Sigma_2}{144}}
)_+^2 + 1 \\
&\geq
(
\sqrt{\frac{ \Sigma_2 }{24 \cdot 2^{2I'(\widetilde{P_1})}}}
-\sqrt{\frac{\Sigma_2}{144}}
)_+^2 + 1 \\
&\geq \Sigma_2( \sqrt{\frac{1}{24 \cdot  e^{1.2411}} }-\sqrt{\frac{1}{144}} )^2 + 1 \\
&\geq 0.0006976\Sigma_2  +1. 
\end{align}

Proof of (f):

Since $\widetilde{P_1} \leq \frac{1}{150}$, there exists $k \geq 3$ such that 
\begin{align}
\frac{a^{2(k-2)}-1}{1-a^{-2}} \leq \frac{1}{24 \widetilde{P_1}} < \frac{a^{2(k-1)}-1}{1-a^{-2}}. \label{eqn:14converseadd10}
\end{align}
Let $k_2=k$ and $k_1 =1$. By putting the parameters of Lemma~\ref{lem:aless22} as these parameters, the lower bound of Lemma~\ref{lem:aless22} reduces to 
\begin{align}
D_L(\widetilde{P_1},\widetilde{P_2})
\geq
(
\sqrt{\frac{a^{2(k-1)}\frac{1-a^{-2(k-1)}}{1-a^{-2}} }{2^{2I'(\widetilde{P_1})}}}
-
\sqrt{\frac{a^{2(k-2)}(1-a^{-(k-1)})^2}{(1-a^{-1})^2}\widetilde{P_1}}
)_+^2+1 \label{eqn:14converseadd11}
\end{align}
We will upper bound $I'(\widetilde{P_1})$.
Since we assumed $\frac{1}{\widetilde{P_1}} \leq \Sigma_2$, by \eqref{eqn:14converseadd10}
we have $\frac{a^{2(k-2)}-1}{1-a^{-2}} \leq \frac{\Sigma_2}{24}$.
Therefore, \eqref{eqn:14converse9} still holds and we have
\begin{align}
&2a^{2(k-2)} \frac{1-a^{-2(k-1)}}{1-a^{-2}} \Sigma
+ 2a^{2(k-2)} \frac{1-a^{-2(k-1)}}{1-a^{-2}} \frac{1-a^{-2(k-1)}}{1-a^{-2}}
\leq
2 \cdot 0.5319\sigma_{v2}^2.
\end{align}
Since $\widetilde{P_1} \leq \frac{1}{24} \frac{1-a^{-2}}{a^{2(k-2)}-1}$, following the same process of \eqref{eqn:14conversee2} we have
\begin{align}
&2a^{2(k-3)}
\frac{1-a^{-2(k-1)}}{1-a^{-2}}
\frac{(1-a^{-(k-2)})(1-a^{-(k-1)})}{(1-a^{-1})^2} \widetilde{P_1}  \\
&\leq \frac{1}{3} 0.5319 \sigma_{v2}^2.
\end{align}
Therefore, $I'(\widetilde{P_1})$ is upper bounded by
\begin{align}
I'(\widetilde{P_1})& \leq \frac{1}{2} \log (1 + \frac{1}{k-1}(
2 \cdot 0.5319 + \frac{1}{3} 0.5319
)^{k-1} \\
&\leq \frac{1}{2} \log (1 + \frac{1}{k-1}(
1.2411
))^{k-1} \\
&\leq \frac{1}{2} \log  e^{1.2411} \label{eqn:14converseadd12}
\end{align}
Moreover, we also have
\begin{align}
&\frac{a^{2(k-2)}(1-a^{-(k-1)})^2}{(1-a^{-1})^2} \widetilde{P_1} \\
&=
\frac{a^{-2}(a^{k-1}-1)^2}{(1-a^{-1})^2} \widetilde{P_1} \\
&\overset{(A)}{\leq}
\frac{a^{-2}(a^{2(k-2)}-1)^2}{(1-a^{-2})^2}(1+a^{-1})^2 \widetilde{P_1} \\
&\overset{(B)}{\leq}
(\frac{1}{24 \widetilde{P_1}})^2 4 \widetilde{P_1} = \frac{1}{144 \widetilde{P_1}} \label{eqn:14converseadd13}
\end{align}
(A): This comes from $2(k-2) \geq (k-1)$.\\
(B): By \eqref{eqn:14converseadd10} and $1 \leq a \leq 2.5$.

Therefore, by plugging \eqref{eqn:14converseadd10}, \eqref{eqn:14converseadd12}, \eqref{eqn:14converseadd13} into \eqref{eqn:14converseadd11}, we have
\begin{align}
D_L(\widetilde{P_1},\widetilde{P_2})
&\geq
(
\sqrt{\frac{a^{2(k-1)} \frac{1-a^{-2(k-1)}}{1-a^{-2}} }{2^{2I'(\widetilde{P_1})}}}
-\sqrt{\frac{1}{144 \widetilde{P_1}}}
)_+^2 + 1 \\
&\geq
(
\sqrt{\frac{ 1 }{24 \widetilde{P_1} \cdot 2^{2I'(\widetilde{P_1})}}}
-\sqrt{\frac{1}{144 \widetilde{P_1}}}
)_+^2 + 1 \\
&\geq \frac{1}{\widetilde{P_1}}( \sqrt{\frac{1}{24 \cdot  e^{1.2411}} }-\sqrt{\frac{1}{144}} )^2 + 1 \\
&\geq \frac{0.000697686...}{\widetilde{P_1}}+1 \\
&\geq \frac{0.0006976}{\widetilde{P_1}}+1. \\
\end{align}

Proof of (g): 

Since $\Sigma_2 \geq 150$, we can find $k_2 \geq 3$ such that
\begin{align}
\frac{a^{2(k_2-2)}-1}{1-a^{-2}} \leq \frac{\Sigma_2}{24} < \frac{a^{2(k_2-1)}-1}{1-a^{-2}}
\end{align}
Let $k_1=1$ and increase $k$ arbitrary large. By plugging such parameters to Lemma~\ref{lem:aless22}, the lemma reduces to 
\begin{align}
D_L(\widetilde{P_1},\widetilde{P_2})&\geq
(
\sqrt{
\frac{ a^{2(k-1)} \frac{1-a^{-2(k_2-1)}}{1-a^{-2}}  }{2^{2I'(\widetilde{P_1})}}
+a^{2(k-k_2)} \frac{1-a^{-2(k-k_2)}}{1-a^{-2}}
}\\
&
-
\sqrt{a^{2(k-2)} \frac{(1-a^{-(k-1)} )^2}{(1-a^{-1})^2} \widetilde{P_1}}
-
\sqrt{a^{2(k-k_2-1)} \frac{(1-a^{-(k-k_2)} )^2}{(1-a^{-1})^2} \widetilde{P_2}}
)_+^2 + 1. \label{eqn:14converseadd20}
\end{align}
We will first upper bound $I'(\widetilde{P_1})$. Following the same steps as \eqref{eqn:14converse9}, we get
\begin{align}
&2a^{2(k-2)} \frac{1-a^{-2(k-1)}}{1-a^{-2}} \Sigma
+ 2a^{2(k-2)} \frac{1-a^{-2(k-1)}}{1-a^{-2}} \frac{1-a^{-2(k-1)}}{1-a^{-2}}
\leq
2 \cdot 0.5319\sigma_{v2}^2.
\end{align}
We also have
\begin{align}
&a^{2(k_2-3)} \frac{1-a^{-2(k_2-1)}}{1-a^{-2}} \frac{(1-a^{-(k_2-2)})(1-a^{-(k-1)})}{(1-a^{-1})^2} \widetilde{P_1}\\
&\overset{(A)}{\leq} a^{2(k_2-3)} \frac{(1-a^{-2(k_2-1)})(1-a^{-(k_2-2)})}{(1-a^{-1})^2} \frac{\widetilde{P_1}}{1-a^{-2}} \\
&\overset{(B)}{\leq} a^{2(k_2-3)} \frac{(1-a^{-2(k_2-1)})^2}{(1-a^{-1})^2} \frac{1}{20}a^2 \\
&= \frac{a^{2(k_2-2)}(1-a^{-2(k_2-1)})^2 }{(1-a^{-2})^2} \frac{1}{20}(1+a^{-1})^2 \\
&\overset{(C)}{\leq} \frac{a^{2(k_2-2)}(1-a^{-2(k_2-1)})^2 }{(1-a^{-2})^2} \frac{1}{5} \\
&\overset{(D)}{\leq} \frac{0.5319 }{5}\sigma_{v2}^2.
\end{align}
(A): Since $0 \leq 1-a^{-(k-1)}  \leq 1$.\\
(B): Since we assumed $\widetilde{P_1} \leq \frac{1}{20}(a^2-1)$.\\
(C): Since $1 \leq a \leq 2.5$.\\
(D): This follows from that \eqref{eqn:14converse1} still holds.

Therefore, $I'(\widetilde{P_1})$ is upper bounded by
\begin{align}
I'(\widetilde{P_1})& \leq \frac{1}{2} \log( 1 + \frac{(2+\frac{2}{5})0.5319}{k_2-1}
)^{k_2-1} \\
&\leq \frac{1}{2} \log( 1 + \frac{1.27656}{k_2-1}
)^{k_2-1} \\
& \leq \frac{1}{2} \log e^{1.27656}. \label{eqn:eval10}
\end{align}
Following the same steps as \eqref{eqn:14converse3}, we still have
\begin{align}
&\frac{a^{2(k-2)}(1-a^{-(k-1)})^2 }{(1-a^{-1})^2} \widetilde{P_1} \leq
\frac{1}{5}\frac{a^{2(k-1)}(1-a^{-2(k-1)})}{(1-a^{-2})}. \label{eqn:14converseadd21}
\end{align}
Following the same steps as \eqref{eqn:eval6}, we still have
\begin{align}
&a^{2(k-k_2-1)} \frac{(1-a^{-2(k-k_2)})^2}{(1-a^{-1})^2} \widetilde{P_2} \leq  \frac{a^{2(k-1)}(1-a^{-2(k-1)})}{(1-a^{-2})} \frac{147}{20000} \label{eqn:14converseadd22}
\end{align}
Therefore, by plugging \eqref{eqn:eval10}, \eqref{eqn:14converseadd21}, \eqref{eqn:14converseadd22} into \eqref{eqn:14converseadd20} we conclude
\begin{align}
D_L(\widetilde{P_1}, \widetilde{P_2}) &\geq \frac{a^{2(k-1)}(1-a^{-2(k-1)})}{(1-a^{-2})} (\sqrt{\frac{1}{e^{1.27656}}}
-\sqrt{\frac{1}{5}}-\sqrt{\frac{147}{20000}})^2 +1 \\
&\geq  0.00002252 \frac{a^{2(k-1)}(1-a^{-2(k-1)})}{(1-a^{-2})}  + 1. \\
\end{align}
Finally, by increasing $k$ arbitrarily large, we can prove $D_L(\widetilde{P_1}, \widetilde{P_2}) = \infty$.

Proof of (h):

Compared to (g), we can notice that only the conditions for the controller $1$ and $2$ are flipped. Thus, by symmetry the proof is the same as (g).

Proof of (i):

Since $\Sigma_2 \geq 150$, we can find $k \geq 3$ such that
\begin{align}
\frac{a^{2(k-2)}-1}{1-a^{-2}} \leq \frac{\Sigma_2}{24} < \frac{a^{2(k-1)}-1}{1-a^{-2}}. \label{eqn:eval:new1}
\end{align}
Let $k_1=1$ and $k_2=k$. By plugging these parameters into Lemma~\ref{lem:aless22}, the lower bound of  Lemma~\ref{lem:aless22} reduces to
\begin{align}
D_L(\widetilde{P_1},\widetilde{P_2})
\geq
(
\sqrt{\frac{a^{2(k-1)}\frac{1-a^{-2(k-1)}}{1-a^{-2}} }{2^{2I'(\widetilde{P_1})}}}
-
\sqrt{\frac{a^{2(k-2)}(1-a^{-(k-1)})^2}{(1-a^{-1})^2}\widetilde{P_1}}
)_+^2+1. \label{eqn:14converseadd30}
\end{align}

Following the same steps as \eqref{eqn:eval10}, we still have
\begin{align}
&I'(\widetilde{P_1}) \leq \frac{1}{2} \log e^{1.27656}. \label{eqn:14converseadd31}
\end{align}
Following the same steps as \eqref{eqn:14converse3}, we can prove
\begin{align}
\frac{a^{2(k-2)}(1-a^{-(k-1)})^2 }{(1-a^{-1})^2} \widetilde{P_1} \leq
\frac{1}{5}\frac{a^{2(k-1)}(1-a^{-2(k-1)})}{(1-a^{-2})}. \label{eqn:14converseadd32}
\end{align}
Therefore, by plugging \eqref{eqn:14converseadd31}, \eqref{eqn:14converseadd32} into \eqref{eqn:14converseadd30}, we conclude
\begin{align}
D_L(\widetilde{P_1},\widetilde{P_2}) &\geq \frac{a^{2(k-1)}(1-a^{-2(k-1)})}{(1-a^{-2})} (\sqrt{\frac{1}{e^{1.27656}}}-\sqrt{\frac{1}{5}})^2 +1 \\
&\geq \frac{a^{2(k-1)}(1-a^{-2(k-1)})}{(1-a^{-2})} 0.00655882... +1 \\
&\overset{(A)}{\geq} \frac{\Sigma_2}{24} 0.00655882... +1 \\
&\geq 0.000273284... \Sigma_2 + 1 \\
&\geq 0.0002732 \Sigma_2 +1.
\end{align}
(A): This comes from \eqref{eqn:eval:new1}.

Proof of (j):

We will prove this by analyzing the centralized controller performance which has both $y_1[n]$, $y_2[n]$ and has no input power constraint.

Define $y_1'[n]:=x[n]+v_1'[n]$ and $y_2'[n]:=x[n]+v_2'[n]$ where $v_1'[n] \sim \mathcal{N}(0,\sigma_1^2)$ and $v_2'[n] \sim \mathcal{N}(0,\sigma_1^2)$ are i.i.d. random variables. Since the costs of centralized controllers are monotone in the variances of observations, the cost of the centralized controller with the observations $y_1[n]$, $y_2[n]$ is larger than the cost of the centralized controller with the observations $y_1'[n]$, $y_2'[n]$. Moreover, by the maximum ratio combining, the cost of the centralized controller with the observations $y_1'[n]$, $y_2'[n]$ is equivalent to the cost of the centralized controller with a scalar observation $\frac{y_1'[n]+y_2'[n]}{2}$. 

Now, we can apply Lemma~\ref{lem:aless21} to analyze the performance of such a controller with the observation $\frac{y_1'[n]+y_2'[n]}{2}$. 
Let $\Sigma_E$ be the Kalman filtering performance with the observation $\frac{y_1'[n]+y_2'[n]}{2}$. Then, by Lemma~\ref{lem:aless21}, $\Sigma_E$ is lower bounded by
\begin{align}
\Sigma_E & = \frac{(a^2-1)(\frac{\sigma_{v1}^2}{2})-1
+\sqrt{((a^2-1)(\frac{\sigma_{v1}^2}{2})-1)^2 + 4a^2 \frac{\sigma_{v1}^2}{2}}
}{2a^2}\\
&\geq  \frac{\max( (a^2-1)\frac{\sigma_{v1}^2}{2}-1 , \sqrt{2}a\sigma_{v1} -1)}{2a^2}. \label{eqn:14converseadd42}
\end{align}
Therefore, for all $\widetilde{P_1}, \widetilde{P_2}$ the decentralized controller's cost is lower bounded as follows:
\begin{align}
D_L(\widetilde{P_1},\widetilde{P_2}) &\overset{(A)}{\geq} \inf_{|a-k| < 1}\frac{(2ak-k^2)\Sigma_E+1}{1-(a-k)^2}\\
&= \inf_{|a-k| < 1} \frac{2ak-k^2}{1-(a-k)^2} \Sigma_E + \frac{1}{1-(a-k)^2} \\
&\overset{(B)}{\geq} \inf_{|a-k| < 1} \frac{1-a^2+2ak-k^2}{1-(a-k)^2} \Sigma_E +  1 \\
&= \Sigma_E + 1 \label{eqn:14converseadd41}
\end{align}
(A): The decentralized control cost is larger than the centralized controller's cost with the observation $\frac{y_1'[n]+y_2'[n]}{2}$. Moreover, when $|a-k| \geq 1$ the centralized control system is unstable, and the cost diverges to infinity. When $|a-k| < 1$, the cost analysis follows from Lemma~\ref{lem:aless21}.\\
(B): This comes from $a >1$ and $2ak-k^2 \geq 1-(a-k)^2 > 0$.

Therefore, by \eqref{eqn:14converseadd42} and \eqref{eqn:14converseadd41} for all $\widetilde{P_1}, \widetilde{P_2}$ we have
\begin{align}
D_L(\widetilde{P_1},\widetilde{P_2}) &\geq \max(\frac{\max( (a^2-1)\frac{\sigma_{v1}^2}{2}-1 , \sqrt{2}a\sigma_{v1} -1)}{2a^2},1)\\
&\geq \frac{\max((a^2-1)\frac{\sigma_{v1}^2}{2}-1, \sqrt{2}a \sigma_{v1} -1) }{4a^2} + \frac{1}{2} \\
&\geq \frac{\max((a^2-1)\frac{\sigma_{v1}^2}{2}-1, \sqrt{2}a \sigma_{v1} -1) }{4a^2} + \frac{1}{2a^2} \\
&\geq \frac{\max((a^2-1)\frac{\sigma_{v1}^2}{2}, \sqrt{2}a \sigma_{v1}, 1)}{4a^2}
\end{align}
By \eqref{eqn:eval3} we already know 
\begin{align}
\Sigma_1 \leq \frac{(1+\sqrt{2})\max(1,(a^2-1)\sigma_{v1}^2,2a\sigma_{v1})}{2a^2}.
\end{align}
Therefore,
\begin{align}
D(\widetilde{P_1},\widetilde{P_2}) &\geq \frac{ \max(1,(a^2-1)\frac{\sigma_{v1}^2}{2}, \sqrt{2}a \sigma_{v1}) }{4a^2} \\
&\geq \min( \frac{\frac{1}{4}}{\frac{1+\sqrt{2}}{2}} , \frac{\frac{1}{8}}{ \frac{1+\sqrt{2}}{2}} , \frac{ \frac{\sqrt{2}}{4}}{1+\sqrt{2}} ) \Sigma_1 \\
&= \frac{1}{4(1+\sqrt{2})} \Sigma_1\\
&\geq 0.1035\Sigma_1.
\end{align}
\end{proof}
As mentioned in \eqref{eqn:14converseadd41}, $D(\widetilde{P_1}, \widetilde{P_2}) \geq 1$. Thus, the statement (j) is true.

\begin{proof}[Proof of Proposition~\ref{prop:1} of page~\pageref{prop:1}]
Consider the power-distortion tradeoff $D(P_1, P_2)$ for the decentralized control problem shown in Problem~\ref{prob:c}. Since we can achieve the tradeoff of the single controller systems by turning on only one controller, we have $(D(P_1, P_2), P_1, P_2) \leq (\min(D_{\sigma_{v1}}(P_1), D_{\sigma_{v2}}(P_2)), P_1, P_2)$ where the definition of $D_{\sigma}(P)$ is shown in Problem~\ref{prob:b}. 

By \cite[Lemma~14]{Park_Approximation_Journal_Parti}, if there exists $c \geq 1$ such that for all $\widetilde{P_1}, \widetilde{P_2} \geq 0$, $\min(D_{\sigma 1}(c \widetilde{P_1}), D_{\sigma 2}(c \widetilde{P_2})) \leq c \cdot D_L(\widetilde{P_1},\widetilde{P_2})$, then for all $q, r_1, r_2 \geq 0$ we have
\begin{align}
\frac{\min_{P_1, P_2 \geq 0} q \min(D_{\sigma 1}(c P_1), D_{\sigma 2}(c P_2)) + r_1 P_1 + r_2 P_2}{\min_{\widetilde{P_1}, \widetilde{P_2} \geq 0} q D_L(\widetilde{P_1},\widetilde{P_2}) + r_1 P_1 + r_2 P_2} \leq c
\end{align}
which finishes the proof. Therefore, we will only prove that such $c$ exists.


Before we start the proof, define the subscript $max$ as $argmax_{i \in \{1,2 \}} \widetilde{P_i}$. For example, if $\widetilde{P_1} < \widetilde{P_2}$ then $\widetilde{P_{max}}=\widetilde{P_2}$, $P_{max}=P_2$, $\Sigma_{max}=\Sigma_2$, $D_{\sigma_{v max}}(P)=D_{\sigma_{v 2}}(P)$ and so on. Furthermore, for notational simplicity, we write $D_{\sigma_{v1}}(\cdot), D_{\sigma_{v2}}(\cdot), D_{\sigma_{v max}}(\cdot)$ as $D_{v1}(\cdot), D_{v2}(\cdot), D_{v max}(\cdot)$ respectively.

For the proof, we will first divide the cases based on $\Sigma_1, \Sigma_2$ then further divide based on $\widetilde{P_1}, \widetilde{P_2}$. Remind that since $\sigma_{v1} \leq \sigma_{v2}$, we have $\Sigma_1 \leq \Sigma_2$. We can use this fact to reduce the cases.

(i) When $\Sigma_1 \leq \Sigma_2 \leq 150$

(i-i) When $\frac{1}{150} \leq \max(\widetilde{P_1},\widetilde{P_2})$

Lower bound: By Corollary~\ref{cor:2} (j)
\begin{align}
D_L(\widetilde{P_1},\widetilde{P_2}) \geq 1
\end{align}

Upper bound:

If $(a^2-1) \leq \frac{1}{\max(1,7.25\Sigma_{max})}$, then the range for $t$ in Corollary~\ref{cor:1} (ii') is not an empty set. Therefore, by plugging $t=\frac{8}{\max(1,7.25 \Sigma_E)}$ we get
\begin{align}
(D_{ \sigma max}(P_{max}),P_{max})&\leq ( \frac{49}{8} \max(2,14.5\Sigma_{max}),\frac{8}{\max(1,7.25\Sigma_{max})}) \\
&\leq (\frac{49}{8} \cdot 14.5 \cdot 150, 8) (\because \Sigma_1 \leq \Sigma_2 \leq 150).
\end{align}

If $(a^2-1) \geq \frac{1}{\max(1,7.25\Sigma_{max})}$, by Corollary~\ref{cor:1} (i') we get
\begin{align}
(D_{\sigma max}(P_{max}),P_{max})& \leq (7.25 \Sigma_{max}+ \frac{6.25}{a^2-1}, (a^2-1)^2 \Sigma_{max}+(a^2-1)) \\
&\leq (7.25 \Sigma_{max}+ \frac{6.25}{a^2-1}, 27.5625 \Sigma_{max}+5.25)  (\because 1 < |a| \leq 2.5)\\
&\leq (7.25 \Sigma_{max}+ 6.25 {\max(1,7.25 \Sigma_{max})}, 27.5625 \Sigma_{max}+5.25) \\
&\leq (7.25 \cdot 150 + 6.25 \cdot {7.25 \cdot 150}, 27.5625 \cdot 150 +5.25) (\because \Sigma_1 \leq \Sigma_2 \leq 150)\\
&\leq (7884.375,4139.625).
\end{align}

Ratio: $c$ is upper bounded by
\begin{align}
c \leq \frac{4139.625}{\frac{1}{150}} < 10^6.
\end{align}

(i-ii) When $\frac{1}{20}(a^2-1) \leq \max(\widetilde{P_1},\widetilde{P_2}) \leq \frac{1}{150}$

Lower bound: By Corollary~\ref{cor:2} (d),
\begin{align}
D_L(\widetilde{P_1},\widetilde{P_2}) \geq 0.00389 \frac{1}{\max(\widetilde{P_1},\widetilde{P_2})}+1. \label{eqn:lowerbound1}
\end{align}

Upper bound:

If $8(a^2-1) \leq \max(\widetilde{P_1},\widetilde{P_2}) \leq \frac{8}{\max(1,7.25\Sigma_{max})}$, then we can put $t=\widetilde{P_{max}}$ in Corollary~\ref{cor:1} (ii') for $D_{\sigma max}(P_{max})$. Therefore, we get
\begin{align}
(D_{\sigma max}(P_{max}),P_{max}) \leq (\frac{49}{\widetilde{P_{max}}}, \widetilde{P_{max}} ).
\end{align}

If $\frac{1}{20}(a^2-1) \leq \max(\widetilde{P_1},\widetilde{P_2}) \leq 8(a^2-1)$

In this case, the lower bound \eqref{eqn:lowerbound1} can be further lower bounded as
\begin{align}
(D_{L}(P_1,P_2),P_{max}) \geq (\frac{0.00389}{24.5(a^2-1)}+1,\frac{1}{20}(a^2-1)).
\end{align}
By Corollary~\ref{cor:1} (i'), we have
\begin{align}
(D_{\sigma max}(P_{max}),P_{max}) &\leq (7.25\Sigma_{max}+\frac{6.25}{a^2-1} , (a^2-1)^2 \Sigma_{max}+(a^2-1) ) \\
&\leq ( 7.25 \cdot 150 + \frac{6.25}{a^2-1}, 5.25 \cdot 150(a^2-1)+(a^2-1)) (\because 1 \leq |a| < 2.5, \Sigma_1 \leq \Sigma_2 \leq 150)\\
&= (\frac{6.25}{a^2-1}+1087.5 , 788.5(a^2-1) ).
\end{align}

If $\frac{8}{\max(1,7.25\Sigma_{max})} \leq \max(\widetilde{P_1},\widetilde{P_2}) \leq \frac{1}{150}$

Notice that this case never happens since
\begin{align}
\frac{8}{\max(1,7.25\Sigma_{max})} \geq \frac{8}{7.25 \cdot 150} > \frac{1}{150}.
\end{align}
%
%
%
%
%

Ratio: $c$ is upper bounded by
\begin{align}
c \leq \frac{24.5 \times 6.25}{0.00389} < 40000.
\end{align}

(i-iii) When $\max(\widetilde{P_1},\widetilde{P_2}) \leq \frac{1}{20}(a^2-1)$

Lower bound: By Corollary~\ref{cor:2} (c),
\begin{align}
D_L(\widetilde{P_1},\widetilde{P_2}) = \infty.
\end{align}
We do not need a corresponding upper bound.

(ii) When $\Sigma_1 \leq 150 \leq \Sigma_2$

(ii-i) When $\frac{20}{a^2-1} \geq \Sigma_2$

(ii-i-i) When $\frac{1}{150} \leq \widetilde{P_1}$

Lower bound: By Corollary~\ref{cor:2} (j),
\begin{align}
D_L(\widetilde{P_1},\widetilde{P_2}) \geq 1
\end{align}

If $(a^2-1) \leq \frac{1}{\max(1,7.25\Sigma_1)}$

Upper bound: By putting $t=\frac{8}{\max(1,7.25 \Sigma_1)}$ to Corollary~\ref{cor:1} (ii') for $D_{\sigma 1}(P_1)$, we have
\begin{align}
(D_{\sigma 1}(P_1),P_1) &\leq
(\frac{49}{8}\max(1,7.25\Sigma_{1}),\frac{8}{\max(1,7.25\Sigma_{1})}) \\
&\leq (\frac{49}{8} \cdot 7.25 \cdot 150, 8) (\because \Sigma_1 \leq 150)
\end{align}

If $(a^2-1) \geq \frac{1}{\max(1,7.25\Sigma_1)}$

Upper bound: By Corollary~\ref{cor:1} (i'),
\begin{align}
(D_{\sigma 1}(P_1),P_1) & \leq (7.25 \Sigma_1 + \frac{6.25}{a^2-1}, 27.5625 \Sigma_1 + 5.25) \\
& \leq ( 7.25 \cdot 150 + 6.25 \max(1,7.25\Sigma_1),27.5625 \cdot 150 + 5.25 ) (\because \Sigma_1 \leq 150)\\
& \leq ( 7.25 \cdot 150 + 6.25 \cdot 7.25 \cdot 150,27.5625 \cdot 150 + 5.25 ).
\end{align}
Ratio: $c$ is upper bounded by
\begin{align}
c \leq \frac{27.5625 \cdot 150 + 5.25}{\frac{1}{150}} = 620943.75.
\end{align}

(ii-i-ii) When $\frac{1}{\Sigma_2} \leq \widetilde{P_1} \leq \frac{1}{150}$

Lower bound: By Corollary~\ref{cor:2} (f)
\begin{align}
D_L(\widetilde{P_1},\widetilde{P_2}) \geq \frac{0.0006976}{\widetilde{P_1}} +1. \label{eqn:lowerbound2}
\end{align}

If $\frac{8}{\max(1,7.25\Sigma_1)} \leq \widetilde{P_1} \leq \frac{1}{150}$,

This never happens since $\frac{8}{\max(1,7.25\Sigma_1)} \geq \frac{8}{7.25 \cdot 150} > \frac{1}{150}$.

If $8(a^2-1) \leq  \widetilde{P_1} \leq \frac{8}{\max(1,7.25\Sigma_1)}$,

Upper bound: By plugging $t=\widetilde{P_1}$ to Corollary~\ref{cor:1} (ii') for $D_{\sigma 1}(P_1)$, we get
\begin{align}
(D_{\sigma 1}(P_1), P_1) \leq (\frac{49}{\widetilde{P_1}}, \widetilde{P_1}).
\end{align}

If $\frac{1}{\Sigma_2} \leq \widetilde{P_1} \leq 8(a^2-1)$,

Here the lower bound of \eqref{eqn:lowerbound2} is further lower bounded by
\begin{align}
(D_{L}(\widetilde{P_1}, \widetilde{P_2}), \widetilde{P_1}) & \geq (\frac{0.0006976}{8(a^2-1)}+1 , \frac{1}{\Sigma_2}). (\because \frac{1}{\Sigma_2} \leq \widetilde{P_1} \leq 8(a^2-1) ).
\end{align}

Upper bound: When $\frac{1}{\Sigma_2} \leq \widetilde{P_1} \leq 8(a^2-1)$ and $(a^2-1) \leq \frac{1}{\max(1,7.25\Sigma_1)}$,
we can plug $t=8(a^2-1)$ for $D_{\sigma 1}(P_1)$ to Corollary~\ref{cor:1} (ii') for $D_{\sigma 1}(P_1)$. Then, we get
\begin{align}
(D_{\sigma 1}(P_1),P_1) &\leq  (\frac{49}{8(a^2-1)}, 8(a^2-1)) \\
&\leq (\frac{49}{8(a^2-1)}, \frac{8 \cdot 20}{\Sigma_2} ) (\because \mbox{In (ii-i), we assumed }\frac{20}{a^2-1} \geq \Sigma_2).
\end{align}

When $\frac{1}{\Sigma_2} \leq P_1 \leq 8(a^2-1)$ and $(a^2-1) > \frac{1}{\max(1,7.25\Sigma_1)}$, by Corollary~\ref{cor:1} (i') we get
\begin{align}
(D_{\sigma 1}(P_1),P_1) &\leq (7.25 \Sigma_1 + \frac{6.25}{a^2-1}, (a^2-1)^2 \Sigma_1 + (a^2-1)) \\
&\leq (7.25 \Sigma_1+\frac{6.25}{a^2-1}, \frac{20^2 \Sigma_1}{ \Sigma_2^2} + \frac{20}{\Sigma_2}) \\
&(\because \mbox{In (ii-i), we assumed } \frac{20}{a^2-1} \geq \Sigma_2)\\
&\leq (\frac{6.25}{a^2-1}+7.25 \cdot 150, \frac{20^2 + 20}{\Sigma_2}). (\because \Sigma_1 \leq 150 \leq \Sigma_2)
\end{align}

Ratio: $c$ is upper bounded by
\begin{align}
c \leq \frac{8 \times 6.25}{0.0006976} < 72000.
\end{align}

(ii-i-iii) When $\widetilde{P_1} \leq \frac{1}{\Sigma_2}$ and $\max(\widetilde{P_1},\widetilde{P_2})=\widetilde{P_2} > \frac{1}{\Sigma_2}$

Lower bound: By Corollary~\ref{cor:2} (e)
\begin{align}
D_L(P_1,P_2) \geq 0.0006976\Sigma_2 + 1.
\end{align}

First, since $\Sigma_2 \geq 150$, we can see that $\max(1,7.25\Sigma_2)=7.25\Sigma_2$.

If $(a^2-1) \leq \frac{1}{7.25 \Sigma2}$

Upper bound: By plugging $t=\frac{8}{7.25 \Sigma_2}$ into Corollary~\ref{cor:1} (ii') for $D_{\sigma 2}(P_2)$, we get
\begin{align}
(D_{\sigma 2}(P_2),P_2) \leq (\frac{49}{8} \cdot 7.25 \Sigma_2 , \frac{8}{7.25\Sigma_2}).
\end{align}

If $(a^2-1) \geq \frac{1}{7.25 \Sigma2}$

Upper bound: By Corollary~\ref{cor:1} (i'), we get
\begin{align}
(D_{\sigma 2}(P_2),P_2) &\leq (7.25 \Sigma_2 + \frac{6.25}{a^2-1}, (a^2-1)^2 \Sigma_2 + (a^2-1)) \\
&\leq (7.25 \Sigma_2+\frac{6.25}{a^2-1}, \frac{20^2 \Sigma_2}{ \Sigma_2^2} + \frac{20}{\Sigma_2}) \\
&(\because \mbox{In (ii-i), we assumed }\frac{20}{a^2-1} \geq \Sigma_2)\\
&\leq (7.25 \Sigma_2+6.25 \cdot 7.25 \Sigma_2, \frac{20^2 + 20}{\Sigma_2}).
\end{align}

Ratio: $c$ is upper bounded by
\begin{align}
c \leq \frac{7.25 + 6.25 \times 7.25}{0.0006976} < 76000.
\end{align}

(ii-i-iv) When $\widetilde{P_1} \leq \frac{1}{\Sigma_2}$ and $\frac{1}{20}(a^2-1) \leq \max(\widetilde{P_1},\widetilde{P_2}) \leq \frac{1}{\Sigma_2}$

Lower bound: By Corollary~\ref{cor:2} (d), we have
\begin{align}
D_{L}(\widetilde{P_1},\widetilde{P_2}) \geq \frac{0.00389}{\max(\widetilde{P_1},\widetilde{P_2})}+1. \label{eqn:lowerbound3}
\end{align}

If $\frac{8}{\max(1,7.25 \Sigma_{max})} \leq \max(\widetilde{P_1},\widetilde{P_2}) \leq \frac{1}{\Sigma_2}$

This case never happens, since $\frac{8}{\max(1, 7.25 \Sigma_{max})} = \frac{8}{7.25 \Sigma_2} > \frac{1}{\Sigma_2}$.

If $8(a^2-1) \leq \max(\widetilde{P_1},\widetilde{P_2}) \leq \frac{8}{\max(1,7.25 \Sigma_{max})}$

Upper bound: By plugging $t=\widetilde{P_{max}}$ into Corollary~\ref{cor:1} (ii') for $D_{\sigma max}(P_{max})$, we have
\begin{align}
(D_{\sigma max}(P_{max}), P_{max}) \leq (\frac{49}{\widetilde{P_{max}}}, \widetilde{P_{max}}).
\end{align}

If $\frac{1}{20}(a^2-1) \leq \max(P_1,P_2) \leq 8(a^2-1)$ 

In this case, the lower bound of \eqref{eqn:lowerbound3} is further lower bounded by
\begin{align}
(D_L(\widetilde{P_1},\widetilde{P_2}), P_{max}) \geq (\frac{0.00389}{8(a^2-1)}+1, \frac{1}{20}(a^2-1)).
\end{align}

Upper bound: By Corollary~\ref{cor:1} (i'), we have
\begin{align}
(D_{\sigma max}(P_{max}),P_{max}) &\leq ( 7.25 \Sigma_{max} + \frac{6.25}{a^2-1},(a^2-1)^2 \Sigma_{max}+(a^2-1))\\
&\leq ( 7.25 \Sigma_2 + \frac{6.25}{a^2-1},(a^2-1)^2 \Sigma_2+(a^2-1)) (\because \Sigma_1 \leq \Sigma_2)\\
&\leq (\frac{7.25 \cdot 20}{a^2-1}+ \frac{6.25}{a^2-1}, 20(a^2-1) + (a^2-1)) (\because \mbox{In (ii-i), we assumed }\frac{20}{a^2-1} \geq \Sigma_2)\\
&\leq ( \frac{7.25 \cdot 20 + 6.25}{a^2-1}, 21(a^2-1))
\end{align}

Ratio: $c$ is upper bounded by
\begin{align}
c \leq \frac{7.25 \cdot 20 + 6.25}{\frac{0.00389}{8}} < 320000.
\end{align}

(ii-i-v) When $\widetilde{P_1} \leq \frac{1}{\Sigma_2}$ and $\max(\widetilde{P_1},\widetilde{P_2}) \leq \frac{1}{20}(a^2-1) $

Lower bound: By Corollary~\ref{cor:2} (c),
\begin{align}
D_L(\widetilde{P_1}, \widetilde{P_2}) \geq \infty.
\end{align}
Thus, we do not need a corresponding upper bound in this case.

(ii-ii) When $\frac{20}{a^2-1} \leq \Sigma_2$

(ii-ii-i) When $\frac{1}{150} \leq \widetilde{P_1}$

Compared to (ii-i-i), the only difference is $\Sigma_2$ and $\Sigma_2$ does not affect the result of (ii-i-i). Therefore, in the same way as (ii-i-i), we can prove that $c$ is bounded by the same constant as (ii-i-i).

(ii-ii-ii) When $\frac{1}{20}(a^2-1) \leq \widetilde{P_1} \leq \frac{1}{150}$

Lower bound: Since in (ii-ii) we assumed $\frac{20}{a^2-1} \leq \Sigma_2$, we have$\frac{1}{\Sigma_2} \leq \frac{1}{20}(a^2-1) \leq P_1  \leq \frac{1}{150}$. Therefore, we can apply Corollary~\ref{cor:2} (f) to get
\begin{align}
D_L(\widetilde{P_1},\widetilde{P_2}) \geq \frac{0.0006976}{\widetilde{P_1}}+1. \label{eqn:lowerbound4}
\end{align}

If $\frac{8}{\max(1,7.25 \Sigma_1)} \leq \widetilde{P_1} \leq \frac{1}{150}$

Since we assumed $\Sigma_1 \leq 150$ in (ii), $\frac{8}{\max(1, 7.25 \Sigma_1)} \geq \frac{8}{7.25 \cdot 150} > \frac{1}{150}$. Therefore, this case never happens.

If $8(a^2-1) \leq \widetilde{P_1} \leq \frac{8}{\max(1,7.25 \Sigma_1)}$

Upper bound: By plugging $t=\widetilde{P_1}$ into Corollary~\ref{cor:1} (ii') for $D_{\sigma 1}(P_1)$, we have
\begin{align}
(D_{\sigma 1}(P_1), P_1) \leq (\frac{49}{\widetilde{P_1}}, \widetilde{P_1}).
\end{align}

If $\frac{1}{20}(a^2-1) \leq \widetilde{P_1} \leq 8(a^2-1)$

In this case, the lower bound of \eqref{eqn:lowerbound4} is further lower bounded by
\begin{align}
(D_L(\widetilde{P_1},\widetilde{P_2}), \widetilde{P_1}) \geq (\frac{0.0006976}{8(a^2-1)}+1, \frac{1}{20}(a^2-1)).
\end{align}

Upper bound: By Corollary~\ref{cor:1} (i')
\begin{align}
(D_{U}(P_1),P_1) &\leq ( 7.25 \Sigma_1 + \frac{6.25}{a^2-1},(a^2-1)^2 \Sigma_1+(a^2-1))\\
&\leq (7.25 \cdot 150 + \frac{6.25}{a^2-1},5.25 \cdot 150 \cdot (a^2-1)+(a^2-1))\\
&(\because \Sigma_1 \leq 150, 1 < |a| \leq 2.5)
\end{align}
Ratio: $c$ is upper bounded by
\begin{align}
c \leq \frac{6.25 \times 8}{0.0006976} < 72000.
\end{align}

(ii-ii-iii) When $\widetilde{P_1} \leq \frac{1}{20}(a^2-1)$ and $\widetilde{P_2} \geq \frac{(a^2-1)^2 \Sigma_2}{40000}$

Lower bound: By Corollary~\ref{cor:2} (i), we have
\begin{align}
D_L(\widetilde{P_1}, \widetilde{P_2}) \geq 0.0002732 \Sigma_2 + 1.
\end{align}

Upper bound: By Corollary~\ref{cor:1} (i'), we have
\begin{align}
(D_{\sigma 2}(P_2),P_2) &\leq (7.25 \Sigma_2 + \frac{6.25}{a^2-1}, (a^2-1)^2 \Sigma_2 + (a^2-1))\\
&\leq (7.25 \Sigma_2 + \frac{6.25}{20}\Sigma_2 , (a^2-1)^2 \Sigma_2 + (a^2-1)^2 \frac{\Sigma_2}{20} )\\
&(\because \mbox{In (ii-ii), we assumed }\frac{20}{a^2-1} \leq \Sigma_2)
\end{align}

Ratio: $c$ is upper bounded by
\begin{align}
c \leq \frac{1+\frac{1}{20}}{\frac{1}{40000}} \leq 42000.
\end{align}

(ii-ii-iv) When $\widetilde{P_1} \leq \frac{1}{20}(a^2-1)$ and $\widetilde{P_2} \leq \frac{(a^2-1)^2 \Sigma_2}{40000}$

Lower bound: By Corollary~\ref{cor:1} (g), we have
\begin{align}
D_L(\widetilde{P_1}, \widetilde{P_2}) = \infty
\end{align}
We do not need a matching upper bound.

(iii) When $150 \leq \Sigma_1 \leq \Sigma_2$

In this case, we can see that $\max(1,7.25 \Sigma_1)=7.25 \Sigma_1$, $\max(1,7.25 \Sigma_2)=7.25 \Sigma_2$.

(iii-i) When $\frac{20}{a^2-1} \geq \Sigma_1$ and $\frac{20}{a^2-1} \geq \Sigma_2$

(iii-i-i)  When $\frac{1}{\Sigma_1} \leq \widetilde{P_1}$

Lower bound: By Corollary~\ref{cor:2} (j), we have
\begin{align}
D_L(\widetilde{P_1}, \widetilde{P_2}) \geq 0.1035 \Sigma_1.
\end{align}

If $(a^2-1) \leq \frac{1}{7.25\Sigma_1}$

Upper bound: By plugging $t=\frac{8}{7.25 \Sigma_1}$ into Corollary~\ref{cor:1} (ii') for $D_{\sigma 1}(P_1)$, we get 
\begin{align}
(D_{\sigma 1}(P_1),P_1) = (49 \cdot \frac{7.25 \Sigma_1}{8}, \frac{8}{7.25 \Sigma_1}). 
\end{align}

If $(a^2-1) \geq \frac{1}{7.25\Sigma_1}$

Upper bound: By Corollary~\ref{cor:1} (i'), we have
\begin{align}
(D_{\sigma 1}(P_1),P_1) &= (7.25 \Sigma_1 + \frac{6.25}{a^2-1}, (a^2-1)^2 \Sigma_1 + (a^2-1))\\
&\leq (7.25 \Sigma_1 + 6.25 \cdot 7.25 \Sigma_1 , (\frac{20}{\Sigma_1})^2 \Sigma_1 + \frac{20}{\Sigma_1})\\
&(\because \mbox{In (iii-i), we assumed } \frac{20}{a^2 - 1} \geq \Sigma_1)\\
&=(7.25^2 \Sigma_1,  \frac{20\cdot 21}{\Sigma_1}).
\end{align}

Ratio: $c$ is upper bounded by
\begin{align}
c \leq \frac{7.25^2}{0.1035} < 510.
\end{align}

(iii-i-ii) When $\frac{1}{\Sigma_2} \leq \widetilde{P_1} \leq \frac{1}{\Sigma_1}$

Lower bound: Since in (iii) we assumed $150 \leq \Sigma_1$, we have $\frac{1}{\Sigma_2} \leq \widetilde{P_1} \leq \frac{1}{\Sigma_1} \leq \frac{1}{150}$. Therefore, we can apply Corollary~\ref{cor:2} (f) to conclude
\begin{align}
D_L(\widetilde{P_1}, \widetilde{P_2}) \geq \frac{0.0006976}{\widetilde{P_1}} +1. \label{eqn:lowerbound5}
\end{align}

If $\frac{8}{7.25 \Sigma_1} \leq \widetilde{P_1} \leq \frac{1}{\Sigma_1}$

Since $\frac{8}{7.25 \Sigma_1}> \frac{1}{\Sigma_1}$, this case never happens.

If $8(a^2-1) \leq \widetilde{P_1} \leq \frac{8}{7.25 \Sigma_1}$

Upper bound: By plugging $t=\widetilde{P_1}$ into Corollary~\ref{cor:1} (ii') for $D_{\sigma 1}(P_1)$, we have
\begin{align}
(D_{\sigma 1}(P_1), P_1) \leq (\frac{49}{\widetilde{P_1}}, \widetilde{P_1}).
\end{align}

If $\frac{1}{\Sigma_2} \leq \widetilde{P_1} \leq 8(a^2-1)$

In this case, the lower bound of \eqref{eqn:lowerbound5} is further lower bounded by
\begin{align}
(D_{L}(\widetilde{P_1}, \widetilde{P_2}), \widetilde{P_1}) & \geq (\frac{0.0006976}{8(a^2-1)}+1 , \frac{1}{\Sigma_2}).
\end{align}

If $\frac{1}{\Sigma_2} \leq P_1 \leq 24.5(a^2-1)$ and $(a^2-1) \leq \frac{1}{7.25\Sigma_1}$

Upper bound: By plugging $t=8(a^2-1)$ into Corollary~\ref{cor:1} (ii') for $D_{\sigma 1}(P_1)$, we have
\begin{align}
(D_{\sigma 1}(P_1), P_1) &\leq  (\frac{49}{8(a^2-1)}, 8(a^2-1)) \\
&\leq (\frac{49}{8(a^2-1)}, \frac{8 \cdot 20}{\Sigma_2} ) \\
&(\because \mbox{In (iii-i), we assumed } \frac{20}{a^2-1} \geq \Sigma_2)
\end{align}

If $\frac{1}{\Sigma_2} \leq P_1 \leq 8(a^2-1)$ and $(a^2-1) > \frac{1}{7.25\Sigma_1}$

Upper bound: By Corollary~\ref{cor:1} (i'), we have
\begin{align}
(D_{\sigma 1}(P_1),P_1) &\leq (7.25 \Sigma_1 + \frac{6.25}{a^2-1}, (a^2-1)^2 \Sigma_1 + (a^2-1)) \\
&\leq (\frac{7.25 \cdot 20}{a^2-1}+\frac{6.25}{a^2-1}, \frac{20^2 \Sigma_1}{ \Sigma_2^2} + \frac{20}{\Sigma_2}) \\
&(\because \mbox{In (iii-i), we assumed } \Sigma_1 \leq \frac{20}{a^2-1}, \Sigma_2 \leq \frac{20}{a^2-1})\\
&\leq (\frac{7.25 \cdot 20 + 6.25 }{a^2-1}, \frac{20^2 + 20}{\Sigma_2})\\
&(\because \Sigma_1 \leq \Sigma_2)
\end{align}

Ratio: $c$ is upper bounded by
\begin{align}
c \leq \frac{7.25 \cdot 20 + 6.25}{ \frac{0.0006976}{8}} \leq 2 \times 10^6.
\end{align}

(iii-i-iii) When $\widetilde{P_1} \leq \frac{1}{\Sigma_2}$ and $\max(\widetilde{P_1}, \widetilde{P_2})=\widetilde{P_2} > \frac{1}{\Sigma_2}$

Lower bound: By Corollary~\ref{cor:2} (e),
\begin{align}
D_L(\widetilde{P_1}, \widetilde{P_2}) \geq 0.0006976 \Sigma_2 + 1.
\end{align}

If $(a^2-1) \leq \frac{1}{7.25\Sigma_2}$

Upper bound: By plugging $t=\frac{8}{7.25 \Sigma_2}$ into Corollary~\ref{cor:1} (ii') for $D_{\sigma 2}(P_2)$, we get
\begin{align}
(D_{\sigma 2}(P_2),P_2) \leq (14.5 \Sigma_2 , \frac{24.5}{7.25\Sigma_2}).
\end{align}

If $(a^2-1) \geq \frac{1}{7.25\Sigma_2}$

Upper bound: By Corollary~\ref{cor:1} (i'), we have
\begin{align}
(D_{\sigma 2}(P_2),P_2) &\leq (7.25 \Sigma_2 + \frac{6.25}{a^2-1}, (a^2-1)^2 \Sigma_2 + (a^2-1)) \\
&\leq (7.25 \Sigma_2 + {6.25 \cdot 7.25}\Sigma_2, \frac{20^2}{\Sigma_2}+\frac{20}{\Sigma_2}) \\
&(\because \mbox{In (iii-i), we assumed } \Sigma_2 \leq \frac{20}{a^2-1})\\
&\leq (7.25^2 \Sigma_2 , \frac{20 \cdot 21}{\Sigma_2}).
\end{align}

Ratio: $c$ is upper bounded by
\begin{align}
c \leq \frac{7.25^2}{0.0006976} < 80000.
\end{align}

(iii-i-iv) When $\widetilde{P_1} \leq \frac{1}{\Sigma_2}$ and $\frac{1}{20}(a^2-1) \leq \max(\widetilde{P_1}, \widetilde{P_2}) \leq \frac{1}{\Sigma_2}$

Lower bound: Since we assumed $\Sigma_2 \geq 150$ in (iii), we have $\max(P_1,P_2) \leq \frac{1}{\Sigma_2} \leq  \frac{1}{150} \leq \frac{1}{75}$. Therefore, by Corollary~\ref{cor:2} (d) we can see
\begin{align}
D_{L}(\widetilde{P_1}, \widetilde{P_2}) \geq \frac{0.00389}{\max(\widetilde{P_1}, \widetilde{P_2})}+1. \label{eqn:lowerbound6}
\end{align}

If $\frac{8}{7.25 \Sigma_{max}} \leq \max(\widetilde{P_1}, \widetilde{P_2}) \leq \frac{1}{\Sigma_2}$

This never happens since
\begin{align}
\frac{8}{7.25\Sigma_{max}} > \frac{1}{\Sigma_{max}} \geq \frac{1}{\Sigma_2}.
\end{align}

If $8(a^2-1) \leq \max(\widetilde{P_1}, \widetilde{P_2}) \leq \frac{8}{7.25 \Sigma_{max}}$

Upper bound: By plugging $t=\widetilde{P_{max}}$ into Corollary~\ref{cor:1} (ii'), we get
\begin{align}
(D_{\sigma max}(P_{max}), P_{max}) \leq (\frac{49}{\widetilde{P_{max}}}, \widetilde{P_{max}} ).
\end{align}

If $\frac{1}{20}(a^2-1) \leq \max(\widetilde{P_1}, \widetilde{P_2}) \leq 8(a^2-1)$ 

In this case, we can notice that the lower bound of \eqref{eqn:lowerbound6} is further lower bounded by 
\begin{align}
(D_L(\widetilde{P_1}, \widetilde{P_2}), \widetilde{P_{max}}) \geq (\frac{0.00389}{8(a^2-1)}+1, \frac{1}{20}(a^2-1))
\end{align}

Upper bound: By Corollary~\ref{cor:1} (i'), we get
\begin{align}
(D_{\sigma max}(P_{max}),P_{max}) &\leq ( 7.25 \Sigma_{max} + \frac{6.25}{a^2-1},(a^2-1)^2 \Sigma_{max}+(a^2-1))\\
&\leq (\frac{7.25 \cdot 20}{a^2-1}+ \frac{6.25}{a^2-1}, 20(a^2-1) + (a^2-1)) \\
&(\because \mbox{In (iii-i), we assumed }\Sigma_1 \leq \frac{20}{a^2-1}, \Sigma_2 \leq \frac{20}{a^2-1}) \\
&\leq ( \frac{7.25 \cdot 20 + 6.25}{a^2-1}, 21(a^2-1))
\end{align}

Ratio: $c$ is upper bounded by
\begin{align}
c \leq \frac{7.25 \cdot 20 + 6.25}{\frac{0.00389}{8}} < 320000.
\end{align}

(iii-i-v) When $\widetilde{P_1} \leq \frac{1}{\Sigma_2}$ and $\max(\widetilde{P_1}, \widetilde{P_2}) \leq \frac{1}{20}(a^2-1) $

Lower bound: By Corollary~\ref{cor:2} (c),
\begin{align}
D_L(\widetilde{P_1}, \widetilde{P_2}) \geq \infty.
\end{align}

We do not need a corresponding upper bound.

(iii-ii) When $\Sigma_1 \leq \frac{20}{a^2-1} \leq \Sigma_2$

(iii-ii-i)  When $\frac{1}{\Sigma_1} \leq \widetilde{P_1}$

Compared to the case (iii-i-i), the conditions for $\Sigma_1$, $\widetilde{P_1}$ are the same and the only difference is the condition for $\Sigma_2$. However, the condition for $\Sigma_2$ does not affect the argument of (iii-i-i). Thus, the same bound on $c$ as (iii-i-i) still holds for this case.

(iii-ii-ii) When $\frac{1}{20}(a^2-1) \leq \widetilde{P_1} \leq \frac{1}{\Sigma_1}$

Lower bound: Since we assumed $150 \leq \Sigma_1$ in (iii), we have $\frac{1}{\Sigma_2} \leq \frac{1}{20}(a^2-1) \leq P_1 \leq \frac{1}{\Sigma_1} \leq \frac{1}{150}$. Thus, we can apply Corollary~\ref{cor:2} (f) to get
\begin{align}
D_L(\widetilde{P_1}, \widetilde{P_2}) \geq \frac{0.0006976}{\widetilde{P_1}}+1. \label{eqn:lowerbound8}
\end{align}

If $\frac{8}{7.25 \Sigma_1} \leq \widetilde{P_1} \leq \frac{1}{\Sigma_1}$

This case never happens.

If $8(a^2-1) \leq \widetilde{P_1} \leq \frac{8}{7.25 \Sigma_1}$

Upper bound: By plugging $t=\widetilde{P_1}$ into Corollary~\ref{cor:1} (ii') for $D_{\sigma 1}(P_1)$, we get
\begin{align}
(D_{\sigma 1}(P_1), P_1) \leq (\frac{49}{\widetilde{P_1}}, \widetilde{P_1}).
\end{align}

If $\frac{1}{20}(a^2-1) \leq P_1 \leq 8(a^2-1)$

In this case, the lower bound of \eqref{eqn:lowerbound8} can be further lower bounded as
\begin{align}
(D_L(\widetilde{P_1}, \widetilde{P_2}), \widetilde{P_1}) \geq (\frac{0.0006976}{8(a^2-1)}+1, \frac{1}{20}(a^2-1))
\end{align}

Upper bound: By Corollary~\ref{cor:1} (i'), we have
\begin{align}
(D_{\sigma 1}(P_1),P_1) &\leq ( 7.25 \Sigma_1 + \frac{6.25}{a^2-1},(a^2-1)^2 \Sigma_1+(a^2-1))\\
&\leq (\frac{7.25 \cdot 20}{a^2-1}+ \frac{6.25}{a^2-1}, 20(a^2-1) + (a^2-1)) \\
&(\because \mbox{In (iii-ii), we assumed } \Sigma_1 \leq \frac{20}{a^2-1}.) \\
&\leq ( \frac{7.25 \cdot 20 + 6.25}{a^2-1}, 21(a^2-1))
\end{align}

Ratio: $c$ is upper bounded by
\begin{align}
c \leq \frac{7.25 \cdot 20 + 6.25}{\frac{0.006976}{8}} < 320000.
\end{align}

(iii-ii-iii) When $\widetilde{P_1} \leq \frac{1}{20}(a^2-1)$ and $\widetilde{P_2} \geq \frac{(a^2-1)^2 \Sigma_2}{40000}$

Lower bound: By Corollary~\ref{cor:2} (i), we have
\begin{align}
D_L(\widetilde{P_1}, \widetilde{P_2}) \geq 0.0002732 \Sigma_2 + 1.
\end{align}

Upper bound: By Corollary~\ref{cor:1} (i'), we have
\begin{align}
(D_{\sigma 2}(P_2),P_2) &\leq (7.25 \Sigma_2 + \frac{6.25}{a^2-1}, (a^2-1)^2 \Sigma_2 + (a^2-1))\\
&\leq (7.25 \Sigma_2 + \frac{6.25}{20}\Sigma_2 , (a^2-1)^2 \Sigma_2 + (a^2-1)^2 \frac{\Sigma_2}{20} )\\
&(\because \mbox{In (iii-ii), we assumed } \frac{20}{a^2-1} \leq \Sigma_2)
\end{align}

Ratio: $c$ is upper bounded by
\begin{align}
c \leq 40000(1+\frac{1}{20}) \leq 42000.
\end{align}

(iii-ii-iv) When $\widetilde{P_1} \leq \frac{1}{20}(a^2-1)$ and $\widetilde{P_2} \leq \frac{(a^2-1)^2 \Sigma_2}{40000}$

Lower bound: By Corollary~\ref{cor:2} (g),
\begin{align}
D_L(\widetilde{P_1}, \widetilde{P_2}) = \infty.
\end{align}

Therefore, we do not need a corresponding upper bound.

(iii-iii) When $\frac{20}{a^2-1} \leq \Sigma_1 \leq \Sigma_2$

(iii-iii-i) When $\widetilde{P_1} \geq \frac{(a^2-1)^2 \Sigma_1}{40000}$

Lower bound: By Corollary~\ref{cor:2} (j), we have
\begin{align}
D_L(\widetilde{P_1}, \widetilde{P_2}) \geq 0.1035 \Sigma_1
\end{align}

Upper bound: By Corollary~\ref{cor:1} (i'), we have
\begin{align}
(D_{\sigma 1}(P_1),P_1) &\leq (7.25 \Sigma_1 + \frac{6.25}{a^2-1}, (a^2-1)^2 \Sigma_1 + (a^2-1))\\
&\leq (7.25 \Sigma_1 + \frac{6.25}{20}\Sigma_1 , (a^2-1)^2 \Sigma_1 + (a^2-1)^2 \frac{\Sigma_1}{20} )\\
&(\because \mbox{In (iii-iii), we assumed }\frac{20}{a^2-1} \leq \Sigma_1)
\end{align}

Ratio: $c$ is upper bounded by
\begin{align}
c \leq 40000(1+\frac{1}{20}) \leq 42000.
\end{align}

(iii-iii-ii) When $\widetilde{P_1} \leq \frac{(a^2-1)^2 \Sigma_1}{40000}$ and $\widetilde{P_2} \geq \frac{(a^2-1)^2 \Sigma_2}{40000}$

Lower bound: By Corollary~\ref{cor:2} (b), we have
\begin{align}
D_L(P_1,P_2) \geq 0.002774 \Sigma_2 +1.
\end{align}

Upper bound: By Corollary~\ref{cor:1} (i'), we have
\begin{align}
(D_{U}(P_2),P_2) &\leq (7.25 \Sigma_2 + \frac{6.25}{a^2-1}, (a^2-1)^2 \Sigma_2 + (a^2-1))\\
&\leq (7.25 \Sigma_2 + \frac{6.25}{20}\Sigma_2 , (a^2-1)^2 \Sigma_2 + (a^2-1)^2 \frac{\Sigma_2}{20} )\\
&(\because \mbox{In (iii-iii), we assume }\frac{20}{a^2-1} \leq \Sigma_2)
\end{align}

Ratio: $c$ is upper bounded by
\begin{align}
c \leq 40000(1+\frac{1}{20}) \leq 42000.
\end{align}

(iii-iii-iii) When $\widetilde{P_1} \leq \frac{(a^2-1)^2 \Sigma_1}{40000}$ and $\widetilde{P_2} \leq \frac{(a^2-1)^2 \Sigma_2}{40000}$

Lower bound: By Corollary~\ref{cor:2} (a), we have
\begin{align}
D_L(\widetilde{P_1}, \widetilde{P_2})=\infty.
\end{align}

Therefore, we do not need a corresponding upper bound.

Finally, by (i), (ii), (iii), we get the constant $c \leq 6 \times 10^6$ and prove the proposition.
\end{proof}

\subsection{Proof of Lemma~\ref{lem:eq11}, Corollary~\ref{cor:3} and Proposition~\ref{prop:2}}
\label{sec:aeq1}

\begin{proof}[Proof of Lemma~\ref{lem:eq11} of Page~\pageref{lem:eq11}]
For simplicity, we assume $a=1$, $1 < k_1 < k_2 < k$. The remaining cases when $a=-1$ or $k_1=1$ or $k_2=k_1$ or $k=k_2$ easily follow with minor modifications.

We essentially follow the proof of Lemma~\ref{lem:aless22}. However, since $|a|=1$, the sum of the sequence, $\frac{1}{|a|}, \frac{1}{|a|^2}, \cdots$, is not less than $1$ any more. Therefore, in the geometric slicing, we replace geometric sequences with arithmetic sequences.

$\bullet$ Geometric Slicing: We apply the slicing idea of Lemma~\ref{lem:slicing2} to get a finite-horizon problem. By putting $\alpha_{k_1}=\frac{1}{k-k_1}$, $\alpha_{k_1+1}=\frac{1}{k-k_1}$, $\cdots$, $\alpha_k=\frac{1}{k-k_1}$ and $\beta_{k_2}=\frac{1}{k-k_2}$, $\beta_{k_2+1}= \frac{1}{k-k_2}$, $\cdots$, $\beta_{k-1}=\frac{1}{k-k_2}$ the average cost is lower bounded by
\begin{align}
&\inf_{u_1, u_2} (q \mathbb{E}[x^2[k]] \\
&+ r_1 \underbrace{( \frac{1}{k-k_1}\mathbb{E}[u_1^2[k_1]]+ \cdots + \frac{1}{k-k_1} \mathbb{E}[u_1^2[k-1]])}_{:=\widetilde{P_1}} \\
&+ r_2 \underbrace{( \frac{1}{k-k_2}\mathbb{E}[u_2^2[k_2]]+ \cdots + \frac{1}{k-k_2} \mathbb{E}[u_2^2[k-1]])}_{:=\widetilde{P_2}} )
\end{align}

$\bullet$ Three stage division: As we did in the proof of Lemma~\ref{lem:aless22}, we divide the resulting finite-horizon problem into three time intervals --- information-limited interval, Witsenhausen's interval, power-limited interval. Define
\begin{align}
&W_1:=w[0]+ \cdots + w[k_1-2]\\
&W_2 := w[k_1-1]+\cdots+w[k_2-2] \\
&W_3 :=  w[k_2-1]+\cdots + w[k-2] \\
&U_{11}:=u_1[1]+ \cdots+ u_1[k_1-1] \\
&U_{21}:=u_2[1]+ \cdots+ u_2[k_1-1] \\
&U_{22} := u_2[k_1]+\cdots+u_2[k_2-1]\\
&U_1 := u_1[k_1]+ \cdots + u_1[k-1] \\
&U_2 := u_2[k_2]+ \cdots + u_2[k-1]\\
&X_1 := W_1 + U_{11} + U_{12} \\
&X_2 := W_2 + U_{22}
\end{align}
Like the proof of Lemma~\ref{lem:aless22}, $W_1, W_2, W_3$ represent the distortions of three intervals. $U_{11}$ and $U_{21}$ represent the first and second controller inputs in the information-limited interval. $U_1$ represents the remaining input of the first controller. $U_{22}$ and $U_2$ represent the second controller's input in Witsenhausen's and power-limited intervals respectively.

The goal of this proof is grouping control inputs, so that we reveal the effects of the controller inputs on the state and isolate their effects according to their characteristics.

$\bullet$ Power-Limited Inputs: We first isolate the power-limited inputs, i.e. the first controller's input in the Witsenhausen's and power-limited interval, and the second controller's input in the power-limited interval. Notice that
\begin{align}
x[k]&=w[k-1]+w[k-2]+ \cdots + w[0] \\
&\quad +u_1[k-1]+u_1[k-2]+ \cdots + u_1[0] \\
&\quad +u_2[k-1]+u_2[k-2]+ \cdots + u_2[0] \\
&=
(w[0]+ \cdots + w[k_1-2] \\
&\quad+u_1[1]+ \cdots+ u_1[k_1-1] \\
&\quad+u_2[1]+ \cdots+ u_2[k_1-1]) \\
&\quad +(
w[k_1-1]+\cdots+w[k_2-2] \\
&\quad+u_2[k_1]+\cdots+u_2[k_2-1])\\
&\quad + (
w[k_2-1]+\cdots + w[k-2]
)\\
&\quad + (
u_1[k_1]+ \cdots + u_1[k-1]
) \\
&\quad + (
u_2[k_2]+ \cdots + u_2[k-1]
) \\
&\quad+w[k-1].
\end{align}
Therefore, by \cite[Lemma~1]{Park_Approximation_Journal_Parti} we have
\begin{align}
\mathbb{E}[x^2[k]]&=
\mathbb{E}[(X_1+X_2+W_3+U_1+U_2+w[k-1])^2] \\
&=\mathbb{E}[(X_1+X_2+W_3+U_1+U_2)^2]+\mathbb{E}[w^2[k-1]] \\
&\geq (\sqrt{\mathbb{E}[(X_1+X_2+W_3)^2]}-\sqrt{\mathbb{E}[U_1^2]}-\sqrt{\mathbb{E}[U_2^2]})_+^2 + 1 \\
&= (\sqrt{\mathbb{E}[(X_1+X_2)^2]+\mathbb{E}[W_3^2]}-\sqrt{\mathbb{E}[U_1^2]}-\sqrt{\mathbb{E}[U_2^2]})_+^2 + 1 \label{eqn:eq111}
\end{align}
where the last equality follows from the causality. Here, we can see that $\mathbb{E}[(X_1+X_2)^2]$ does not depend on the power-limited inputs.

$\bullet$ Information-Limited Interval: We will bound the remaining state distortion after the information-limited interval. Denote $y_1'$ and $y_2'$ as follows:
\begin{align}
&y_1'[k]= w[0]+ w[1]+ \cdots + w[k-1]+ v_1[k]\\
&y_2'[k]= w[0]+ w[1]+ \cdots + w[k-1]+ v_2[k]
\end{align}
Here, $y_1'[k]$, $y_2'[k]$ can be obtained by removing $u_1[1:k-1]$, $u_2[1:k-1]$ from $y_1[k]$, $y_2[k]$, and $u_1[k]$ and $u_2[k]$ are functions of $y_1[1:k]$ and $y_2[1:k]$ respectively. Therefore, we can see that $y_1[1:k]$, $y_2[1:k]$ are functions of $y_1'[1:k]$, $y_2'[1:k]$. Moreover, $W_1$, $y_1'[1:k_1-1]$, $y_2'[1:k_1-1]$ are jointly Gaussian.

Let
\begin{align}
&W_1' := W_1 - \mathbb{E}[W_1|y_1'[1:k_1-1],y_2'[1:k_1-1]]\\
&W_1'' := \mathbb{E}[W_1|y_1'[1:k_1-1],y_2'[1:k_1-1]].
\end{align}
Then, $W_1'$, $W_1''$, $W_2$ are independent Gaussian random variables. Moreover, $W_1', W_2$ are independent from $y_1'[1:k_1-1],y_2'[1:k_1-1]$. $W_1''$ is a function of $y_1'[1:k_1-1],y_2'[1:k_1-1]$.

Now, let's lower bound $\mathbb{E}[(X_1+X_2)^2]$. Since Gaussian maximizes the entropy, we have
\begin{align}
&\frac{1}{2}\log( 2 \pi e \mathbb{E}[(X_1+X_2)^2] \\
&\geq h(X_1+X_2)\\
&\geq h(X_1+X_2|y_1'[1:k_1-1],y_2'[1:k_1-1],y_2[k_1:k_2-1])\\
&= h(W_1'+W_1''+U_{11}+U_{12}+W_2+U_{22}|y_1'[1:k_1-1],y_2'[1:k_1-1],y_2[k_1:k_2-1])\\
&= h(W_1'+W_2|y_1'[1:k_1-1],y_2'[1:k_1-1],y_2[k_1:k_2-1])  \label{eqn:eq13}
\end{align}
We will first lower bound the variance of $W_1'$. Notice that
\begin{align}
\mathbb{E}[y_1'[k]^2]&= \mathbb{E}[w^2[0]] + \cdots + \mathbb{E}[w^2[k-1]] + \mathbb{E}[v_1^2[k]]=k + \sigma_{v1}^2\\
\end{align}
and
\begin{align}
\mathbb{E}[y_2'[k]^2]&=k + \sigma_{v2}^2.
\end{align}

Thus, we have
\begin{align}
&I(W_1;y_1'[1:k_1-1],y_2'[1:k_1-1]) \\
&= h(y_1'[1:k_1-1],y_2'[1:k_1-1])- h(y_1'[1:k_1-1],y_2'[1:k_1-1]|W_1)\\
&\leq \sum_{1 \leq i \leq k_1-1} h(y_1'[i]) + \sum_{1 \leq i \leq k_1-1} h(y_2'[i])
-\sum_{1 \leq i \leq k_1-1} h(v_1[i]) - \sum_{1 \leq i \leq k_1-1} h(v_2[i])\\
&\leq
\sum_{1 \leq k \leq k_1 -1} \frac{1}{2} \log(\frac{ k + \sigma_{v1}^2}{\sigma_{v1}^2}) +
\sum_{1 \leq k \leq k_1 -1} \frac{1}{2} \log(\frac{ k + \sigma_{v2}^2}{\sigma_{v2}^2}) \\
& = \frac{1}{2} \log(\prod_{1 \leq k \leq k_1 -1} \frac{ k + \sigma_{v1}^2}{\sigma_{v1}^2} )+
\frac{1}{2} \log(\prod_{1 \leq k \leq k_1 -1} \frac{ k + \sigma_{v2}^2}{\sigma_{v2}^2} )\\
& \overset{(A)}{\leq}
\frac{k_1-1}{2} \log( \frac{1}{k_1-1}\sum_{1 \leq k \leq k_1 -1} \frac{ k + \sigma_{v1}^2}{\sigma_{v1}^2} )+
\frac{k_1-1}{2} \log( \frac{1}{k_1-1}\sum_{1 \leq k \leq k_1 -1} \frac{ k + \sigma_{v2}^2}{\sigma_{v2}^2} )\\
& =
\frac{k_1-1}{2} \log( 1+ \frac{1}{k_1-1}\sum_{1 \leq k \leq k_1 -1} \frac{ k }{\sigma_{v1}^2} )+
\frac{k_1-1}{2} \log( 1+ \frac{1}{k_1-1}\sum_{1 \leq k \leq k_1 -1} \frac{ k }{\sigma_{v2}^2} )\\
&\leq
\frac{k_1-1}{2} \log( 1+ \frac{1}{k_1-1}\sum_{1 \leq k \leq k_1 -1} \frac{ k_1-1}{\sigma_{v1}^2} )+
\frac{k_1-1}{2} \log( 1+ \frac{1}{k_1-1}\sum_{1 \leq k \leq k_1 -1} \frac{ k_1-1}{\sigma_{v2}^2} )\\
&=
\frac{k_1-1}{2} \log( 1+ \frac{1}{(k_1-1)\sigma_{v1}^2}
(k_1-1)^2)+
\frac{k_1-1}{2} \log( 1+ \frac{1}{(k_1-1)\sigma_{v2}^2}
(k_1-1)^2)\\
&=
\frac{k_1-1}{2} \log( 1+ \frac{k_1-1}{\sigma_{v1}^2})+
\frac{k_1-1}{2} \log( 1+ \frac{k_1-1}{\sigma_{v2}^2}) \label{eqn:eq11}
\end{align}
(A): Arithmetic-Geometric mean.

Let's denote the last equation as $I$. We also have 
\begin{align}
\mathbb{E}[W_1^2] &= k_1-1 \label{eqn:eq12}
\end{align}
Now, we can bound the variance of the Gaussian random variable $W_1'$ as follows:
\begin{align}
&\frac{1}{2}\log(2\pi e \mathbb{E}[W_1'^2]) = h(W_1') \\
&\geq h(W_1'|y_1'[1:k_1-1],y_2'[1:k_1-1]) \\
&= h(W_1|y_1'[1:k_1-1],y_2'[1:k_1-1]) \\
&= h(W_1) - I(W_1;y_1'[1:k_1-1],y_2'[1:k_1-1]) \\
&\geq \frac{1}{2} \log(2 \pi e (k_1-1)) - I
\end{align}
where the last inequality follows from \eqref{eqn:eq11} and \eqref{eqn:eq12}.

Thus, 
\begin{align}
\mathbb{E}[W_1'^2] \geq \frac{k_1-1}{2^{2I}} \label{eqn:eq1100}
\end{align}
and denote the last term as $\Sigma$. Since $W_1'$ is Gaussian, $W_1'=W_1'''+W_1''''$ where $W_1'''\sim \mathcal{N}(0,\Sigma)$, and $W_1''',W_1''''$ are independent.

Moreover, we also have
\begin{align}
\mathbb{E}[W_2^2] = \mathbb{E}[(w[k_1-1]+ \cdots + w[k_2-2])^2] = k_2 - k_1.  \label{eqn:eq1101}
\end{align}

By \eqref{eqn:eq13}, we have
\begin{align}
&\frac{1}{2} \log(2 \pi e \mathbb{E}[(X_1+X_2)^2]) \\
&\geq h(W_1'+W_2| y_1'[1:k_1-1],y_2'[1:k_1-1],y_2[k_1:k_2-1])\\
&\geq h(W_1'+W_2| W_1'''', y_1'[1:k_1-1],y_2'[1:k_1-1],y_2[k_1:k_2-1])\\
&= h(W_1'''+W_2| W_1'''', y_1'[1:k_1-1],y_2'[1:k_1-1],y_2[k_1:k_2-1])\\
&= h(W_1'''+W_2| W_1'''', y_1'[1:k_1-1],y_2'[1:k_1-1])\\
& - I(W_1'''+W_2 ; y_2[k_1:k_2-1] | W_1'''', y_1'[1:k_1-1],y_2'[1:k_1-1])\\
&= h(W_1'''+W_2)\\
& - I(W_1'''+W_2 ; y_2[k_1:k_2-1] | W_1'''', y_1'[1:k_1-1],y_2'[1:k_1-1]) \\
&\geq \frac{1}{2} \log(2 \pi e(\Sigma+ k_2-k_1))\\
& - I(W_1'''+W_2 ; y_2[k_1:k_2-1] | W_1'''', y_1'[1:k_1-1],y_2'[1:k_1-1]) 
\label{eqn:eq16}
\end{align}
where the last inequality comes from the fact that $W_1'''$ and $W_2$ are independent Gaussian, and \eqref{eqn:eq1100}, \eqref{eqn:eq1101}. Now, the question boils down to the upper bound of the last mutual information term, which can be understood as the information contained in the second controller's observation in Witsenhausen's interval.

$\bullet$ Second controller's observation in Witsenhausen's interval: We will bound the amount of information contained in the second controller's observation in Witsenhausen's interval. For $n \geq k_1$, define
\begin{align}
y_2''[n]&:=W_1''' + w[k_1-1] + w[k_1]+ \cdots + w[n-1]\\
&+u_1[k_1]+ \cdots + u_1[n-1]\\
&+v_2[n]
\end{align}
Notice the relationship between $y_2[n]$ and $y_2''[n]$:
\begin{align}
y_2[n]=y_2''[n]+u_2[k_1]+ \cdots + u_2[n-1]+ W_1'''+ \mathbb{E}[W_1|y_1'[1:k_1-1], y_2'[1:k_1-1]].
\label{eqn:eq1110}
\end{align}
The mutual information of \eqref{eqn:eq16} is bounded as follows:
\begin{align}
&I(W_1'''+W_2; y_2[k_1:k_2-1] | W_1'''',y_1'[1:k_1-1], y_2'[1:k_1-1])\\
&= h(y_2[k_1:k_2-1]|  W_1'''',y_1'[1:k_1-1], y_2'[1:k_1-1])\\
&- h(y_2[k_1:k_2-1]|  W_1'''+W_2,W_1'''',y_1'[1:k_1-1], y_2'[1:k_1-1]) \\
&= \sum_{k_1 \leq i \leq k_2-1} h(y_2[i]|y_2[k_1:i-1],W_1'''',y_1'[1:k_1-1], y_2'[1:k_1-1])\\
&- \sum_{k_1 \leq i \leq k_2-1} h(y_2[i]|y_2[k_1:i-1],W_1'''+W_2,W_1'''',y_1'[1:k_1-1], y_2'[1:k_1-1]) \\
&\overset{(A)}{=} \sum_{k_1 \leq i \leq k_2-1} h(y_2''[i]|y_2[k_1:i-1],W_1'''',y_1'[1:k_1-1], y_2'[1:k_1-1])\\
&- \sum_{k_1 \leq i \leq k_2-1} h(y_2[i]|y_2[k_1:i-1],W_1'''+W_2,W_1'''',y_1'[1:k_1-1], y_2'[1:k_1-1]) \\
&\overset{(B)}{\leq} \sum_{k_1 \leq i \leq k_2-1} h(y_2''[i]) - \sum_{k_1 \leq i \leq k_2-1} h(v_2[i]) \\
&\leq \sum_{k_1 \leq i \leq k_2-1} \frac{1}{2} \log (2 \pi e \mathbb{E}[y_2''[i]^2])- \sum_{k_1 \leq i \leq k_2-1} \frac{1}{2} \log (2 \pi e \sigma_{v2}^2) \label{eqn:eq14}
\end{align}
(A): Since $y_2[1:k_1-1]$ is a function of $y_2'[1:k_1-1]$, $u_2[k_1], \cdots, u_2[i]$ are functions of $y_2[k_1:i-1], y_2'[1:k_1-1]$. Thus, all the terms in \eqref{eqn:eq1110} except $y_2''[i]$ can be vanished by the conditioning\\
(B): By causality, $v_2[i]$ is independent from all conditioning random variables.

First, let's bound the variance of $y_2''[n]$. By \cite[Lemma~1]{Park_Approximation_Journal_Parti}, we have
\begin{align}
\mathbb{E}[y_2''[n]^2] &\leq 2 \mathbb{E}[(W_1'''+w[k_1-1] + w[k_1]+ \cdots + w[n-1])^2] \\
&+ 2 \mathbb{E}[(u_1[k_1]+ \cdots + u_1[n-1])^2] + \sigma_{v2}^2 \\
&= 2 (\Sigma + n-k_1+1 )  + 2 \mathbb{E}[(u_1[k_1]+ \cdots + u_1[n-1])^2] + \sigma_{v2}^2.
\end{align}
By \cite[Lemma~1]{Park_Approximation_Journal_Parti}, we have
\begin{align}
&\mathbb{E}[(u_1[k_1]+ \cdots + u_1[n-1])^2] \\
&\leq (\sqrt{\mathbb{E}[u_1^2[k_1]]} + \cdots + \sqrt{\mathbb{E}[u_1^2[n-1]]})^2 \\
&\overset{(A)}{\leq} (n-k_1)(\mathbb{E}[u_1^2[k_1]]+  \mathbb{E}[u_1^2[k_1+1]]+ \cdots +  \mathbb{E}[u_1^2[n-1]]) \\
&\leq (n-k_1)(k-k_1) \widetilde{P_1}.
\end{align}
(A): Cauchy-Schwarz inequality

Thus, the variance of $y_2''[n]$ is bounded as:
\begin{align}
\mathbb{E}[y_2''[n]^2] \leq 2 \Sigma + 2(n-k_1+1) + 2(n-k_1)(k-k_1) \widetilde{P_1} + \sigma_{v2}^2
\end{align}
Therefore, we have
\begin{align}
&\sum_{k_1 \leq n \leq k_2-1} \mathbb{E}[y_2''[n]^2] \\
&\leq \sum_{k_1 \leq n \leq k_2-1}
(2 \Sigma + 2(n-k_1+1) + 2(n-k_1) (k-k_1) \widetilde{P_1} + \sigma_{v2}^2)\\
&\leq
2(k_2-k_1)\Sigma
+\sum_{k_1 \leq n \leq k_2 -1} 2(k_2-k_1)+\sum_{k_1 \leq n \leq k_2 -1} 2(k_2-k_1-1)(k-k_1) \widetilde{P_1} + (k_2-k_1) \sigma_{v2}^2 \\
&=
2(k_2-k_1)\Sigma+2(k_2-k_1)^2+2(k_2-k_1)(k_2-k_1-1)(k-k_1) \widetilde{P_1} + (k_2-k_1) \sigma_{v2}^2 \label{eqn:eq15}
\end{align}
Therefore, by \eqref{eqn:eq14} and \eqref{eqn:eq15} we conclude
\begin{align}
&I(W_1'''+W_2; y_2[k_1:k_2-1] | W'''', y_1'[1:k_1-1], y_2'[1:k_1-1])\\
&\leq \sum_{k_1 \leq n \leq k_2-1} \frac{1}{2} \log( \frac{ \mathbb{E}[y_2''[n]^2] }{\sigma_{v2}^2} ) \\
&= \frac{1}{2} \log( \prod_{k_1 \leq n \leq k_2-1} \frac{ \mathbb{E}[y_2''[n]^2] }{\sigma_{v2}^2} ) \\
&\overset{(A)}{\leq} \frac{k_2-k_1}{2} \log( \frac{1}{k_2-k_1} \sum_{k_1 \leq n \leq k_2-1} \frac{ \mathbb{E}[y_2''[n]^2] }{\sigma_{v2}^2} ) \\
&\leq \frac{k_2-k_1}{2} \log(
1+\frac{1}{(k_2-k_1)\sigma_{v2}^2}(2(k_2-k_1)\Sigma+2(k_2-k_1)^2+2(k_2-k_1)(k_2-k_1-1)(k-k_1) \widetilde{P_1} )
)\\
&\leq \frac{k_2-k_1}{2} \log(
1+\frac{1}{\sigma_{v2}^2}(2\Sigma+2(k_2-k_1)+2(k_2-k_1-1)(k-k_1) \widetilde{P_1} )
)
\end{align}
(A): Arithmetic-Geometric mean

Denote the last equation as $I'(\widetilde{P_1})$. By \eqref{eqn:eq16}, we can conclude
\begin{align}
&\frac{1}{2}\log( 2 \pi e \mathbb{E}[(X_1+X_2)^2] \geq \frac{1}{2} \log( 2 \pi e( \Sigma + k_2-k_1)) - I'(\widetilde{P_1})
\end{align}
which implies
\begin{align}
\mathbb{E}[(X_1+X_2)^2] \geq \frac{ \Sigma + k_2-k_1  }{2^{2I'(\widetilde{P_1})}}. \label{eqn:eq17}
\end{align}

$\bullet$ Final lower bound: Now, we can merge the inequalities to prove the lemma. The variance of $W_3$ is
\begin{align}
\mathbb{E}[W_3^2] &= k-k_2. \label{eqn:eq18}
\end{align}
By \cite[Lemma~1]{Park_Approximation_Journal_Parti} and Cauchy-Schwarz inequality, the variance of $U_1$ is upper bounded as follows:
\begin{align}
\mathbb{E}[U_1^2]
& \leq (\sqrt{a^{2(k-k_1-1)}\mathbb{E}[u_1^2[k_1]]}+\cdots+\sqrt{\mathbb{E}[u_1^2[k-1]]})^2 \\
& \leq
(k-k_1)(\mathbb{E}[u_1^2[k_1]]+\mathbb{E}[u_1^2[k_1+1]]+\cdots+\mathbb{E}[u_1^2[k-1]]) \\
&= (k-k_1)^2 \widetilde{P_1}. \label{eqn:eq19}
\end{align}
Likewise, the variance of $U_2$ can be bounded as
\begin{align}
\mathbb{E}[U_2^2] \leq (k-k_2)^2 \widetilde{P_2}. \label{eqn:eq110}
\end{align}
By plugging \eqref{eqn:eq17}, \eqref{eqn:eq18}, \eqref{eqn:eq19}, \eqref{eqn:eq110} into \eqref{eqn:eq111}, we finally prove the lemma.
\end{proof}

\begin{proof}[Proof of Corollary~\ref{cor:4} of Page~\pageref{cor:4}]
Proof of (a):

Since $\sigma_{v2} \geq 16$, we can find $k_2 \geq 6$ such that
\begin{align}
k_2 -2 \leq  \frac{\sigma_{v2}}{4} < k_2-1 \label{eqn:eq1add1}
\end{align}
We put such $k_3$, $k_1=1$ and $k=k_2$ as the parameters of Lemma~\ref{lem:eq11}.
Then, the lower bound of Lemma~\ref{lem:eq11} reduces to 
\begin{align}
D_L(\widetilde{P_1},\widetilde{P_2}) \geq ( \sqrt{ \frac{k_2-1}{2^{2I'(\widetilde{P_1})}}} - \sqrt{(k_2-1)^2 \widetilde{P_1}})_+^2  +1. \label{eqn:eq1add4}
\end{align}
Since $k_2 \geq 6$, we have
\begin{align}
\frac{k_2-2}{k_2-1} \geq \frac{4}{5}. \label{eqn:eq112}
\end{align}
Thus, $I'(\widetilde{P_1})$ is lower bounded by
\begin{align}
I'(\widetilde{P_1})&= \frac{1}{2} \log(1+\frac{1}{\sigma_{v2}^2}(2(k_2-1) + 2(k_2-2)(k_2-1) \widetilde{P_1}))^{k_2-1} \\
&= \frac{1}{2} \log(1+\frac{1}{k_2-1}( \frac{2(k_2-1)^2}{\sigma_{v2}^2} + \frac{2(k_2-2)(k_2-1)^2 \widetilde{P_1}}{\sigma_{v2}^2}))^{k_2-1} \\
&\overset{(A)}{\leq} \frac{1}{2} \log(1+\frac{1}{k_2-1}( 2(\frac{5}{4})^2\frac{(k_2-2)^2}{\sigma_{v2}^2} + 2(\frac{5}{4})^2\frac{(k_2-2)(k_2-2)^2 }{4\sigma_{v2}^3}))^{k_2-1} \\
&\overset{(B)}{\leq} \frac{1}{2} \log(1+\frac{1}{k_2-1}( 2(\frac{5}{4})^2(\frac{1}{4})^2 + 2(\frac{5}{4})^2(\frac{1}{4})^3))^{k_2-1} \\
&\leq \frac{1}{2} \log e^{\frac{125}{512}} \label{eqn:eq1add2}
\end{align}
(A): \eqref{eqn:eq112} and $\widetilde{P_1} \leq \frac{1}{4 \sigma_{v2}}$. \\
(B): \eqref{eqn:eq1add1}.

Moreover, we have
\begin{align}
(k_2-1)^2 \widetilde{P_1} &\overset{(A)}{\leq} \frac{5}{4}(k_2-1) (k_2-2) \widetilde{P_1}\\
& \overset{(B)}{\leq} \frac{5}{4}(k_2-1) \frac{k_2-2}{4\sigma_{v2}} \\
& \overset{(C)}{\leq} \frac{5}{4}(k_2-1) \frac{1}{16} \\
& = \frac{5}{64} (k_2-1) \label{eqn:eq1add3}
\end{align}
(A): \eqref{eqn:eq112}\\
(B): $\widetilde{P_1} \leq \frac{1}{4 \sigma_{v2}}$ \\
(C): \eqref{eqn:eq1add1}

Therefore, by plugging \eqref{eqn:eq1add2}, \eqref{eqn:eq1add3} into \eqref{eqn:eq1add4}, we get
\begin{align}
D_L(\widetilde{P_1},\widetilde{P_2}) &\geq ( \sqrt{ \frac{k_2-1}{2^{2I'(P_1)}}} - \sqrt{(k_2-1)^2 \widetilde{P_1}})_+^2  +1  \\
& \geq ( \sqrt{ \frac{k_2-1}{e^{\frac{125}{512}}}} - \sqrt{\frac{5}{64}(k_2-1)} )_+^2 + 1 \\
& = 0.366724...(k_2-1) +1 \\
& \geq 0.366724...\frac{\sigma_{v2}}{4} +1 \\
& = 0.09168106... \sigma_{v2} +1
\end{align}
where the last inequality follows from \eqref{eqn:eq1add1}.

Proof of (b):

Since $\frac{1}{\widetilde{P_1}} \geq 64$, we can find $k_2 \geq 6$ such that
\begin{align}
k_2 -2 \leq  \frac{1}{16 \widetilde{P_1}} < k_2-1. \label{eqn:eq1add20}
\end{align}
We put such $k_2$, $k_1=1$ and $k=k_2$ as the parameters of Lemma~\ref{lem:eq11}. Then, the lower bound of Lemma~\ref{lem:eq11} reduces to
\begin{align}
D_L(\widetilde{P_1},\widetilde{P_2})\geq ( \sqrt{ \frac{k_2-1}{2^{2I'(\widetilde{P_1})}}} - \sqrt{(k_2-1)^2 \widetilde{P_1}})_+^2  +1. \label{eqn:eq1add23}
\end{align}
First, $I'(\widetilde{P_1})$ is lower bounded by
\begin{align}
I'(\widetilde{P_1})&= \frac{1}{2} \log(1+\frac{1}{\sigma_{v2}^2}(2(k_2-1) + 2(k_2-2)(k_2-1)\widetilde{P_1}))^{k_2-1} \\
&= \frac{1}{2} \log(1+\frac{1}{k_2-1}( \frac{2(k_2-1)^2}{\sigma_{v2}^2} + \frac{2(k_2-2)(k_2-1)^2 \widetilde{P_1}}{\sigma_{v2}^2}))^{k_2-1} \\
&\overset{(A)}{\leq} \frac{1}{2} \log(1+\frac{1}{k_2-1}( 2(k_2-1)^2 (4\widetilde{P_1})^2 + 2(k_2-2)(k_2-1)^2 \widetilde{P_1} (4\widetilde{P_1})^2))^{k_2-1} \\
&\overset{(B)}{\leq} \frac{1}{2} \log(1+\frac{1}{k_2-1}( 2(\frac{5}{4})^2(k_2-2)^2 (4 \widetilde{P_1})^2 + 2(\frac{5}{4})^2(k_2-2)^3 \widetilde{P_1} (4\widetilde{P_1})^2))^{k_2-1} \\
&\overset{(C)}{\leq} \frac{1}{2} \log(1+\frac{1}{k_2-1}( 2(\frac{5}{4})^2 (\frac{1}{4})^2 + 2(\frac{5}{4})^2 \frac{1}{16} (\frac{1}{4})^2 )^{k_2-1}\\
&\leq \frac{1}{2} \log e^{\frac{425}{2048}}. \label{eqn:eq1add21}
\end{align}
(A): $\frac{1}{4 \sigma_{v2}} \leq P_1$\\
(B): Since $k_2 \geq 6$, \eqref{eqn:eq112} still holds.\\
(C): \eqref{eqn:eq1add20}

Moreover, we also have
\begin{align}
(k_2-1)^2 \widetilde{P_1} &\overset{(A)}{\leq} \frac{5}{4}(k_2-1) (k_2-2) \widetilde{P_1}\\
& \overset{(B)}{\leq} \frac{5}{4}(k_2-1) \frac{1}{16} \\
& = \frac{5}{64} (k_2-1). \label{eqn:eq1add22}
\end{align}
(A): Since $k_2 \geq 6$, \eqref{eqn:eq112} still holds. \\
(B): \eqref{eqn:eq1add20}

Therefore, plugging \eqref{eqn:eq1add21}, \eqref{eqn:eq1add22} into \eqref{eqn:eq1add23} we can conclude
\begin{align}
D_L(\widetilde{P_1}, \widetilde{P_2}) &\geq ( \sqrt{ \frac{k_2-1}{2^{2I'(\widetilde{P_1})}}} - \sqrt{(k_2-1)^2 \widetilde{P_1}})_+^2  +1  \\
& \geq ( \sqrt{ \frac{k_2-1}{e^{\frac{425}{2048}}}} - \sqrt{\frac{5}{64}(k_2-1)} )_+^2 + 1 \\
& = 0.386801...(k_2-1) +1 \\
& \geq 0.386801...\frac{1}{16 \widetilde{P_1}} +1 \\
& = \frac{0.0241750...}{16 \widetilde{P_1}}  +1
\end{align}
where the last inequality comes from \eqref{eqn:eq1add20}.

Proof of (c):

Denote $P:=\max(\widetilde{P_1},\widetilde{P_2})$. Since $P \leq \frac{1}{50}$, there exists $k \geq 3$ such that
\begin{align}
k-2 \leq \frac{1}{50 P} < k-1. \label{eqn:eq1add30}
\end{align}
We put such $k$ and $k_1=k_2=1$ as the parameters of Lemma~\ref{lem:eq11}. Then, the lower bound of Lemma~\ref{lem:eq11} reduces to
\begin{align}
D_L(\widetilde{P_1},\widetilde{P_2}) \geq (\sqrt{k-1} - \sqrt{(k-1)^2 \widetilde{P_1}} - \sqrt{(k-1)^2 \widetilde{P_2}})_+^2 +1.
\end{align}
Since $k \geq 3$, we have
\begin{align}
\frac{k-2}{k-1} \geq \frac{1}{2}. \label{eqn:eq1add31}
\end{align}
Therefore, we conclude
\begin{align}
D_L(\widetilde{P_1},\widetilde{P_2}) &\geq (\sqrt{k-1} - \sqrt{4(k-1)^2 P})^2_+ +1 \\
&\overset{(A)}{\geq} (\sqrt{k-1} - \sqrt{16(k-2)^2 P})^2_+ +1 \\
&\overset{(B)}{\geq} (\sqrt{\frac{1}{50P}} - \sqrt{\frac{16}{50^2 P}})^2_+ +1 \\
&\geq 0.00377258...\frac{1}{P} +1.
\end{align}
(A): \eqref{eqn:eq1add30}\\
(B): \eqref{eqn:eq1add31}

Proof of (d):

As mentioned in the proof of Corollary~\ref{cor:2} (j), the centralized controller's distortion that has both observations $y_1[n], y_2[n]$ and has no input power constraints is a lower bound on the decentralized controller's distortion.

Let $y_1'[n]:=x[n]+v_1'[n]$ and $y_2'[n]:=x[n]+v_2'[n]$ where $v_1'[n] \sim \mathcal{N}(0,\sigma_1^2)$ and $v_2'[n] \sim \mathcal{N}(0,\sigma_1^2)$ are i.i.d. random variables. Just like the proof of Corollary~\ref{cor:2} (j), the performance of the centralized controller with both observations is equivalent to a centralized controller with observation $\frac{y_1'[n]+y_2'[n]}{2}$  by the maximum ratio combining.

Let $\Sigma_E$ be the estimation error of the Kalman filtering with a scalar observation $\frac{y_1'[n]+y_2'[n]}{2}$. By Lemma~\ref{lem:aless21},
\begin{align}
\Sigma_E &= \frac{-1 + \sqrt{4 \frac{\sigma_{v1}^2}{2}+1}}{2}\\
&=\frac{-1 + \sqrt{2 \sigma_{v1}^2 + 1}}{2}.
\end{align}
Then, for all $\widetilde{P_1}$ and $\widetilde{P_2}$, the cost of the decentralized controllers is lower bounded as follows:
\begin{align}
D_L(\widetilde{P_1},\widetilde{P_2}) &\overset{(A)}{\geq} \inf_{|1-k| < 1} \frac{(2k-k^2)\Sigma_E + 1}{1-(1-k)^2} \\
&= \inf_{|1-k| < 1} \Sigma_E + \frac{1}{1-(1-k)^2}\\
&\geq \Sigma_E.
\end{align}
(A): The decentralized control cost is larger than the centralized controller's cost with the observation $\frac{y_1'[n]+y_2'[n]}{2}$. Moreover, when $|a-k| \geq 1$ the centralized control system is unstable, and the cost diverges to infinity. When $|a-k| < 1$, the cost analysis follows from Lemma~\ref{lem:aless21}.

By Lemma~\ref{lem:eq11}, $D_L(\widetilde{P_1},\widetilde{P_2}) \geq 1$. Finally, for all $\widetilde{P_1},\widetilde{P_2}$ we have
\begin{align}
D_L(\widetilde{P_1},\widetilde{P_2}) & \geq \max(\Sigma_E,1) = \max( \frac{-1 + \sqrt{2 \sigma_{v1}^2 + 1}}{2} , 1) \\
& \geq \frac{1}{2} ( \frac{-1 + \sqrt{2 \sigma_{v1}^2 + 1}}{2}) + \frac{1}{2} \\
& \geq \frac{1}{4} + \frac{\sqrt{2 \sigma_{v1}^2 + 1}}{2} \\
& \geq \frac{\sqrt{2}}{2} \sigma_{v1}.
\end{align}
Since we already know $D_L(\widetilde{P_1},\widetilde{P_2}) \geq 1$, the statement (d) of the corollary is true.
\end{proof}

\begin{proof}[Proof of Proposition~\ref{prop:2} of Page~\pageref{prop:2}]
Like the proof of Proposition~\ref{prop:1}, we define the subscript $max$ as $argmax_{i \in \{ 1,2\}} \widetilde{P_i}$, and write $D_{\sigma_{v1}}(\cdot), D_{\sigma_{v2}}(\cdot), D_{\sigma_{v max}}(\cdot)$ as $D_{v1}(\cdot), D_{v2}(\cdot), D_{v max}(\cdot)$ respectively.

By the same argument as the proof of Proposition~\ref{prop:1}, it is enough to show that there exists $c \leq 10^6$ such that for all $\widetilde{P_1}, \widetilde{P_2} \geq 0$, $\min(D_{\sigma 1}(c \widetilde{P_1}), D_{\sigma 2}(c \widetilde{P_2})) \leq c \cdot D_L(\widetilde{P_1}, \widetilde{P_2})$.

In the proof, we first divide the cases based on $\sigma_1 , \sigma_2$ (essentially equivalent to $\Sigma_1$, $\Sigma_2$), and then based on $\widetilde{P_1}, \widetilde{P_2}$. Here, we can use the fact that $\sigma_1 \leq \sigma_2$ to reduce the cases.

(i) When $\sigma_{v1} \leq 16$, $\sigma_{v2} \leq 16$

(i-i) If $\max(\widetilde{P_1}, \widetilde{P_2}) \geq \frac{1}{64}$

Lower bound: By Corollary~\ref{cor:4} (d),
\begin{align}
D_L(\widetilde{P_1}, \widetilde{P_2}) \geq 1.
\end{align}

Upper bound: Since $\sigma_{v1}, \sigma_{v2} \leq 16$, we can plug $t=\frac{1}{15.008}$ into the equation \eqref{eqn:cor:32} of Corollary~\ref{cor:3}. Thus, we have
\begin{align}
(D_{\sigma max}(P_{max}),P_{max}) \leq ( 30.016, \frac{1}{15.008}).
\end{align}

Ratio: $c$ is upper bounded by
\begin{align}
c \leq 30.016.
\end{align}

(i-ii) If $\max(\widetilde{P_1}, \widetilde{P_2}) \leq \frac{1}{64}$

Lower bound: Since $\max(\widetilde{P_1}, \widetilde{P_2}) \leq \frac{1}{64} \leq \frac{1}{50}$, by Corollary~\ref{cor:4} (c) we can conclude
\begin{align}
D_L(\widetilde{P_1}, \widetilde{P_2}) \geq \frac{0.003772}{\max(\widetilde{P_1}, \widetilde{P_2})}+1.
\end{align}

Upper bound: Since $\sigma_{v1}, \sigma_{v2} \leq 16$ and $\widetilde{P_{max}} \leq \frac{1}{64} \leq \frac{1}{15.008}$, we can plug $t=\widetilde{P_{max}}$ into the equation \eqref{eqn:cor:32} of  Corollary~\ref{cor:3}. Thus, we have
\begin{align}
(D_{\sigma max}(P_{max}),P_{max}) \leq ( \frac{2}{\widetilde{P_{max}}}, \widetilde{P_{max}}).
\end{align}

Ratio: $c$ is upper bounded by
\begin{align}
c \leq \frac{2}{0.003772} < 540.
\end{align}

(ii) When $ \sigma_{v1} \leq 16 \leq \sigma_{v2}$

(ii-i) If $\widetilde{P_1} \geq \frac{1}{64}$

Lower bound: By Corollary~\ref{cor:4} (d), we have
\begin{align}
D_L(\widetilde{P_1}, \widetilde{P_2}) \geq 1.
\end{align}

Upper bound: Since $\sigma_1 \leq 16$, we can plug $t=\frac{1}{15.008}$ into the equation \eqref{eqn:cor:32} of Corollary~\ref{cor:3}. Thus, we have
\begin{align}
(D_{\sigma 1}(P_1),P_1) \leq (30.016, \frac{1}{15.008}).
\end{align}

Ratio: $c$ is upper bounded by
\begin{align}
c \leq 30.016.
\end{align}

(ii-ii) If $\frac{1}{4 \sigma_{v2}} \leq \widetilde{P_1} \leq \frac{1}{64}$

Lower bound: By Corollary~\ref{cor:4} (b), we have
\begin{align}
D_L(\widetilde{P_1}, \widetilde{P_2}) \geq \frac{0.02417}{\widetilde{P_1}} +1.
\end{align}

Upper bound: Since $\widetilde{P_1} \leq \frac{1}{64} \leq \frac{1}{15.008}$, we can plug $t=\widetilde{P_1}$ into the equation \eqref{eqn:cor:32} of Corollary~\ref{cor:3}. Thus, we have
\begin{align}
(D_{\sigma 1}(P_1), P_1) \leq (\frac{2}{\widetilde{P_1}}, \widetilde{P_1}).
\end{align}

Ratio: $c$ is upper bounded by
\begin{align}
c \leq \frac{2}{0.02417} < 83.
\end{align}

(ii-iii) If $\widetilde{P_1} \leq \frac{1}{4 \sigma_{v2}}$ and $\widetilde{P_2} \geq \frac{1}{4 \sigma_{v2}}$

Lower bound: By Corollary~\ref{cor:4} (a), we have
\begin{align}
D_L(\widetilde{P_1}, \widetilde{P_2}) \geq 0.09168 \sigma_{v2} + 1
\end{align}

Upper bound: Since $\sigma_{v2} \geq 16$, we can plug $t=\frac{1}{1.0005\sigma_{v2}}$ into the equation \eqref{eqn:cor:31} of Corollary~\ref{cor:3}. Thus, we have
\begin{align}
(D_{\sigma 2}(\widetilde{P_2}), \widetilde{P_2}) \leq (2.001 \sigma_{v2}, \frac{1}{1.0005 \sigma_{v2}}).
\end{align}

Ratio: $c$ is upper bounded by
\begin{align}
c \leq \frac{2.001}{0.09168} < 22.
\end{align}

(ii-iv) If $\widetilde{P_1} \leq \frac{1}{4 \sigma_{v2}}$ and $\widetilde{P_2} \leq \frac{1}{4 \sigma_{v2}}$

Lower bound: Since $\widetilde{P_1} \leq \frac{1}{4 \sigma_{v2}} \leq \frac{1}{64} \leq \frac{1}{50}$, $\widetilde{P_2} \leq \frac{1}{4 \sigma_{v2}} \leq \frac{1}{64} \leq \frac{1}{50}$, by Corollary~\ref{cor:4} (c) we have
\begin{align}
D_L(\widetilde{P_1}, \widetilde{P_2}) \geq \frac{0.003772}{\max(\widetilde{P_1}, \widetilde{P_2})} + 1.
\end{align}

Upper bound: Since $\widetilde{P_1} \leq \frac{1}{4 \sigma_{v2}} \leq \frac{1}{64} \leq \frac{1}{15.008}$, $\widetilde{P_2} \leq \frac{1}{4\sigma_{v2}} \leq \frac{1}{1.0005\sigma_{v2}}$, these satisfies the conditions for \eqref{eqn:cor:31}, \eqref{eqn:cor:32} of Corollary~\ref{cor:3} respectively. Therefore, by plugging $t=\widetilde{P_{max}}$ into Corollary~\ref{cor:3}, we have
\begin{align}
(D_{\sigma max}(P_{max}), P_{max}) \leq (\frac{2}{\widetilde{P_{max}}}, \widetilde{P_{max}}).
\end{align}

Ratio: $c$ is upper bounded by
\begin{align}
c \leq \frac{2}{0.003772} < 540.
\end{align}

(iii) When $\sigma_{v1} \geq 16$ and $\sigma_{v2} \geq 16$

(iii-i) If $\widetilde{P_1} \geq \frac{1}{4 \sigma_{v1}}$

Lower bound: By Corollary~\ref{cor:4} (d), we have
\begin{align}
D_L(\widetilde{P_1}, \widetilde{P_2}) \geq \frac{\sqrt{2}}{2} \sigma_{v1}.
\end{align}

Upper bound: Since $\sigma_{v1} \geq 16$, we can plug $t=\frac{1}{1.0005 \sigma_{v1}}$ into \eqref{eqn:cor:31} of Corollary~\ref{cor:3}. Thus, we have
\begin{align}
(D_{\sigma 1}(P_1),P_1 ) \leq (2.001 \sigma_{v1}, \frac{1}{1.0005 \sigma_{v1}}).
\end{align}

Ratio: $c$ is upper bounded by
\begin{align}
c \leq \frac{4}{1.0005} < 4.
\end{align}

(iii-ii) If $\frac{1}{4 \sigma_{v2}} \leq \widetilde{P_1} \leq \frac{1}{4 \sigma_{v1}}$

Lower bound: Since $\frac{1}{4 \sigma_{v2}} \leq \widetilde{P_1} \leq \frac{1}{4 \sigma_{v1}} \leq \frac{1}{64}$, by Corollary~\ref{cor:4} (b) we have
\begin{align}
D_L(\widetilde{P_1}, \widetilde{P_2}) \geq \frac{0.02417}{\widetilde{P_1}} + 1.
\end{align}

Upper bound: Since $\frac{1}{\widetilde{P_1}} \leq \frac{1}{4\sigma_{v1}} \leq \frac{1}{1.0005\sigma_{v1}}$, we can plug $t=\widetilde{P_1}$ into the equation \eqref{eqn:cor:31} of Corollary~\ref{cor:3}. Thus, we have
\begin{align}
(D_{\sigma 1}(P_1),P_1) \leq ( \frac{2}{\widetilde{P_1}}, \widetilde{P_1}).
\end{align}

Ratio: $c$ is upper bounded by
\begin{align}
c \leq \frac{2}{0.02417} < 83.
\end{align}

(iii-iii) If $\widetilde{P_1} \leq \frac{1}{4 \sigma_{v2}}$ and $\widetilde{P_2} \geq \frac{1}{4 \sigma_{v2}}$

Lower bound: By Corollary~\ref{cor:4} (a), we have
\begin{align}
D_L(\widetilde{P_1}, \widetilde{P_2}) \geq 0.09168 \sigma_{v2} + 1.
\end{align}

Upper bound: Since $\sigma_{v2} \geq 16$, we can plug $t=\frac{1}{1.0005 \sigma_{v2}}$ into the equation \eqref{eqn:cor:31} of Corollary~\ref{cor:3}. Thus, we have
\begin{align}
(D_{\sigma 2}(P_2), P_2) \leq (2.001\sigma_{v2}, \frac{1}{1.0005\sigma_{v2}}).
\end{align}

Ratio: $c$ is upper bounded by
\begin{align}
c \leq \frac{2.001}{0.09168} < 22.
\end{align}

(iii-iv) If $\widetilde{P_1} \leq \frac{1}{4 \sigma_{v2}}$ and $\widetilde{P_2} \leq \frac{1}{4 \sigma_{v2}}$

Lower bound: Since $\widetilde{P_1} \leq \frac{1}{4\sigma_{v2}}\leq \frac{1}{64} \leq \frac{1}{50}$, $\widetilde{P_2} \leq \frac{1}{4\sigma_{v2}}\leq \frac{1}{64} \leq \frac{1}{50}$, by Corollary~\ref{cor:4} (c) we have
\begin{align}
D_L(\widetilde{P_1}, \widetilde{P_2}) \geq \frac{0.003772}{\max(\widetilde{P_1}, \widetilde{P_2})} +1.
\end{align}

Upper bound: Since $\widetilde{P_1} \leq \frac{1}{4 \sigma_{v2}} \leq \frac{1}{4 \sigma_{v1}} \leq \frac{1}{1.0005\sigma_{v1}}$ and $\widetilde{P_2} \leq \frac{1}{4 \sigma_{v2}} \leq \frac{1}{1.0005 \sigma_{v2}}$, we can plug $t=\widetilde{P_{max}}$ into the equation \eqref{eqn:cor:31} of Corollary~\ref{cor:3}. Thus, we have
\begin{align}
(D_{\sigma max}(P_{max}), P_{max}) \leq ( \frac{2}{\widetilde{P_{max}}}, \widetilde{P_{max}}).
\end{align}

Ratio: $c$ is upper bounded by
\begin{align}
c \leq \frac{2}{0.003772} < 540.
\end{align}

Finally, by (i), (ii), (iii), the constant $c$ is upper bounded by $10^6$ and the proposition is proved.
\end{proof}

\subsection{Proof of Lemma~\ref{lem:less11}, Corollary~\ref{cor:3} and Proposition~\ref{prop:2}}
\label{sec:aless1}

\begin{proof}[Proof of Lemma~\ref{lem:less11} of Page~\pageref{lem:less11}]
For simplicity, we assume $0 \leq a < 1$, $1<k_1 < k_2 < k$. The remaining case when $-1 < a  \leq 0$ or $k_1=1$ or $k_1=k_2$ or $k=k_2$ easily follow with minor modifications.

$\bullet$ Geometric Slicing: We apply the geometric slicing idea of Lemma~\ref{lem:slicing2} to get a finite-horizon problem. By putting $\alpha_{k_1}=( \frac{1-a}{1-a^{k-k_1}} ) a^{k-k_1-1}$, $\alpha_{k_1+1}=( \frac{1-a}{1-a^{k-k_1}} ) a^{k-k_1-2}$, $\cdots$, $\alpha_k=\frac{1-a}{1-a^{k-k_1}}$ and $\beta_{k_2}=( \frac{1-a^{-1}}{1-a^{-(k-k_2)}} )$, $\beta_{k_2+1}=( \frac{1-a^{-1}}{1-a^{-(k-k_2)}} )a^{-1}$, $\cdots$, $\beta_{k-1}=( \frac{1-a^{-1}}{1-a^{-(k-k_2)}} )a^{-k+1+k_2}$ the average cost is lower bounded by
\begin{align}
&q \mathbb{E}[x^2[k]] \\
&+r_1 \underbrace{(( \frac{1-a}{1-a^{k-k_1}} ) a^{k-k_1-1} \mathbb{E}[u_1^2[k_1]]+( \frac{1-a}{1-a^{k-k_1}} ) a^{k-k_1-2} \mathbb{E}[u_1^2[k_1+1]]+\cdots+
( \frac{1-a}{1-a^{k-k_1}} ) \mathbb{E}[u_1^2[k-1]])}_{:=\widetilde{P_1}}\\
&+r_2 \underbrace{(( \frac{1-a}{1-a^{k-k_2}} ) a^{k-k_1-1} \mathbb{E}[u_2^2[k_2]]+( \frac{1-a}{1-a^{k-k_2}} ) a^{k-k_1-2}\mathbb{E}[u_2^2[k_2+1]]+\cdots+
( \frac{1-a}{1-a^{k-k_2}} ) \mathbb{E}[u_2^2[k-1]])}_{:=\widetilde{P_2}})
\end{align}
Here, we denote the second and third terms as $\widetilde{P_1}$ and $\widetilde{P_2}$ respectively.

$\bullet$ Three stage division: As we did in the proof of Lemma~\ref{lem:aless22}, we will divide the finite-horizon problem into three time intervals --- information-limited interval, Witsenhausen's interval, power-limited interval. Define 
\begin{align}
&W_1:=a^{k-1}w[0]+ \cdots + a^{k-k_1+1}w[k_1-2]\\
&W_2 := a^{k-k_1}w[k_1-1]+\cdots+a^{k-k_2+1}w[k_2-2] \\
&W_3 := a^{k-k_2} w[k_2-1]+\cdots + a w[k-2] \\
&U_{11}:=a^{k-2}u_1[1]+ \cdots+ a^{k-k_1} u_1[k_1-1] \\
&U_{21}:=a^{k-2}u_2[1]+ \cdots+ a^{k-k_1} u_2[k_1-1] \\
&U_1 := a^{k-k_1-1}u_1[k_1]+ \cdots + u_1[k-1] \\
&U_{22} := a^{k-k_1-1}u_2[k_1]+\cdots+a^{k-k_2}u_2[k_2-1])\\
&U_2 := a^{k-k_2-1}u_2[k_2]+ \cdots + u_2[k-1]\\
&X_1 := W_1 + U_{11} + U_{21} \\
&X_2 := W_2 + U_{22}
\end{align}
$W_1, W_2, W_3$ represent the distortions of three intervals respectively. $U_{11}$ and $U_{21}$ represent the first and second controller inputs in the information-limited interval respectively. $U_1$ represent the remaining input of the first controller. $U_{22}$ and $U_2$ represent the second controller's input in Witsenhausen's and power-limited intervals respectively.

The goal of this proof is grouping control inputs and expanding $x[n]$, so that we reveal the effects of the controller inputs on the state and isolate their effects according to their characteristics.

$\bullet$ Power-Limited Inputs: We will first isolate the power limited inputs, i.e. the first controller's input in the Witsenhausen's and power-limited intervals, and the second controller's input in the power-limited interval. Notice that
\begin{align}
x[k]&=w[k-1]+aw[k-2]+ \cdots + a^{k-1}w[0] \\
&\quad +u_1[k-1]+au_1[k-2]+ \cdots + a^{k-1}u_1[0] \\
&\quad +u_2[k-1]+au_2[k-2]+ \cdots + a^{k-1}u_2[0] \\
&=
(a^{k-1}w[0]+ \cdots + a^{k-k_1+1}w[k_1-2] \\
&\quad+a^{k-2}u_1[1]+ \cdots+ a^{k-k_1} u_1[k_1-1] \\
&\quad+a^{k-2}u_2[1]+ \cdots+ a^{k-k_1} u_2[k_1-1]) \\
&\quad +(
a^{k-k_1}w[k_1-1]+\cdots+a^{k-k_2+1}w[k_2-2] \\
&\quad+a^{k-k_1-1}u_2[k_1]+\cdots+a^{k-k_2}u_2[k_2-1])\\
&\quad + (
a^{k-k_2} w[k_2-1]+\cdots + a w[k-2]
)\\
&\quad + (
a^{k-k_1-1}u_1[k_1]+ \cdots + u_1[k-1]
) \\
&\quad + (
a^{k-k_2-1}u_2[k_2]+ \cdots + u_2[k-1]
) \\
&\quad+w[k-1].
\end{align}
Therefore, by \cite[Lemma~1]{Park_Approximation_Journal_Parti} we have
\begin{align}
\mathbb{E}[x^2[k]]&=
\mathbb{E}[(X_1+X_2+W_3+U_1+U_2+w[k-1])^2] \\
&=\mathbb{E}[(X_1+X_2+W_3+U_1+U_2)^2]+\mathbb{E}[w^2[k-1]] \\
&\geq (\sqrt{\mathbb{E}[(X_1+X_2+W_3)^2]}-\sqrt{\mathbb{E}[U_1^2]}-\sqrt{\mathbb{E}[U_2^2]})_+^2 + 1 \\
&= (\sqrt{\mathbb{E}[(X_1+X_2)^2]+\mathbb{E}[W_3^2]}-\sqrt{\mathbb{E}[U_1^2]}-\sqrt{\mathbb{E}[U_2^2]})_+^2 + 1  \label{eqn:less111}
\end{align}
where the last equality comes form the causality. Here, we can see that $\mathbb{E}[(X_1+X_2)^2]$ does not depend on the inputs from the power-limited intervals.

$\bullet$ Information-Limited Interval: We will bound the remaining state distortion after the information-limited interval. Define $y_1'$ and $y_2'$ as follows:
\begin{align}
&y_1'[k]= a^{k-1}w[0]+ a^{k-2}w[1]+ \cdots + w[k-1]+ v_1[k]\\
&y_2'[k]= a^{k-1}w[0]+ a^{k-2}w[1]+ \cdots + w[k-1]+ v_2[k].
\end{align}
Here, $y_1'[k]$, $y_2'[k]$ can be obtained by removing $u_1[1:k-1]$, $u_2[1:k-1]$ from $y_1[k]$, $y_2[k]$, and $u_1[k]$ and $u_2[k]$ are functions of $y_1[1:k]$ and $y_2[1:k]$ respectively. Therefore, we can see that $y_1[1:k], y_2[1:k]$ are functions of $y_1'[1:k], y_2'[1:k]$. Moreover, $W_1$, $y_1'[1:k_1-1]$, $y_2'[1:k_1-1]$ are jointly Gaussian.

Let
\begin{align}
&W_1' := W_1 - \mathbb{E}[W_1|y_1'[1:k_1-1],y_2'[1:k_1-1]]\\
&W_1'' := \mathbb{E}[W_1|y_1'[1:k_1-1],y_2'[1:k_1-1]].
\end{align}
Then, $W_1'$, $W_1''$, $W_2$ are independent Gaussian random variables. Moreover, $W_1', W_2$ are independent from $y_1'[1:k_1-1],y_2'[1:k_1-1]$. $W_1''$ is a function of $y_1'[1:k_1-1],y_2'[1:k_1-1]$.

Now, let's lower bound $\mathbb{E}[(X_1+X_2)^2]$. Since Gaussian maximizes the entropy, we have
\begin{align}
&\frac{1}{2}\log( 2 \pi e \mathbb{E}[(X_1+X_2)^2] \\
&\geq h(X_1+X_2)\\
&\geq h(X_1+X_2|y_1'[1:k_1-1],y_2'[1:k_1-1],y_2[k_1:k_2-1])\\
&= h(W_1'+W_1''+U_{11}+U_{12}+W_2+U_{22}|y_1'[1:k_1-1],y_2'[1:k_1-1],y_2[k_1:k_2-1])\\
&= h(W_1'+W_2|y_1'[1:k_1-1],y_2'[1:k_1-1],y_2[k_1:k_2-1]). \label{eqn:less13}
\end{align}
We will first lower bound the variance of $W_1'$. Notice that
\begin{align}
\mathbb{E}[y_1'[k]^2]&=a^{2(k-1)}+ a^{2(k-2)} + \cdots + 1 + \sigma_{v1}^2 \\
&= \frac{1-a^{2k}}{1-a^2} + \sigma_{v1}^2
\end{align}
and
\begin{align}
\mathbb{E}[y_2'[k]^2]&=a^{2(k-1)}+ a^{2(k-2)} + \cdots + 1 + \sigma_{v1}^2 \\
&= \frac{1-a^{2k}}{1-a^2} + \sigma_{v2}^2.
\end{align}
Thus, we have
\begin{align}
&I(W_1;y_1'[1:k_1-1],y_2'[1:k_1-1])  \\
&= h(y_1'[1:k_1-1],y_2'[1:k_1-1])-h(y_1'[1:k_1-1],y_2'[1:k_1-1]| W_1) \\
&\leq \sum_{1 \leq i \leq k_1-1} h(y_1'[i]) + \sum_{1 \leq i \leq k_1-1} h(y_2'[i]) - \sum_{1 \leq i \leq k_1-1} h(v_1[i]) - \sum_{1 \leq i \leq k_1-1} h(v_2[i])\\
&\leq \sum_{1 \leq k \leq k_1 -1} \frac{1}{2} \log(\frac{ \frac{1-a^{2k}}{1-a^2} + \sigma_{v1}^2}{\sigma_{v1}^2}) +
\sum_{1 \leq k \leq k_1 -1} \frac{1}{2} \log(\frac{ \frac{1-a^{2k}}{1-a^2} + \sigma_{v2}^2}{\sigma_{v2}^2}) \\
& = \frac{1}{2} \log(\prod_{1 \leq k \leq k_1 -1} \frac{ \frac{1-a^{2k}}{1-a^2} + \sigma_{v1}^2}{\sigma_{v1}^2} )+
\frac{1}{2} \log(\prod_{1 \leq k \leq k_1 -1} \frac{ \frac{1-a^{2k}}{1-a^2} + \sigma_{v2}^2}{\sigma_{v2}^2} )\\
&\overset{(A)}{\leq}
\frac{k_1-1}{2} \log( \frac{1}{k_1-1}\sum_{1 \leq k \leq k_1 -1} \frac{ \frac{1-a^{2k}}{1-a^2} + \sigma_{v1}^2}{\sigma_{v1}^2} )+
\frac{k_1-1}{2} \log( \frac{1}{k_1-1}\sum_{1 \leq k \leq k_1 -1} \frac{ \frac{1-a^{2k}}{1-a^2} + \sigma_{v2}^2}{\sigma_{v2}^2} )\\
& =
\frac{k_1-1}{2} \log( 1+ \frac{1}{k_1-1}\sum_{1 \leq k \leq k_1 -1} \frac{ \frac{1-a^{2k}}{1-a^2} }{\sigma_{v1}^2} )+
\frac{k_1-1}{2} \log( 1+ \frac{1}{k_1-1}\sum_{1 \leq k \leq k_1 -1} \frac{ \frac{1-a^{2k}}{1-a^2} }{\sigma_{v2}^2} )\\
&\leq
\frac{k_1-1}{2} \log( 1+ \frac{1}{k_1-1}\sum_{1 \leq k \leq k_1 -1} \frac{ \frac{1-a^{2(k_1-1)}}{1-a^2} }{\sigma_{v1}^2} )+
\frac{k_1-1}{2} \log( 1+ \frac{1}{k_1-1}\sum_{1 \leq k \leq k_1 -1} \frac{ \frac{1-a^{2(k_1-1)}}{1-a^2} }{\sigma_{v2}^2} )\\
&=
\frac{k_1-1}{2} \log( 1+ \frac{1}{\sigma_{v1}^2} \frac{1-a^{2(k_1-1)}}{1-a^2}) +
\frac{k_1-1}{2} \log( 1+ \frac{1}{\sigma_{v2}^2} \frac{1-a^{2(k_1-1)}}{1-a^2}). \label{eqn:less11}
\end{align}
(A): Arithmetic-Geometric mean\\
Let's denote the last equation as $I$. We also have
\begin{align}
\mathbb{E}[W_1^2] &= a^{2(k-1)} + \cdots + a^{2(k-k_1 +1)} \\
&= a^{2(k-k_1+1)} (a^{2(k_1-2)}+\cdots + 1) \\
&= a^{2(k-k_1+1)} \frac{1-a^{2(k_1-1)}}{1-a^2}. \label{eqn:less12}
\end{align}
Now, we can bound the variance of a Gaussian random variable $W_1'$ as follows:
\begin{align}
&\frac{1}{2}\log(2\pi e \mathbb{E}[W_1'^2]) = h(W_1') \\
&\geq h(W_1' |y_1'[1:k_1-1],y_2'[1:k_1-1]) \\
&= h(W_1|y_1'[1:k_1-1],y_2'[1:k_1-1]) \\
&= h(W_1) - I(W_1;y_1'[1:k_1-1],y_2'[1:k_1-1])\\
&\geq \frac{1}{2}\log( 2 \pi e a^{2(k-k_1+1)} \frac{1-a^{2(k_1-1)}}{1-a^2} ) - I
\end{align}
where the last inequality follows from \eqref{eqn:less11} and \eqref{eqn:less12}. 

Thus,
\begin{align}
\mathbb{E}[W_1'^2] \geq \frac{ a^{2(k-k_1+1)}\frac{1-a^{2(k_1-1)}}{1-a^2} }{2^{2I}} \label{eqn:less1100}
\end{align}
and denote the last term as $\Sigma$. Since $W_1'$ is Gaussian, we can write $W_1'=W_1'''+W_1''''$ where $W_1'''\sim \mathcal{N}(0,\Sigma)$, and $W_1''',W_1''''$ are independent.

Moreover, we also have
\begin{align}
\mathbb{E}[W_2^2] &= a^{2(k-k_1)} + \cdots + a^{2(k-k_2+1)}\\
&= a^{2(k-k_2+1)}\frac{1-a^{2(k_2-k_1)}}{1-a^2}. \label{eqn:less1101}
\end{align}
By \eqref{eqn:less13}, we have
\begin{align}
&\frac{1}{2} \log(2 \pi e \mathbb{E}[(X_1+X_2)^2]) \\
&\geq h(W_1'+W_2| y_1'[1:k_1-1],y_2'[1:k_1-1],y_2[k_1:k_2-1])\\
&\geq h(W_1'+W_2| W_1'''', y_1'[1:k_1-1],y_2'[1:k_1-1],y_2[k_1:k_2-1])\\
&= h(W_1'''+W_2| W_1'''', y_1'[1:k_1-1],y_2'[1:k_1-1],y_2[k_1:k_2-1])\\
&= h(W_1'''+W_2| W_1'''', y_1'[1:k_1-1],y_2'[1:k_1-1])\\
& - I(W_1'''+W_2 ; y_2[k_1:k_2-1] | W_1'''', y_1'[1:k_1-1],y_2'[1:k_1-1])\\
&= h(W_1'''+W_2)\\
& - I(W_1'''+W_2 ; y_2[k_1:k_2-1] | W_1'''', y_1'[1:k_1-1],y_2'[1:k_1-1])\\
&\geq \frac{1}{2} \log(2 \pi e (\Sigma+ a^{2(k-k_2+1)}\frac{1-a^{2(k_2-k_1)}}{1-a^2}))\\
& - I(W_1'''+W_2 ; y_2[k_1:k_2-1] | W_1'''', y_1'[1:k_1-1],y_2'[1:k_1-1]) \label{eqn:less16}
\end{align}
where the last inequality comes from the fact that $W_1'''$ and $W_2$ are independent Gaussian, and \eqref{eqn:less1100} and \eqref{eqn:less1101}.

Now, the question boils down to the upper bound of the last mutual information term, which can be understood as the information contained in the second controller's observation in Witsenhausen's interval.

$\bullet$ Second controller's observation in Witsenhausen's interval: We will bound the amount of information contained in the second controller's observation in Witsenhausen's interval. For $n \geq k_1$, define
\begin{align}
y_2''[n]&:=a^{n-k} W_1''' + a^{n-k_1} w[k_1-1] + a^{n-k_1-1}w[k_1]+ \cdots + w[n-1]\\
&+a^{n-k_1-1}u_1[k_1]+ \cdots + u_1[n-1]+v_2[n].
\end{align}
Notice the relationship between $y_2[n]$ and $y_2''[n]$ is
\begin{align}
y_2[n]&=y_2''[n]+a^{n-k_1-1}u_2[k_1]+ \cdots + u_2[n-1]\\
&+a^{n-k}W_1'''' + a^{n-k} \mathbb{E}[W_1 | y_1'[1:k_1-1], y_2'[1:k_1-1]]. \label{eqn:less1200}
\end{align}

The mutual information in \eqref{eqn:less16} is bounded as follows:
\begin{align}
&I(W_1'''+W_2; y_2[k_1:k_2-1] | W_1'''',y_1'[1:k_1-1], y_2'[1:k_1-1])\\
&= h(y_2[k_1:k_2-1]|  W_1'''',y_1'[1:k_1-1], y_2'[1:k_1-1])\\
&- h(y_2[k_1:k_2-1]|  W_1'''+W_2,W_1'''',y_1'[1:k_1-1], y_2'[1:k_1-1]) \\
&= \sum_{k_1 \leq i \leq k_2-1} h(y_2[i]|y_2[k_1:i-1],W_1'''',y_1'[1:k_1-1], y_2'[1:k_1-1])\\
&- \sum_{k_1 \leq i \leq k_2-1} h(y_2[i]|y_2[k_1:i-1],W_1'''+W_2,W_1'''',y_1'[1:k_1-1], y_2'[1:k_1-1]) \\
&\overset{(A)}{=} \sum_{k_1 \leq i \leq k_2-1} h(y_2''[i]|y_2[k_1:i-1],W_1'''',y_1'[1:k_1-1], y_2'[1:k_1-1])\\
&- \sum_{k_1 \leq i \leq k_2-1} h(y_2[i]|y_2[k_1:i-1],W_1'''+W_2,W_1'''',y_1'[1:k_1-1], y_2'[1:k_1-1]) \\
&\overset{(B)}{\leq} \sum{k_1 \leq i \leq k_2-1} h(y_2''[i]) - \sum_{k_1 \leq i \leq k_2-1} h(v_2[i])\\
&\leq \sum_{k_1 \leq i \leq k_2-1} \frac{1}{2} \log (2 \pi e \mathbb{E}[y_2''[i]^2])- \sum_{k_1 \leq i \leq k_2-1} \frac{1}{2} \log (2 \pi e \sigma_{v2}^2) \label{eqn:less14}
\end{align}
(A): Since $y_2[1:k_1-1]$ is a function of $y_2'[1:k_1-1]$, $u_2[k_1], \cdots, u_2[i]$ are functions of $y_2[k_1:i-1], y_2'[1:k_1-1]$. Thus, all the terms in \eqref{eqn:less1200} except $y_2''[i]$ can be vanished by the conditioned.\\
(B): By causality, $v_2[i]$ is independent from all conditioning random variables. 

First, let's bound the variance of $y_2''[n]$. By \cite[Lemma~1]{Park_Approximation_Journal_Parti}, we have
\begin{align}
\mathbb{E}[y_2''[n]^2] &\leq 2 \mathbb{E}[(a^{n-k}W_1'''+a^{n-k_1} w[k_1-1] + a^{n-k_1-1}w[k_1]+ \cdots + w[n-1])^2] \\
&+ 2 \mathbb{E}[(a^{n-k_1-1}u_1[k_1]+ \cdots + u_1[n-1])^2] + \sigma_{v2}^2 \\
&= 2 (a^{2(n-k)} \Sigma + a^{2(n-k_1)} + \cdots + 1 )  \\
&+ 2 \mathbb{E}[(a^{n-k_1-1}u_1[k_1]+ \cdots + u_1[n-1])^2] + \sigma_{v2}^2
\end{align}
Here, by \cite[Lemma~1]{Park_Approximation_Journal_Parti}, we have
\begin{align}
&\mathbb{E}[(a^{n-k_1-1}u_1[k_1]+ \cdots + u_1[n-1])^2] \\
&\leq (\sqrt{ a^{2(n-k_1-1)} \mathbb{E}[u_1^2[k_1]]} + \cdots + \sqrt{\mathbb{E}[u_1^2[n-1]]})^2 \\
&\overset{(A)}{\leq} (a^{(n-k_1-1)} + a^{(n-k_1-2)}+ \cdots + 1 )
(a^{(n-k_1-1)}\mathbb{E}[u_1^2[k_1]]+a^{(n-k_1-2)}\mathbb{E}[u_1^2[k_1+1]]+ \cdots + \mathbb{E}[u_1^2[n-1]]  ) \\
&= \frac{1-a^{n-k_1}}{1-a} \cdot a^{n-k}
(a^{k-k_1-1} \mathbb{E}[u_1^2[k_1]]+ a^{k-k_1-2} \mathbb{E}[u_1^2[k_1+1]] + \cdots + a^{k-n}\mathbb{E}[u_1^2[n-1]])\\
&\leq \frac{1-a^{n-k_1}}{1-a} \cdot a^{n-k} \frac{1-a^{k-k_1}}{1-a} \widetilde{P_1} \\
&= a^{n-k} \frac{(1-a^{n-k})(1-a^{k-k_1})}{(1-a)^2} \widetilde{P_1}.
\end{align}
(A): Cauchy-Schwarz inequality

Thus, the variance of $y_2''[n]$ is bounded as:
\begin{align}
\mathbb{E}[y_2''[n]^2] \leq 2 a^{2(n-k)} \Sigma + 2 \frac{1-a^{2(n-k_1+1)}}{1-a^2}
+ 2 a^{n-k} \frac{(1-a^{n-k})(1-a^{k-k_1})}{(1-a)^2} \widetilde{P_1}+ \sigma_{v2}^2.
\end{align}
Therefore, we have
\begin{align}
&\sum_{k_1 \leq n \leq k_2-1} \mathbb{E}[y_2''[n]^2] \\
&\leq \sum_{k_1 \leq n \leq k_2-1}
2 a^{2(n-k)} \Sigma + 2 \frac{1-a^{2(n-k_1+1)}}{1-a^2}
+ 2 a^{n-k} \frac{(1-a^{n-k})(1-a^{k-k_1})}{(1-a)^2} \widetilde{P_1}+ \sigma_{v2}^2\\
&\leq
2(a^{2(k_1-k)}+ \cdots + a^{2(k_2-1-k)})\Sigma
+\sum_{k_1 \leq n \leq k_2 -1} 2 \frac{1-a^{2(k_2-1-k_1+1)}}{1-a^2} \\
&+\sum_{k_1 \leq n \leq k_2 -1} 2a^{n-k} \frac{(1-a^{k_2-1-k_1})(1-a^{k-k_1})}{(1-a)^2} \widetilde{P_1}
+(k_2-k_1) \sigma_{v2}^2 \\
&\leq
2a^{2(k_1-k)} \frac{1-a^{2(k_2-k_1)}}{1-a^{2}} \Sigma
+ 2 (k_2-k_1)\frac{1-a^{2(k_2-1-k_1+1)}}{1-a^2} \\
&+ 2 a^{k_1-k} \frac{1-a^{k_2-k_1}}{1-a}\frac{(1-a^{k_2-1-k_1})(1-a^{k-k_1})}{(1-a)^2} \widetilde{P_1}
+ (k_2 - k_1) \sigma_{v2}^2.  \label{eqn:less15}
\end{align}
Therefore, by \eqref{eqn:less14} and \eqref{eqn:less15} we conclude
\begin{align}
&I(W_1'''+W_2; y_2[k_1:k_2-1] | W'''', y_1'[1:k_1-1], y_2'[1:k_1-1])\\
&\leq \sum_{k_1 \leq n \leq k_2-1} \frac{1}{2} \log( \frac{ \mathbb{E}[y_2''[n]^2] }{\sigma_{v2}^2} ) \\
&= \frac{1}{2} \log( \prod_{k_1 \leq n \leq k_2-1} \frac{ \mathbb{E}[y_2''[n]^2] }{\sigma_{v2}^2} ) \\
&\overset{(A)}{\leq} \frac{k_2-k_1}{2} \log( \frac{1}{k_2-k_1} \sum_{k_1 \leq n \leq k_2-1} \frac{ \mathbb{E}[y_2''[n]^2] }{\sigma_{v2}^2} ) \\
&\leq \frac{k_2-k_1}{2} \log(
1+\frac{1}{(k_2-k_1)\sigma_{v2}^2}(2a^{2(k_1-k)} \frac{1-a^{2(k_2-k_1)}}{1-a^{2}} \Sigma
+ 2 (k_2-k_1)\frac{1-a^{2(k_2-1-k_1+1)}}{1-a^2} \\
&+ 2 a^{k_1-k} \frac{1-a^{k_2-k_1}}{1-a}\frac{(1-a^{k_2-1-k_1})(1-a^{k-k_1})}{(1-a)^2} \widetilde{P_1} )
)
\end{align}
(A): Arithmetic-Geometric mean

Denote the last equation as $I'(\widetilde{P_1})$. By \eqref{eqn:less16} we conclude
\begin{align}
&\frac{1}{2}\log( 2 \pi e \mathbb{E}[(X_1+X_2)^2] \\
&\geq h(W_1'''+W_2)- I( W_1'''+W_2; y_2[k_1:k_2-1]|W_1'''',y_1'[1:k_1-1],y_2'[1:k_1-1]) \\
&\geq \frac{1}{2} \log( 2 \pi e( \Sigma + a^{2(k-k_2+1)}\frac{1-a^{2(k_2-k_1)}}{1-a^2})) - I'(\widetilde{P_1})
\end{align}
which implies
\begin{align}
\mathbb{E}[(X_1+X_2)^2] \geq \frac{ \Sigma + a^{2(k-k_2+1)}\frac{1-a^{2(k_2-k_1)}}{1-a^2}  }{2^{2I'(\widetilde{P_1})}}. \label{eqn:less110}
\end{align}

$\bullet$ Final lower bound: Now, we can merge the inequalities to prove the lemma. The variance of $W_3$ is given as follows:
\begin{align}
\mathbb{E}[W_3^2] &= a^{2(k-k_2)} + \cdots + a^2 \\
&= a^2 \frac{1-a^{2(k-k_2)}}{1-a^2}. \label{eqn:less17}
\end{align}

By \cite[Lemma~1]{Park_Approximation_Journal_Parti}, the variance of $U_1$ is bounded as follows:
\begin{align}
\mathbb{E}[U_1^2]
& \leq (\sqrt{a^{2(k-k_1-1)}\mathbb{E}[u_1^2[k_1]]}+\cdots+\sqrt{\mathbb{E}[u_1^2[k-1]]})^2 \\
& \overset{(A)}{\leq} (a^{(k-k_1-1)}+a^{(k-k_1-2)}+ \cdots + 1)
( a^{(k-k_1-1)}\mathbb{E}[u_1^2[k_1]]+ a^{(k-k_1-2)}\mathbb{E}[u_1^2[k_1+1]]+\cdots+\mathbb{E}[u_1^2[k-1]] )\\
& =
\frac{1-a^{k-k_1}}{1-a}
( a^{(k-k_1-1)}\mathbb{E}[u_1^2[k_1]]+ a^{(k-k_1-2)}\mathbb{E}[u_1^2[k_1+1]]+\cdots+\mathbb{E}[u_1^2[k-1]] )\\
&=
\frac{1-a^{k-k_1}}{1-a}  \frac{1-a^{k-k_1}}{1-a} \widetilde{P_1} \\
&=
(\frac{1-a^{k-k_1}}{1-a} )^2 \widetilde{P_1}. \label{eqn:less18}
\end{align}
(A): Cauchy-Schwarz inequality

Likewise, the variance of $U_2$ can be bounded as
\begin{align}
\mathbb{E}[U_2^2] \leq (\frac{1-a^{k-k_2}}{1-a} )^2 \widetilde{P_2} \label{eqn:less19}
\end{align}
By plugging in \eqref{eqn:less110}, \eqref{eqn:less17}, \eqref{eqn:less18}, \eqref{eqn:less19} into \eqref{eqn:less111}, we finally prove the lemma.
\end{proof}

\begin{proof}[Proof of Corollary~\ref{cor:6} of Page~\pageref{cor:6}]
For simplicity, we will prove for $0.9 \leq a <1$. The proof for $-1 < a \leq -0.9$ is simply follows by replacing $a$ by $|a|$.

First, we can notice
\begin{align}
\Sigma_1 &=
\frac{(a^2-1){\sigma_{v1}^2} -1 + \sqrt{((a^2-1){\sigma_{v1}^2} -1)^2+4a^2 {\sigma_{v1}^2}}}{2a^2} \\
&=
\frac{4 a^2 {\sigma_{v1}^2}}{2a^2( -(a^2-1){\sigma_{v1}^2}+1+
\sqrt{
((a^2-1){\sigma_{v1}^2}-1)^2 + 4a^2 {\sigma_{v1}^2}
})} \\
&=
\frac{2 {\sigma_{v1}^2}}{ (1-a^2){\sigma_{v1}^2}+1+
\sqrt{
((1-a^2){\sigma_{v1}^2}+1)^2 + 4a^2 {\sigma_{v1}^2}
}}
\end{align}
Since $0.9 \leq a < 1$, $(1-a^2) \sigma_{v1}^2 \geq 0$. Thus, $\Sigma_1$ is upper bounded by
\begin{align}
\Sigma_1 &\leq \frac{2\sigma_{v1}^2}{\sqrt{4 a^2 \sigma_{v1}^2}} = \frac{\sigma_{v1}^2}{a \sigma_{v1}} = \frac{\sigma_{v1}}{a} \leq \frac{10}{9} \sigma_{v1}
\end{align}
and
\begin{align}
\Sigma_1 &\leq \frac{2\sigma_{v1}^2}{(1-a^2){\sigma_{v1}^2} + (1-a^2){\sigma_{v1}^2}} \\
&= \frac{1}{1-a^2}
\end{align}
Likewise, we also have
\begin{align}
\Sigma_2 \leq \frac{10}{9} \sigma_{v2} \label{eqn:less126}
\end{align}
and
\begin{align}
\Sigma_2 \leq \frac{1}{1-a^2} \label{eqn:less120}
\end{align}

Proof of (a):

Notice that by $\Sigma_2 \geq 40$, $0.9 \leq a <1$ and \eqref{eqn:less120} we have
\begin{align}
&\frac{\Sigma_2}{40} \leq \frac{1}{40} \frac{1}{1-a^2} < \frac{a^2}{1-a^2} \\
&\frac{\Sigma_2}{40} \geq 1 \geq a^2 = \frac{a^2-a^4}{1-a^2}
\end{align}
Thus, we can find $k \geq 3$ such that 
\begin{align}
\frac{a^2 - a^{2(k-1)}}{1-a^2} \leq \frac{\Sigma_2}{40} < \frac{a^2 - a^{2k}}{1-a^2} \label{eqn:less121}
\end{align}

Let's put such $k$ and $k_1=1, k_2=k$ as the parameters of Lemma~\ref{lem:less11}. Then, the lower bound of Lemma~\ref{lem:less11} reduces to 
\begin{align}
D_L(\widetilde{P_1},\widetilde{P_2}) \geq
(
\sqrt{\frac{\frac{a^2-a^{2k}}{1-a^2}}{2^{2I'(\widetilde{P_1})}}}
-
\sqrt{(\frac{1-a^{k-1}}{1-a})^2 \widetilde{P_1}}
)_+^2 +1 \label{eqn:less1add1}
\end{align}
where
\begin{align}
I'(\widetilde{P_1}) = \frac{1}{2} \log(1+
\frac{1}{(k-1)\sigma_{v2}^2}(
2(k-1)\frac{1-a^{2(k-1)}}{1-a^2} +
2a^{1-k} \frac{1-a^{k-1}}{1-a}\frac{(1-a^{k-2})(1-a^{k-1})}{(1-a)^2} \widetilde{P_1}
))^{k-1}.
\end{align}
Let's first upper bound $I'(\widetilde{P_1})$. By \eqref{eqn:less120} and \eqref{eqn:less121}, we first have
\begin{align}
&\frac{a^2-a^{2(k-1)}}{1-a^2} \leq \frac{\Sigma_2}{40} \leq \frac{1}{40} \frac{1}{1-a^2} \\
&(\Rightarrow) a^2 - a^{2(k-1)} \leq \frac{1}{40} \\
&(\Leftrightarrow) a^2 - \frac{1}{40} \leq a^{2(k-1)} \\
&(\Rightarrow) (\frac{9}{10})^2 - \frac{1}{40}  \leq a^{2(k-1)} (\because 0.9 \leq a < 1)\\
&(\Leftrightarrow) \frac{157}{200}  \leq a^{2(k-1)}. \label{eqn:leq1converse1}
\end{align}
We also have
\begin{align}
a^{2(k-1)}(k-1) \leq 1+a^2+a^4+ \cdots + a^{2(k-2)} = \frac{1-a^{2(k-1)}}{1-a^2} \label{eqn:leq1converse2}
\end{align}
where the first inequality comes from that $0.9 \leq a < 1$ so $a^{2(k-1)} \leq 1$, $\cdots$, $a^{2(k-1)} \leq a^{2(k-2)}$.

Therefore, by \eqref{eqn:leq1converse1} and \eqref{eqn:leq1converse2}
\begin{align}
&\frac{157}{200}(k-1) \leq \frac{1-a^{2(k-1)}}{1-a^2}\\
&(\Rightarrow) k-1 \leq \frac{200}{157} \frac{1-a^{2(k-1)}}{1-a^2} \label{eqn:less122}
\end{align}

Moreover, we also have
\begin{align}
\frac{1-a^{2(k-1)}}{a^2 - a^{2(k-1)}} &= \frac{1-a^2}{a^2 - a^{2(k-1)}} + \frac{a^2-a^{2(k-1)}}{a^2-a^{2(k-1)}}\\
&= \frac{1-a^2}{a^2 - a^{2(k-1)}} + 1\\
&\leq \frac{1-a^2}{a^2 - a^4} + 1 (\because k \geq 3) \\
&= \frac{1}{a^2} + 1 \\
&\leq (\frac{10}{9})^2 + 1 = \frac{181}{81} (\because 0.9 \leq a <1)
\end{align}
which implies
\begin{align}
\frac{1-a^{2(k-1)}}{1-a^2} \leq \frac{181}{81} \frac{a^2-a^{2(k-1)}}{1-a^2}. \label{eqn:less123}
\end{align}
We also have
\begin{align}
\frac{1-a^{k-1}}{1-a} &\leq \frac{1-a^{k-1}}{1-a} \frac{2}{1+a} (\because 0.9 \leq a < 1)\\
&\leq \frac{1-a^{2(k-2)}}{1-a} \frac{2}{1+a} (\because k \geq 3 \mbox{ implies } 2(k-2) \geq k-1) \\
&\leq \frac{1-a^{2(k-2)}}{1-a} \frac{2}{1+a} \frac{a^2}{0.9^2} (\because 0.9 \leq a < 1)\\
&= \frac{2}{0.9^2} \frac{a^2 - a^{2(k-1)}}{1-a^2}. \label{eqn:leq1converse5}
\end{align}

Therefore, the terms in $I'(\widetilde{P_1})$ are upper bounded as:
\begin{align}
2(k-1) \frac{1-a^{2(k-1)}}{1-a^2} & \leq 2 \frac{200}{157} (\frac{1-a^{2(k-1)}}{1-a^2})^2 (\because \eqref{eqn:less122}) \\
&\leq 2 \frac{200}{157} (\frac{181}{81}\frac{a^2 - a^{2(k-1)}}{1-a^2} )^2 (\because \eqref{eqn:less123})\\
&\leq 2 \frac{200}{157} (\frac{181}{81}\frac{\Sigma_2}{40} )^2 (\because \eqref{eqn:less121}) \label{eqn:less124}
\end{align}
and
\begin{align}
&2a^{1-k} \frac{1-a^{k-1}}{1-a} \frac{(1-a^{k-2})(1-a^{k-1})}{(1-a)^2} \widetilde{P_1}\\
&\leq 2 \sqrt{\frac{200}{157}} ( \frac{1-a^{k-1}}{1-a} )^3 \widetilde{P_1} (\because \eqref{eqn:leq1converse1} \mbox{ and } a^{k-1} \leq a^{k-2}) \\
&\leq 2 \sqrt{\frac{200}{157}} ( \frac{2}{0.9^2} (\frac{a^2 - a^{2(k-1)}}{1-a^2}))^3 \widetilde{P_1} (\because \eqref{eqn:leq1converse5})\\
&\leq 2 \sqrt{\frac{200}{157}} ( \frac{2}{0.9^2} \frac{\Sigma_2}{40} )^3 \frac{1}{\Sigma_2} (\because \eqref{eqn:less121} \mbox{ and the assumption }\widetilde{P_1} \leq \frac{1}{\Sigma_2}) \\
&= 2 \sqrt{\frac{200}{157}} ( \frac{5}{81 })^3 \Sigma_2^2 \label{eqn:less125}
\end{align}

Now, we can upper bound $I'(\widetilde{P_1})$ by
\begin{align}
I'(\widetilde{P_1}) & \overset{(A)}{\leq} \frac{1}{2} \log(1+\frac{\Sigma_2^2}{(k-1)\sigma_{v2}^2}
(
2 \frac{200}{157} (\frac{181}{81 \cdot 40})^2 +
2 \sqrt{\frac{200}{157}}(\frac{5}{81 })^3
))^{k-1}\\
& \overset{(B)}{\leq}
\frac{1}{2} \log(1+\frac{1}{(k-1)} (\frac{10}{9})^2
(
2 \frac{200}{157} (\frac{181}{81 \cdot 40})^2 +
2 \sqrt{\frac{200}{157}}(\frac{5}{81})^3
))^{k-1}\\
&\leq
\frac{1}{2} \log(1+\frac{0.010471667...}{k-1})^{k-1} \\
&\leq
\frac{1}{2} \log e^{0.01047} \label{eqn:less1add2}
\end{align}
(A): \eqref{eqn:less124}, \eqref{eqn:less125} \\
(B): \eqref{eqn:less126}

Moreover, we also have
\begin{align}
(\frac{1-a^{k-1}}{1-a})^2 \widetilde{P_1}  &\leq  ( \frac{2}{0.9^2} \frac{a^2 - a^{2(k-1)}}{1-a^2})^2 \widetilde{P_1} (\because \eqref{eqn:leq1converse5}) \\
&\leq  ( \frac{2}{0.9^2} \frac{\Sigma_2}{40})^2 \frac{1}{\Sigma_2}  (\because \eqref{eqn:less121} \mbox{ and the assumption }\widetilde{P_1} \leq \frac{1}{\Sigma_2})\\
&= (\frac{5}{81})^2 \Sigma_2. \label{eqn:less1add3}
\end{align}
Finally, by plugging \eqref{eqn:less1add2}, \eqref{eqn:less1add3} into \eqref{eqn:less1add1} we can conclude
\begin{align}
D_L(\widetilde{P_1},\widetilde{P_2}) &\geq ( \sqrt{\frac{\Sigma_2}{40 e^{0.01047}}} - \sqrt{(\frac{5}{81})^2 \Sigma_2})_+^2  +1   \\
&\geq  0.0091316992... \Sigma_2 +1.
\end{align}

Proof of (b):

Notice that by $\Sigma_2 \geq 40$, $\frac{1}{\Sigma_2} \leq \widetilde{P_1} \leq \frac{1}{40}$, \eqref{eqn:less120},
\begin{align}
&\frac{1}{40 \widetilde{P_1}} \leq \frac{\Sigma_2}{40} \leq \frac{1}{40} \frac{1}{1-a^2} < \frac{a^2}{1-a^2} \\
&\frac{1}{40 \widetilde{P_1}} \geq 1 \geq a^2 = \frac{a^2-a^4}{1-a^2}
\end{align}
Thus, we can find $k \geq 3$ such that 
\begin{align}
\frac{a^2 - a^{2(k-1)}}{1-a^2} \leq \frac{1}{40 \widetilde{P_1}} < \frac{a^2 - a^{2k}}{1-a^2} \label{eqn:less127}
\end{align}
Let's put such $k$ and $k_1=1$, $k_2=k$ as the parameters of Lemma~\ref{lem:less11}. Then, the lower bound of Lemma~\ref{lem:less11} reduces to
\begin{align}
D_L(\widetilde{P_1},\widetilde{P_2}) \geq
(
\sqrt{\frac{\frac{a^2-a^{2k}}{1-a^2}}{2^{2I'(\widetilde{P_1})}}}
-
\sqrt{(\frac{1-a^{k-1}}{1-a})^2 \widetilde{P_1}}
)_+^2 +1
\end{align}
First, we will upper bound $I'(\widetilde{P_1})$. Since $\frac{1}{40 \widetilde{P_1}} \leq \frac{\Sigma_2}{40}$, we still have \eqref{eqn:leq1converse1}, \eqref{eqn:leq1converse2}, \eqref{eqn:less122} which are
\begin{align}
&\frac{157}{200} \leq a^{2(k-1)} \label{eqn:less1add21}\\
& k-1 \leq \frac{200}{157} \frac{1-a^{2(k-1)}}{1-a^2}. \label{eqn:less1add23}
\end{align}
Since $k \geq 3$, we still have \eqref{eqn:less123} and \eqref{eqn:leq1converse5} which are
\begin{align}
& \frac{1-a^{2(k-1)}}{1-a^2} \leq \frac{181}{81} \frac{a^2 - a^{2(k-1)}}{1-a^2} \label{eqn:less1add24}\\
& \frac{1-a^{k-1}}{1-a} \leq \frac{2}{0.9^2} \frac{a^2 - a^{2(k-1)}}{1-a^2}. \label{eqn:less1add25}
\end{align}

Then, the terms in $I'(\widetilde{P_1})$ are upper bounded as:
\begin{align}
2(k-1) \frac{1-a^{2(k-1)}}{1-a^2} & \leq 2 \frac{200}{157} (\frac{1-a^{2(k-1)}}{1-a^2})^2 (\because \eqref{eqn:less1add23})\\
&\leq 2 \frac{200}{157} (\frac{181}{81}\frac{a^2 - a^{2(k-1)}}{1-a^2} )^2 (\because \eqref{eqn:less1add24})\\
&\leq 2 \frac{200}{157} (\frac{181}{81}\frac{1}{40 \widetilde{P_1}} )^2 (\because \eqref{eqn:less127})\\
&\leq 2 \frac{200}{157} (\frac{181}{81}\frac{1}{40} )^2  \Sigma_2^2 (\because \mbox{Assumption }\frac{1}{\widetilde{P_1}} \leq \Sigma_2) \label{eqn:less1add32}
\end{align}
and
\begin{align}
&2a^{1-k} \frac{1-a^{k-1}}{1-a} \frac{(1-a^{k-2})(1-a^{k-1})}{(1-a)^2} \widetilde{P_1}\\
&\leq 2 \sqrt{\frac{200}{157}} ( \frac{1-a^{k-1}}{1-a} )^3 \widetilde{P_1} (\because \eqref{eqn:less1add21}) \\
&\leq 2 \sqrt{\frac{200}{157}} ( \frac{2}{0.9^2} (\frac{a^2 - a^{2(k-1)}}{1-a^2}))^3 \widetilde{P_1} (\because \eqref{eqn:less1add25})\\
&\leq 2 \sqrt{\frac{200}{157}} ( \frac{2}{0.9^2} \frac{1}{40 \widetilde{P_1}} )^3 \widetilde{P_1} (\because \eqref{eqn:less127})\\
&= 2 \sqrt{\frac{200}{157}} ( \frac{5}{81})^3 \frac{1}{\widetilde{P_1}^2}\\
&\leq 2 \sqrt{\frac{200}{157}} ( \frac{5}{81})^3 \Sigma_2^2. (\because \mbox{Assumption } \frac{1}{\widetilde{P_1}} \leq \Sigma_2)
\label{eqn:less1add33}
\end{align}
Therefore, by \eqref{eqn:less1add32} and \eqref{eqn:less1add33}, $I'(\widetilde{P_1})$ is upper bounded as:
\begin{align}
I'(\widetilde{P_1}) & \leq \frac{1}{2} \log(1+\frac{\Sigma_2^2}{(k-1)\sigma_{v2}^2}
(
2 \frac{200}{157} (\frac{181}{81 \cdot 40})^2 +
2 \sqrt{\frac{200}{157}}(\frac{5}{81 })^3
))^{k-1}\\
&\leq
\frac{1}{2} \log(1+\frac{1}{(k-1)} (\frac{10}{9})^2
(
2 \frac{200}{157} (\frac{181}{81 \cdot 40})^2 +
2 \sqrt{\frac{200}{157}}(\frac{5}{81 })^3
))^{k-1} (\because \eqref{eqn:less126})\\
&\leq
\frac{1}{2} \log(1+\frac{0.010471667...}{k-1})^{k-1} \\
&\leq
\frac{1}{2} \log e^{0.01047}. \label{eqn:less1add35}
\end{align}
Moreover, we also have
\begin{align}
(\frac{1-a^{k-1}}{1-a})^2 \widetilde{P_1}  &\leq  ( \frac{2}{0.9^2} \frac{a^2 - a^{2(k-1)}}{1-a^2})^2 \widetilde{P_1} (\because \eqref{eqn:less1add25})\\
&\leq  ( \frac{2}{0.9^2} \frac{1}{40 \widetilde{P_1}})^2 \widetilde{P_1} (\because \eqref{eqn:less127})\\
&= (\frac{5}{81})^2 \frac{1}{\widetilde{P_1}}. \label{eqn:less1add36}
\end{align}
Finally, by \eqref{eqn:less127}, \eqref{eqn:less1add35}, \eqref{eqn:less1add36}, we can conclude
\begin{align}
D_L(\widetilde{P_1},\widetilde{P_2}) &\geq ( \sqrt{\frac{1}{40 e^{0.01047}\widetilde{P_1}}} - \sqrt{(\frac{5}{81})^2 \frac{1}{\widetilde{P_1}}})_+^2  +1   \\
&\geq  \frac{0.0091316992... }{\widetilde{P_1}} +1
\end{align}

Proof of (c):

Let $P:=\max(\widetilde{P_1},\widetilde{P_2})$. Notice that since $\frac{1-a^2}{20} \leq P \leq \frac{1}{40}$ and $0.9 \leq a < 1$ we have
\begin{align}
& \frac{1}{40P} \leq \frac{1}{2(1-a^2)} < \frac{a^2}{1-a^2}\\
& \frac{1}{40P} \geq 1 \geq a^2 = \frac{a^2 - a^4}{1-a^2}
\end{align}
Therefore, we can find $k \geq 3$ such that
\begin{align}
\frac{a^2-a^{2(k-1)}}{1-a^2} \leq \frac{1}{40P} < \frac{a^2-a^{2k}}{1-a^2} \label{eqn:less1add40}
\end{align}
Let's put such $k$ and $k_1=k_2=1$ as the parameters of Lemma~\ref{lem:less11}.
Then, the lower bound of Lemma~\ref{lem:less11} reduces to
\begin{align}
D_L(\widetilde{P_1},\widetilde{P_2}) \geq ( \sqrt{\frac{a^2-a^{2k}}{1-a^2}} - \sqrt{(\frac{1-a^{k-1}}{1-a})^2 \widetilde{P_1}} - \sqrt{(\frac{1-a^{k-1}}{1-a})^2 \widetilde{P_2}})_+^2 + 1. \label{eqn:less1add42}
\end{align}
Since $k \geq 3$, we still have \eqref{eqn:leq1converse5} which tells $\frac{1-a^{k-1}}{1-a} \leq \frac{2}{0.9^2}\frac{a^2 - a^{2(k-1)}}{1-a^2}$. Thus, by \eqref{eqn:leq1converse5}, we have
\begin{align}
\frac{1-a^{k-1}}{1-a}
&\leq \frac{2}{0.9^2}\frac{a^2 - a^{2(k-1)}}{1-a^2} \\
&\leq \frac{2}{0.9^2} \frac{1}{40P} (\because \eqref{eqn:less1add40})\\
&= \frac{5}{81P} \label{eqn:less1add41}
\end{align}
Thus, by plugging \eqref{eqn:less1add40}, \eqref{eqn:less1add41} into \eqref{eqn:less1add42}, we can conclude
\begin{align}
D_L(\widetilde{P_1},\widetilde{P_2}) &\geq ( \sqrt{\frac{1}{40P}} - \sqrt{(\frac{5}{81P})^2 P} - \sqrt{(\frac{5}{81P} )^2 P} )_+^2 + 1\\
&= (\sqrt{\frac{1}{40}} - \sqrt{(\frac{5}{81} )^2} - \sqrt{(\frac{5}{81} )^2})^2 \frac{1}{P} +1\\
&= 0.0012011... \frac{1}{P} +1
\end{align}

Proof (d):

Since $\frac{1}{2} < a^4$, there exists $k \geq 3$ such that
\begin{align}
a^{2k} \leq \frac{1}{2} < a^{2(k-1)} \label{eqn:less128}
\end{align}
Let's put such $k$ and $k_1=k_2=1$ as the parameters of Lemma~\ref{lem:less11}. Then, the lower bound of Lemma~\ref{lem:less11} reduces to
\begin{align}
D_L(\widetilde{P_1},\widetilde{P_2}) \geq ( \sqrt{\frac{a^2-a^{2k}}{1-a^2}} - \sqrt{(\frac{1-a^{k-1}}{1-a})^2 \widetilde{P_1}} - \sqrt{(\frac{1-a^{k-1}}{1-a})^2 \widetilde{P_2}})_+^2 + 1
\end{align}
Here, we have
\begin{align}
\frac{a^2-a^{2k}}{1-a^2} &\geq \frac{a^2 - \frac{1}{2}}{1-a^2} (\because \eqref{eqn:less128})\\
&\geq \frac{0.9^2 - \frac{1}{2}}{1-a^2} (\because 0.9 \leq a <1)\\
&=\frac{0.31}{1-a^2} \label{eqn:less1add51}
\end{align}
and
\begin{align}
\frac{1-a^{k-1}}{1-a} &\leq \frac{1-\frac{1}{\sqrt{2}}}{1-a} (\because \eqref{eqn:less128})\\
&\leq \frac{1-\frac{1}{\sqrt{2}}}{1-a} \frac{2}{1+a} (\because 0.9 \leq a <1)\\
&= \frac{2(1-\frac{1}{\sqrt{2}})}{1-a^2} \label{eqn:less1add52}
\end{align}
Finally, by the assumption $\max(\widetilde{P_1},\widetilde{P_2}) \leq \frac{1-a^2}{20}$ and \eqref{eqn:less1add51}, \eqref{eqn:less1add52} we can conclude
\begin{align}
D_L(\widetilde{P_1},\widetilde{P_2}) &\geq ( \sqrt{\frac{0.31}{1-a^2}} - \sqrt{ (\frac{2(1-\frac{1}{\sqrt{2}})}{1-a^2})^2 \frac{1-a^2}{20}} - \sqrt{ (\frac{2(1-\frac{1}{\sqrt{2}})}{1-a^2})^2 \frac{1-a^2}{20}})_+^2 +1 \\
&=(\sqrt{0.31} - \sqrt{(2(1-\frac{1}{\sqrt{2}}))^2 \frac{1}{20}} - \sqrt{(2(1-\frac{1}{\sqrt{2}}))^2 \frac{1}{20}} )^2 \frac{1}{1-a^2} +1 \\
&= \frac{0.0869099... }{1-a^2} + 1
\end{align}

Proof of (e):

As mentioned in the proof of Corollary~\ref{cor:2} (j), the centralized controller's distortion that has both observations $y_1[n], y_2[n]$ and has no input power constraints is a lower bound on the decentralized controller's distortion.

Let $y_1'[n]:=x[n]+v_1'[n]$ and $y_2'[n]:=x[n]+v_2'[n]$ where $v_1'[n] \sim \mathcal{N}(0,\sigma_1^2)$ and $v_2'[n] \sim \mathcal{N}(0,\sigma_1^2)$ are i.i.d. random variables. Just like the proof of Corollary~\ref{cor:2} (j), the performance of the centralized controller with both observations is equivalent to a centralized controller with observation $\frac{y_1'[n]+y_2'[n]}{2}$  by the maximum ratio combining.

Let $\Sigma_E$ be the estimation error of the Kalman filtering with a scalar observation $\frac{y_1'[n]+y_2'[n]}{2}$. By Lemma~\ref{lem:aless21},
\begin{align}
\Sigma_E &= \frac{(a^2-1)\frac{\sigma_{v1}^2}{2} -1 + \sqrt{((a^2-1)\frac{\sigma_{v1}^2}{2} -1)^2+4a^2 \frac{\sigma_{v1}^2}{2}}}{2a^2} \\
&=
\frac{4 a^2 \frac{\sigma_{v1}^2}{2}}{2a^2( -(a^2-1)\frac{\sigma_{v1}^2}{2}+1+
\sqrt{
((a^2-1)\frac{\sigma_{v1}^2}{2}-1)^2 + 4a^2 \frac{\sigma_{v1}^2}{2}
})} \\
&=
\frac{\sigma_{v1}^2}{ (1-a^2)\frac{\sigma_{v1}^2}{2}+1+
\sqrt{
((1-a^2)\frac{\sigma_{v1}^2}{2}+1)^2 + 4a^2 \frac{\sigma_{v1}^2}{2}
}}.
\end{align}
Here, we have
\begin{align}
\Sigma_E &\leq \frac{\sigma_{v1}^2}{(1-a^2)\frac{\sigma_{v1}^2}{2}+(1-a^2)\frac{\sigma_{v1}^2}{2}}= \frac{1}{1-a^2}. \label{eqn:less131}
\end{align}

Then, for all $\widetilde{P_1}$ and $\widetilde{P_2}$, the cost of the decentralized controllers is lower bounded as follows:
\begin{align}
D_L(\widetilde{P_1},\widetilde{P_2}) & \overset{(A)}{\geq} \inf_{k: |a-k| < 1} \frac{(2ak-k^2)\Sigma_E+1}{1-(a-k)^2} \\
&=\inf_{k: |a-k| < 1} \frac{a^2-1}{1-(a-k)^2}\Sigma_E + \frac{1-a^2+2ak-k^2}{1-(a-k)^2} \Sigma_E + \frac{1}{1-(a-k)^2} \\
&=\inf_{k: |a-k| < 1} \frac{a^2-1}{1-(a-k)^2}\Sigma_E + \Sigma_E + \frac{1}{1-(a-k)^2} \label{eqn:less130}
\end{align}
(A): The decentralized control cost is larger than the centralized controller's cost with the observation $\frac{y_1'[n]+y_2'[n]}{2}$. Moreover, when $|a-k| \geq 1$ the centralized control system is unstable, and the cost diverges to infinity. When $|a-k| < 1$, the cost analysis follows from Lemma~\ref{lem:aless21}.

Let $k^\star$ be $k$ achieving the infimum of \eqref{eqn:less130}. Then, for all $\widetilde{P_1}, \widetilde{P_2} \geq 0$ we have
\begin{align}
D_L(\widetilde{P_1},\widetilde{P_2}) &\geq \frac{a^2-1}{1-(a-k^\star)^2}\Sigma_E + \Sigma_E + \frac{1}{1-(a-k^\star)^2} \\
&\geq \frac{a^2-1}{1-(a-k^\star)^2} \frac{1}{1-a^2} + \Sigma_E + \frac{1}{1-(a-k^\star)^2} (\because \eqref{eqn:less131})\\
&= \Sigma_E. \label{eqn:less136}
\end{align}

Therefore, $\Sigma_E$ is a lower bound on $D_L(\widetilde{P_1},\widetilde{P_2})$, and we will compare this with $\Sigma_1$ which is
\begin{align}
\Sigma_1 &= \frac{2 \sigma_{v1}^2}{ (1-a^2)\sigma_{v1}^2+1+\sqrt{((1-a^2)\sigma_{v1}^2+1)^2 + 4a^2 \sigma_{v1}^2}}.
\end{align}

To this end, let's divide the case based on three quantities $(1-a^2)\frac{\sigma_{v1}^2}{2},1 ,a \frac{\sigma_{v1}}{\sqrt{2}}$.

(i) When $\max((1-a^2)\frac{\sigma_{v1}^2}{2},1 ,a \frac{\sigma_{v1}}{\sqrt{2}})=(1-a^2) \frac{\sigma_{v1}^2}{2}$,

In this case, by the definition of $\Sigma_E$ we have
\begin{align}
\Sigma_E &\geq \frac{\sigma_{v1}^2}{2(1-a^2) \frac{\sigma_{v1}^2}{2}+\sqrt{(2(1-a^2) \frac{\sigma_{v1}^2}{2})^2+4((1-a^2) \frac{\sigma_{v1}^2}{2})^2}} \\
&= \frac{1}{(1-a^2)+\sqrt{(1-a^2)^2 + (1-a^2)^2}} \\
&= \frac{1}{1+\sqrt{2}} \frac{1}{1-a^2}. \label{eqn:less133}
\end{align}
By the definition of $\Sigma_1$, we also have
\begin{align}
\Sigma_1 &\leq \frac{2\sigma_{v1}^2}{(1-a^2){\sigma_{v1}^2} + (1-a^2){\sigma_{v1}^2}} \\
&= \frac{1}{1-a^2}. \label{eqn:less132}
\end{align}

Therefore, we have
\begin{align}
D_L(\widetilde{P_1},\widetilde{P_2}) &\geq \Sigma_E (\because \eqref{eqn:less136})\\
&\geq \frac{1}{1+\sqrt{2}}\frac{1}{1-a^2} (\because \eqref{eqn:less133})\\
&\geq \frac{1}{1+\sqrt{2}}\Sigma_1. (\because \eqref{eqn:less132})
\end{align}
%
%

(ii) When $\max((1-a^2)\frac{\sigma_{v1}^2}{2},1 ,a \frac{\sigma_{v1}}{\sqrt{2}})=a \frac{\sigma_{v1}}{\sqrt{2}}$,

In this case, by the definition of $\Sigma_E$ we have
\begin{align}
\Sigma_E &\geq \frac{\sigma_{v1}^2}{a\frac{\sigma_{v1}}{\sqrt{2}}+a\frac{\sigma_{v1}}{\sqrt{2}}+ \sqrt{(a\frac{\sigma_{v1}}{\sqrt{2}}+a\frac{\sigma_{v1}}{\sqrt{2}})^2+ 4 a^2 \frac{\sigma_{v1}^2}{2}}} \\
&= \frac{\sigma_{v1}}{\frac{2a}{\sqrt{2}}+\sqrt{ (\frac{2a}{\sqrt{2}})^2+ 2a^2}} \\
&\geq \frac{\sigma_{v1}}{\frac{2}{\sqrt{2}}+\sqrt{2+2}} (\because 0.9 \leq a <1)\\
&= \frac{\sigma_{v1}}{\frac{2}{\sqrt{2}}+2}. \label{eqn:less134}
\end{align}
By the definition of $\Sigma_1$, we also have
\begin{align}
\Sigma_1 &\leq \frac{2\sigma_{v1}^2}{\sqrt{4 a^2 \sigma_{v1}^2}} = \frac{\sigma_{v1}^2}{a \sigma_{v1}} = \frac{\sigma_{v1}}{a} \\
&\leq \frac{10}{9} \sigma_{v1}. (\because 0.9 \leq a <1) \label{eqn:less135}
\end{align}

Therefore, we have
\begin{align}
D_L(P_1,P_2) &\geq \Sigma_E (\because \eqref{eqn:less136})\\
&\geq \frac{\sigma_{v1}}{\frac{2}{\sqrt{2}}+2} \eqref{eqn:less134} \\
&\geq \frac{1}{\frac{2}{\sqrt{2}}+2} \frac{9}{10} \Sigma_1. \eqref{eqn:less135}
\end{align}

(iii) When $\max((1-a^2)\frac{\sigma_{v1}^2}{2},1 ,a \frac{\sigma_{v1}}{\sqrt{2}})=1$,

By the assumption, we have $a\frac{\sigma_{v1}}{\sqrt{2}} \leq 1$. Since $0.9 \leq a < 1$, we can see
\begin{align}
\sigma_{v1} \leq \frac{\sqrt{2}}{a} \leq \frac{10 \sqrt{2}}{9}. \label{eqn:less1add70}
\end{align}

Furthermore, by the definition of $\Sigma_1$, we can see that \eqref{eqn:less135} still holds. Therefore, by \eqref{eqn:less1add70}, we have
\begin{align}
\Sigma_1 \leq \frac{10}{9}\sigma_{v1} \leq \frac{10}{9} (\frac{10 \sqrt{2}}{9}) \label{eqn:less1add71}
\end{align}
Moreover, by Lemma~\ref{lem:less11}, we know that for all $\widetilde{P_1}, \widetilde{P_2} \geq 0$, $D_L(\widetilde{P_1},\widetilde{P_2}) \geq 1$. Thus, by \eqref{eqn:less1add71} we can conclude
\begin{align}
D_L(\widetilde{P_1},\widetilde{P_2}) \geq 1 \geq \frac{9}{10}(\frac{9}{10 \sqrt{2}}) \Sigma_1
\end{align}

Finally, by (i),(ii),(iii),
\begin{align}
D &\geq \min(\frac{1}{1+\sqrt{2}}, \frac{1}{\frac{2}{\sqrt{2}}+2} \frac{9}{10}, \frac{9}{10}(\frac{9}{10 \sqrt{2}}) ) \Sigma_1 \\
& =\frac{1}{\frac{2}{\sqrt{2}}+2} \frac{9}{10} \Sigma_1\\
& \geq 0.26360... \Sigma_1.
\end{align}
Since by Lemma~\ref{lem:less11} we already know $D_L(\widetilde{P_1},\widetilde{P_2}) \geq 1$, the statement (e) of the corollary is true.
\end{proof}

\begin{proof}[Proof of Proposition~\ref{prop:3} of Page~\pageref{prop:3}]
Like the proof of Proposition~\ref{prop:1}, we define the subscript $max$ as $argmax_{i \in \{ 1,2\}} \widetilde{P_i}$, and write $D_{\sigma_{v1}}(\cdot), D_{\sigma_{v2}}(\cdot), D_{\sigma_{v max}}(\cdot)$ as $D_{v1}(\cdot), D_{v2}(\cdot), D_{v max}(\cdot)$ respectively.

By the same argument as the proof of Proposition~\ref{prop:1}, it is enough to show that there exists $c \leq 10^6$ such that for all $\widetilde{P_1}, \widetilde{P_2} \geq 0$, $\min(D_{\sigma 1}(c \widetilde{P_1}), D_{\sigma 2}(c \widetilde{P_2})) \leq c \cdot D_L(\widetilde{P_1}, \widetilde{P_2})$.

In the proof, we first divide the cases based on $\Sigma_1$, $\Sigma_2$, and then based on $\widetilde{P_1}, \widetilde{P_2}$. Here, we know $\Sigma_1 \leq \Sigma_2$ since $\sigma_1 \leq \sigma_2$. Using this, we can reduce the cases.


(i) When $\Sigma_1 \leq 40$, $\Sigma_2 \leq 40$

(i-i) When $\max(\widetilde{P_1}, \widetilde{P_2}) \geq \frac{1}{40}$

Lower bound: By Corollary~\ref{cor:6} (e), we have
\begin{align}
D_L(\widetilde{P_1}, \widetilde{P_2}) \geq 1
\end{align}

If $\max(\widetilde{P_1}, \widetilde{P_2}) \geq \frac{1}{40}$ and $\Sigma_{max} \geq \frac{1}{1-a^2}$

Upper bound: By Corollary~\ref{cor:5} \eqref{eqn:tradeoff2}, we have
\begin{align}
(D_{\sigma max}(P_{max}), P_{max}) \leq (\frac{1}{1-a^2},0) \leq (\Sigma_{max},0) \leq (40,0).
\end{align}

If $\max(\widetilde{P_1}, \widetilde{P_2}) \geq \frac{1}{40}$ and $\Sigma_{max} \leq \frac{1}{1-a^2}$

Upper bound: By plugging $t=\frac{1}{\max (1, \Sigma_{max})}$ into Corollary~\ref{cor:5} \eqref{eqn:tradeoff3}, we have
\begin{align}
(D_{\sigma max}(P_{max}),P_{max}) &\leq ( 2 \max(1,\Sigma_{max}), \frac{1}{\max(1,\Sigma_{max})}) \\
& \leq (2 \cdot 40 , 1) (\because \mbox{In (i), we assumed }\Sigma_1 \leq 40, \Sigma_2 \leq 40)
\end{align}

Ratio: $c$ is upper bounded by
\begin{align}
c \leq 2 \cdot 40.
\end{align}

(i-ii) When $\frac{1-a^2}{20} \leq \max(\widetilde{P_1}, \widetilde{P_2}) \leq \frac{1}{40}$

Lower bound: By Corollary~\ref{cor:6} (c), we have
\begin{align}
D_L(\widetilde{P_1}, \widetilde{P_2}) \geq \frac{0.001201}{\max(\widetilde{P_1}, \widetilde{P_2})}+1. \label{eqn:lowerbound31}
\end{align}

If $\Sigma_{max} \geq \frac{1}{1-a^2}$

Upper bound: By Corollary~\ref{cor:5} \eqref{eqn:tradeoff2}, we have
\begin{align}
(D_{\sigma max}(P_{max}), P_{max}) \leq (\frac{1}{1-a^2},0) \leq (\Sigma_{max},0) \leq (40,0)
\end{align}

If $\Sigma_{max} \leq \frac{1}{1-a^2}$ and $1-a^2 \leq \max(\widetilde{P_1}, \widetilde{P_2}) \leq \frac{1}{40}$

Since we assume $\Sigma_1 \leq 40$, $\Sigma_2 \leq 40$ in (i),  we have $1-a^2 \leq \widetilde{P_{max}} \leq \frac{1}{40} \leq \frac{1}{\max(1,\Sigma_{max})}$. Thus, we can plug $t=\widetilde{P_{max}}$ to Corollary~\ref{cor:5} \eqref{eqn:tradeoff3}, and conclude 
\begin{align}
(D_{\sigma max}(P_{max}),P_{max}) \leq (\frac{2}{\widetilde{P_{max}}}, \widetilde{P_{max}}).
\end{align}

If $\Sigma_{max} \leq \frac{1}{1-a^2}$ and $\frac{1-a^2}{20} \leq \max(\widetilde{P_1}, \widetilde{P_2}) \leq 1-a^2$

In this case, the lower bound of \eqref{eqn:lowerbound31} can be further lower bounded by
\begin{align}
D_L(\widetilde{P_1}, \widetilde{P_2}) \geq \frac{0.001201}{1-a^2} +1.
\end{align}

Upper bound: By Corollary~\ref{cor:5} \eqref{eqn:tradeoff2}, we have
\begin{align}
(D_{\sigma max}(P_{max}), P_{max}) \leq ( \frac{1}{1-a^2},0).
\end{align}

Ratio: $c$ is upper bounded by
\begin{align}
c \leq \frac{2}{0.001201} < 2000.
\end{align}

(i-iii) When $\max(\widetilde{P_1}, \widetilde{P_2}) \leq \frac{1-a^2}{20}$

Lower bound: By Corollary~\ref{cor:6} (d), we have
\begin{align}
D_L(\widetilde{P_1}, \widetilde{P_2}) \geq \frac{0.0869}{1-a^2}+1
\end{align}

Upper bound: By Corollary~\ref{cor:5} \eqref{eqn:tradeoff2}, we have
\begin{align}
(D_{\sigma max}(P_{max}),P_{max}) \leq ( \frac{1}{1-a^2}, 0).
\end{align}

Ratio: $c$ is upper bounded by
\begin{align}
c \leq \frac{1}{0.0869} < 12.
\end{align}

(ii) When $\Sigma_1 \leq 40 \leq \Sigma_2$

(ii-i) When $\widetilde{P_1} \geq \frac{1}{40}$

Lower bound: By Corollary~\ref{cor:6} (e), we have
\begin{align}
D_L(\widetilde{P_1}, \widetilde{P_2}) \geq 1.
\end{align}

If $\widetilde{P_1} \geq \frac{1}{40}$ and $\Sigma_1 \geq \frac{1}{1-a^2}$

Upper bound: By Corollary~\ref{cor:5} \eqref{eqn:tradeoff2}, we have
\begin{align}
(D_{\sigma 1}(P_1), P_1) \leq (\frac{1}{1-a^2} , 0) \leq (\Sigma_1 , 0) \leq (40,0).
\end{align}

If $\widetilde{P_1} \geq \frac{1}{40}$ and $\Sigma_1 \leq \frac{1}{1-a^2}$

Upper bound: By plugging $t=\frac{1}{\max(1, \Sigma_1)}$ into the equation \eqref{eqn:tradeoff3} of Corollary~\ref{cor:5}, we have
\begin{align}
(D_{\sigma 1}(P_1),P_1) &\leq (2 \max(1,\Sigma_1), \frac{1}{\max(1,\Sigma_1)}) \\
& \leq (2 \cdot 40 ,1) (\because \mbox{In (ii), we assumed }\Sigma_1 \leq 40)
\end{align}

Ratio: $c$ is upper bounded by
\begin{align}
c \leq 2 \cdot 40.
\end{align}

(ii-ii) When $\frac{1}{\Sigma_2} \leq \widetilde{P_1} \leq \frac{1}{40}$

Lower bound: By Corollary~\ref{cor:6} (b), we have
\begin{align}
D_L(\widetilde{P_1}, \widetilde{P_2}) \geq \frac{0.009131}{\widetilde{P_1}} + 1. \label{eqn:lowerbound32}
\end{align}

If $\Sigma_1 \geq \frac{1}{1-a^2}$

Upper bound: By Corollary~\ref{cor:5} \eqref{eqn:tradeoff2}, we have 
\begin{align}
(D_{\sigma 1}(P_1), P_1) \leq (\frac{1}{1-a^2} , 0) \leq (\Sigma_1 , 0) \leq (40,0)
\end{align}

If $\Sigma_1 \leq \frac{1}{1-a^2}$ and $1-a^2 \leq \widetilde{P_1} \leq \frac{1}{40}$

Upper bound: Since $1-a^2 \leq \widetilde{P_1} \leq \frac{1}{40} \leq \frac{1}{\max(1,\Sigma_1)}$, we can plug $t=\widetilde{P_1}$ into Corollary~\ref{cor:5} \eqref{eqn:tradeoff3}. Thus, we have
\begin{align}
(D_{\sigma 1}(P_1),P_1) \leq (\frac{2}{\widetilde{P_1}}, \widetilde{P_1}).
\end{align}

If $\Sigma_1 \leq \frac{1}{1-a^2}$ and $\frac{1}{\Sigma_2} \leq \widetilde{P_1} \leq 1-a^2$

In this case, the lower bound of \eqref{eqn:lowerbound32} can be further lower bounded by
\begin{align}
D_L(P_1,P_2) \geq \frac{0.009131}{1-a^2} + 1.
\end{align}

Upper bound: By Corollary~\ref{cor:5} \eqref{eqn:tradeoff2}, we have
\begin{align}
(D_{\sigma 1}(P_1), P_1) \leq (\frac{1}{1-a^2},0).
\end{align}

Ratio: $c$ is upper bounded by
\begin{align}
c \leq \frac{2}{0.009131} < 220.
\end{align}

(ii-iii) When $\widetilde{P_1} \leq \frac{1}{\Sigma_2}$ and $\max(\widetilde{P_1}, \widetilde{P_2}) = \widetilde{P_2} > \frac{1}{\Sigma_2}$

Lower bound: By Corollary~\ref{cor:6} (a), we have
\begin{align}
D_L(\widetilde{P_1}, \widetilde{P_2}) \geq 0.009131 \Sigma_2 + 1. \label{eqn:lowerbound33}
\end{align}

If $\Sigma_2 \geq \frac{1}{1-a^2}$, $\widetilde{P_1} \leq \frac{1}{\Sigma_2}$ and $\max(\widetilde{P_1}, \widetilde{P_2}) > \frac{1}{\Sigma_2}$

The lower bound of \eqref{eqn:lowerbound33} can be further lower bonded by
\begin{align}
D_L(\widetilde{P_1}, \widetilde{P_2}) \geq 0.009131 \Sigma_2 + 1 \geq 0.009131 \frac{1}{1-a^2} + 1.
\end{align}

Upper bound: By Corollary~\ref{cor:5} \eqref{eqn:tradeoff2}, we have 
\begin{align}
(D_{\sigma 2}(P_2), P_2) \leq (\frac{1}{1-a^2},0).
\end{align}

If $\Sigma_2 \leq \frac{1}{1-a^2}$, $\widetilde{P_1} \leq \frac{1}{\Sigma_2}$ and $\max(\widetilde{P_1}, \widetilde{P_2}) > \frac{1}{\Sigma_2}$ 

Upper bound: Since we assumed $\Sigma_2 \geq 40$ in (ii), $\max(1, \Sigma_2)=\Sigma_2$. Thus, we can plug $t=\frac{1}{\Sigma_2}$ into \eqref{eqn:tradeoff3} of Corollary~\ref{cor:5}, and conclude
\begin{align}
(D_{\sigma 2}(P_2), P_2) \leq (2 \Sigma_2, \frac{1}{\Sigma_2}).
\end{align}

Ratio: $c$ is upper bounded by
\begin{align}
c \leq \frac{2}{0.009131} < 220.
\end{align}

(ii-iv) When $\widetilde{P_1} \leq \frac{1}{\Sigma_2}$ and $\frac{1-a^2}{20} \leq \max(\widetilde{P_1}, \widetilde{P_2}) \leq \frac{1}{\Sigma_2}$

Lower bound: Since $\frac{1-a^2}{20} \leq \max(\widetilde{P_1}, \widetilde{P_2}) \leq \frac{1}{\Sigma_2} \leq \frac{1}{40}$, by Corollary~\ref{cor:6} (c) we have
\begin{align}
D_L(\widetilde{P_1}, \widetilde{P_2}) \geq \frac{0.001201}{\max(\widetilde{P_1}, \widetilde{P_2})} +1. \label{eqn:lowerbound34}
\end{align}

If $\Sigma_{max} \geq \frac{1}{1-a^2}$, $\widetilde{P_1} \leq \frac{1}{\Sigma_2}$ and $\frac{1-a^2}{20} \leq \max(\widetilde{P_1}, \widetilde{P_2}) \leq \frac{1}{\Sigma_2}$

In this case, the lower bound of \eqref{eqn:lowerbound34} can be further lower bounded by
\begin{align}
D_L(\widetilde{P_1}, \widetilde{P_2}) &\geq \frac{0.001201}{\max(\widetilde{P_1}, \widetilde{P_2})} +1 \geq 0.001201 \Sigma_{2} +1 \\
&\geq 0.001201 \Sigma_{max} +1 \geq \frac{0.001201}{1-a^2}+1.
\end{align}

Upper bound: By Corollary~\ref{cor:5} \eqref{eqn:tradeoff2}, we have
\begin{align}
(D_{\sigma max}(P_{max}),P_{max}) \leq (\frac{1}{1-a^2},0).
\end{align}

If $\Sigma_{max} \leq \frac{1}{1-a^2}$, $\widetilde{P_1} \leq \frac{1}{\Sigma_2}$ and $ \frac{1}{\max(1,\Sigma_{max})} <  \max(\widetilde{P_1}, \widetilde{P_2}) \leq \frac{1}{\Sigma_2}$

This case never happens since $\Sigma_2 \geq \max(1,\Sigma_{max})$.

If $\Sigma_{max} \leq \frac{1}{1-a^2}$, $\widetilde{P_1} \leq \frac{1}{\Sigma_2}$ and $1-a^2 \leq \max(\widetilde{P_1}, \widetilde{P_2}) \leq \frac{1}{\max(1,\Sigma_{max})}$

Upper bound: By plugging $t=\widetilde{P_{max}}$ into \eqref{eqn:tradeoff3} of Corollary~\ref{cor:5}, we have
\begin{align}
(D_{\sigma max}(P_{max}),P_{max}) \leq  (\frac{2}{\widetilde{P_{max}}}, \widetilde{P_{max}}).
\end{align}

If $\Sigma_{max} \leq \frac{1}{1-a^2}$, $\widetilde{P_1} \leq \frac{1}{\Sigma_2}$ and $\frac{1-a^2}{20} \leq \max(\widetilde{P_1}, \widetilde{P_2}) \leq 1-a^2$

In this case, the lower bound of \eqref{eqn:lowerbound34} can be further lower bounded by
\begin{align}
D_L(\widetilde{P_1}, \widetilde{P_2}) \geq \frac{0.0012011}{1-a^2} + 1.
\end{align}

Upper bound: By Corollary~\ref{cor:5} \eqref{eqn:tradeoff2}, we have
\begin{align}
(D_{\sigma max}(P_{max}), P_{max}) \leq ( \frac{1}{1-a^2} , 0).
\end{align}

Ratio: $c$ is upper bounded by
\begin{align}
c \leq \frac{2}{0.0012011} \leq 1700.
\end{align}

(ii-v) When $\widetilde{P_1} \leq \frac{1}{\Sigma_2}$ and $\max(\widetilde{P_1}, \widetilde{P_2}) \leq \frac{1-a^2}{20}$

Lower bound: By Corollary~\ref{cor:6} (d), we have
\begin{align}
D_L(\widetilde{P_1}, \widetilde{P_2}) \geq \frac{0.0869}{1-a^2} +1.
\end{align}

Upper bound: By Corollary~\ref{cor:5} \eqref{eqn:tradeoff2}, we have
\begin{align}
(D_{\sigma max}(P_{max}), P_{max}) \leq ( \frac{1}{1-a^2} , 0).
\end{align}

Ratio: $c$ is upper bounded by
\begin{align}
c \leq \frac{1}{0.0869} \leq 12.
\end{align}

(iii) When $40 \leq \Sigma_1 \leq \Sigma_2$

(iii-i) When $\widetilde{P_1} \geq \frac{1}{\Sigma_1}$

Lower bound: By Corollary~\ref{cor:6} (e), we have
\begin{align}
D_L(\widetilde{P_1}, \widetilde{P_2}) \geq 0.2636 \Sigma_1.
\end{align}

If $\widetilde{P_1} \geq \frac{1}{\Sigma_1}$ and $\Sigma_1 \geq \frac{1}{1-a^2}$

Upper bound: By Corollary~\ref{cor:5} \eqref{eqn:tradeoff2}, we have
\begin{align}
(D_{\sigma 1}(\widetilde{P_1}), \widetilde{P_1}) \leq (\frac{1}{1-a^2},0) \leq (\Sigma_1, 0).
\end{align}

If $\widetilde{P_1} \geq \frac{1}{\Sigma_1}$ and $\Sigma_1 \leq \frac{1}{1-a^2}$

Upper bound: By plugging $t=\frac{1}{\Sigma_1}$ into \eqref{eqn:tradeoff3} of Corollary~\ref{cor:5}, we have
\begin{align}
(D_{\sigma 1}(P_1),P_1) \leq (2 \Sigma_1, \frac{1}{\Sigma_1}).
\end{align}

Ratio: $c$ is upper bounded by
\begin{align}
c \leq \frac{2}{0.2636} < 8.
\end{align}

(iii-ii) When $\frac{1}{\Sigma_2} \leq \widetilde{P_1} \leq \frac{1}{\Sigma_1}$

Lower bound: Since $\frac{1}{\Sigma_2} \leq \widetilde{P_1} \leq \frac{1}{\Sigma_1} \leq \frac{1}{40}$, by Corollary~\ref{cor:6} (b)
\begin{align} we have
D_L(\widetilde{P_1}, \widetilde{P_2}) \geq \frac{0.009131}{\widetilde{P_1}} + 1. \label{eqn:lowerbound35}
\end{align}

If $\Sigma_1 \geq \frac{1}{1-a^2}$ and $\frac{1}{\Sigma_2} \leq \widetilde{P_1} \leq \frac{1}{\Sigma_1}$

In this case, the lower bound of \eqref{eqn:lowerbound35} can be further lower bounded by
\begin{align}
D_L(\widetilde{P_1}, \widetilde{P_2}) \geq 0.009131 \Sigma_1 +1.
\end{align}

Upper bound: By Corollary~\ref{cor:5} \eqref{eqn:tradeoff2}, we have
\begin{align}
(D_{\sigma 1}(P_1),P_1) \leq (\frac{1}{1-a^2},0) \leq (\Sigma_1,0).
\end{align}

If $\Sigma_1 \leq \frac{1}{1-a^2}$ and $1-a^2 \leq \widetilde{P_1} \leq \frac{1}{\Sigma_1}$

Upper bound: By plugging $t=\widetilde{P_1}$ into \eqref{eqn:tradeoff3} of Corollary~\ref{cor:5}, we have
\begin{align}
(D_{\sigma 1}(P_1), P_1) \leq (\frac{2}{\widetilde{P_1}}, \widetilde{P_1}).
\end{align}

If $\Sigma_1 \leq \frac{1}{1-a^2}$ and $\frac{1}{\Sigma_2} \leq \widetilde{P_1} \leq 1-a^2$

In this case, the lower bound \eqref{eqn:lowerbound35} can be further lower bounded by 
\begin{align}
D_L(\widetilde{P_1}, \widetilde{P_2}) \geq \frac{0.009131}{1-a^2} +1.
\end{align}

Upper bound: By Corollary~\ref{cor:5} \eqref{eqn:tradeoff2}, we have
\begin{align}
(D_{\sigma 1}(P_1), P_1) \leq (\frac{1}{1-a^2},0).
\end{align}

Ratio: $c$ is upper bounded by
\begin{align}
c \leq \frac{2}{0.009131} < 220.
\end{align}

(iii-iii) When $\widetilde{P_1} \leq \frac{1}{\Sigma_2}$ and $\max(\widetilde{P_1}, \widetilde{P_2}) > \frac{1}{\Sigma_2}$

Compared to the case (ii-iii), the only difference is the condition on $\Sigma_1$. Moreover, since the argument of (ii-iii) does not depend on the condition on $\Sigma_1$, we can get the same bound on $c$ following the same argument as (ii-iii).
%
%
%
%

(iii-iv) When $\widetilde{P_1} \leq \frac{1}{\Sigma_2}$ and $\frac{1-a^2}{20} \leq  \max(\widetilde{P_1}, \widetilde{P_2})\leq \frac{1}{\Sigma_2}$

Compared to the case (ii-iv), the only difference is the condition on $\Sigma_1$. Moreover, since the argument of (ii-iv) does not depend on the condition on $\Sigma_1$, we can get the same bound on $c$ following the same argument as (ii-iv).

(iii-v) When $\widetilde{P_1} \leq \frac{1}{\Sigma_2}$ and $\max(\widetilde{P_1}, \widetilde{P_2}) \leq \frac{1-a^2}{20}$

Compared to the case (ii-v), the only difference is the condition on $\Sigma_1$. Moreover, since the argument of (ii-v) does not depend on the condition on $\Sigma_1$, we can get the same bound on $c$ following the same argument as (ii-v).
\end{proof}

\end{document}